\documentclass[12pt]{amsart}

\usepackage{amsmath, amssymb, latexsym}
\usepackage{mathrsfs}
\usepackage{enumerate}
\usepackage[all, knot]{xy}
\xyoption{arc}
\usepackage[usenames,dvipsnames]{color}

\usepackage[colorlinks=true, pdfstartview=FitV, linkcolor=blue,%
citecolor=blue, urlcolor=blue]{hyperref}

\newtheorem{theorem}{Theorem}[section]
\newtheorem{proposition}[theorem]{Proposition}
\newtheorem{corollary}[theorem]{Corollary}

\newtheorem{lemma}[theorem]{Lemma}

\theoremstyle{definition}
\newtheorem{definition}[theorem]{Definition}
\newtheorem{example}[theorem]{Example}
\newtheorem{remark}[theorem]{Remark}

\numberwithin{equation}{section}

\def\g{\mathfrak g}
\def\oS{\overline{s}}

\def\k{\mathbf k}
\def\Z{\mathbb Z}
\def\P{\mathcal P}
\def\F{\mathcal F}
\def\cQ{\mathcal Q}
\def\Pev{\P^{{\rm ev}}}
\def\A{\mathbb A}
\def\Hom{{\rm Hom}}
\def\degZ{{\rm deg}_\Z}
\def\pa{{\rm p}}
\def\Mod{{\rm Mod }}

\def\wt{{\rm wt}}
\def\Uqg{U_q(\g)}
\def\Uqmg{U_q^-(\g)}
\def\Bqg{B_q(\g)}
\def\Iev{I_{\rm even}}
\def\Iod{I_{\rm odd}}
\def\ve{\varepsilon}

\def\Proj{{\rm Proj}}
\def\Rep{{\rm Rep}}

\newcommand{\ind}{{\rm Ind}}
\newcommand{\res}{{\rm Res}}
\newcommand{\soc}{{\rm soc}}
\newcommand{\hd}{{\rm hd}}

\def\ot{\otimes}
\def\bt{\boxtimes}

\newcommand{\nc}{\newcommand}
\nc{\be}{\begin{enumerate}} \nc{\ee}{\end{enumerate}}
\nc{\bnum}{\be[{\rm(i)}]} \nc{\bna}{\be[{\rm(a)}]}
\nc{\eq}{\begin{eqnarray}} \nc{\eneq}{\end{eqnarray}}
\nc{\eqn}{\begin{eqnarray*}} \nc{\eneqn}{\end{eqnarray*}}
\nc{\Oint}{\mathcal{O}_{\mathrm{int}}} \nc{\noi}{\noindent}
\nc{\hs}{\hspace*} \nc{\ba}{\begin{array}} \nc{\ea}{\end{array}}
\nc{\seteq}{\mathbin{:=}} \nc{\set}[2]{\left\{#1\mid #2\right\}}
\nc{\bl}{\bigl(} \nc{\br}{\bigr)} \nc{\cl}{\colon}

\nc{\To}[1][]{\xrightarrow{\rule{.6ex}{0ex}#1\rule{.6ex}{0ex}}}

\newcommand{\shc}{{\mathscr{C}}}
\nc{\car}{A}
\nc{\vphi}{\varphi}
\nc{\Q}{\mathbb{Q}}
\nc{\eps}{\varepsilon}
\nc{\tp}{\mathrm{top}}
\newcommand{\isoto}[1][]{\mathop{\xrightarrow[#1]%
{\rule{0pt}{.9ex}%
{\raisebox{-.6ex}[0ex][-.7ex]{$\mspace{4mu}\sim\mspace{3mu}$}}}}}
\newcommand{\isofrom}[1][]{\mathop{\xleftarrow[#1]%
{\rule{0pt}{.9ex}%
{{\raisebox{-.6ex}[0ex][-.6ex]{$\mspace{3mu}\sim\mspace{4mu}$}}}}}}

\nc{\soplus}{\mathop{\text{\scriptsize\raisebox{.5ex}{$\displaystyle\bigoplus$}\ }}}

\nc{\id}{\mathrm{id}}
\nc{\tens}{\mathop\otimes\limits}
\nc{\tf}{\tilde{f}}
\nc{\ol}{\overline}
\nc{\ro}{{\rm(}}
\nc{\rf}{{\rm)}}
\nc{\al}{\alpha}
\nc{\te}{\tilde{e}}
\nc{\la}{\lambda}
\nc{\lan}{\langle}
\nc{\ran}{\rangle}
\nc{\on}{\operatorname}
\nc{\Ir}{\on{Irr}}
\nc{\vs}{\vspace*}

\begin{document}
\title[Supercategorification of quantum Kac-Moody algebras]
{Supercategorification of \\ quantum Kac-Moody algebras}
\author[Seok-Jin Kang]{Seok-Jin Kang$^{1}$}
\thanks{$^1$ This work was supported by NRF Grant \# 2012-005700 and NRF Grant \# 2011-0027952.}
\address{Department of Mathematical Sciences and Research Institute of Mathematics,
Seoul National University, 599 Gwanak-ro, Gwanak-gu, Seoul 151-747,
Korea} \email{sjkang@snu.ac.kr}

\author[Masaki Kashiwara]{Masaki Kashiwara$^{2}$}
\thanks{$^2$ This work was supported by Grant-in-Aid for
Scientific Research (B) 22340005, Japan Society for the Promotion of
Science.}
\address{Research Institute for Mathematical Sciences, Kyoto University, Kyoto 606-8502,
Japan, and Department of Mathematical Sciences, Seoul National
University, 599 Gwanak-ro, Gwanak-gu, Seoul 151-747, Korea}
\email{masaki@kurims.kyoto-u.ac.jp}

\author[Se-jin Oh]{Se-jin Oh$^{3}$}
\thanks{$^{3}$ This work was supported by BK21 Mathematical Sciences
Division and NRF Grant \# 2012-005700.}
\address{Department of Mathematical Sciences, Seoul National University,
599 Gwanak-ro, Gwanak-gu, Seoul 151-747, Korea}
\email{sj092@snu.ac.kr}

%\date{\today}
%\urladdr{}
%\dedicatory{}

\begin{abstract}

We show that the quiver Hecke superalgebras and their cyclotomic
quotients provide a supercategorification of quantum Kac-Moody
algebras and their integrable highest weight modules.

\end{abstract}

 \maketitle
\tableofcontents
\vskip 2em

\section{Introduction}

The idea of categorification dates back to I. B. Frenkel. He
proposed that one can construct a tensor category whose Grothendieck
 group is isomorphic to  
the quantum group $U_q({\mathfrak{sl}}_{2})$
(see, for example, \cite{CF94, FK97}).
 Since then, a construction of a
tensor category or a 2-category whose Grothendieck group possesses a
given algebraic structure has been referred to as {\em
categorification}.

One of the most prominent examples of categorification is the
Lascoux-Leclerc-Thibon-Ariki theory, which clarifies the mysterious
connection between the representation theory of quantum affine
algebras of type $A_{n-1}^{(1)}$ and the modular representation
theory of Hecke algebras at roots of unity. In \cite{MM}, Misra and
Miwa constructed an integrable representation of the quantum affine
algebra $U_q(\widehat{{\mathfrak{sl}}_{n}})$, called the Fock space
representation, on the space spanned by colored Young diagrams. They
also showed that the affine crystal consisting of $n$-reduced
colored Young diagrams is isomorphic to the highest weight crystal
$B(\Lambda_0)$. In \cite{LLT96}, using the Misra-Miwa construction,
Lascoux, Leclerc and Thibon discovered an algorithm of computing
Kashiwara's lower global basis (=Lusztig's canonical basis) elements
corresponding to $n$-reduced Young diagrams and conjectured that the
transition matrices evaluated at 1 coincide with the composition
multiplicities of simple modules inside Specht modules over Hecke
algebras.

The Lascoux-Leclerc-Thibon conjecture was proved by Ariki in a more
general form \cite{Ar96}. Combining the geometric method of
Kazhdan-Lusztig and Ginzburg with combinatorics of Young diagrams
and Young tableaux, Ariki proved that the Grothendieck group of the
category of finitely generated projective modules over cyclotomic
Hecke algebras give a categorification of integrable highest weight
modules over affine Kac-Moody algebras of type $A_{n-1}^{(1)}$.
Moreover, he showed that the Kashiwara-Lusztig global basis is
mapped onto the isomorphism classes of projective indecomposable
modules, from which the Lascoux-Leclerc-Thibon conjecture follows.
Actually, he proved the conjecture in a more general form because
his categorification theorem holds for highest weight modules of
arbitrary level.

Next problem is to prove a quantum version or a graded version of
Ariki's categorification theorem. The key to this problem was
discovered by Khovanov-Lauda and Rouquier. In \cite{KL1, KL2, R08},
Khovanov-Lauda and Rouquier independently  introduced a new
family of graded algebras, called the {\it Khovanov-Lauda-Rouquier
algebras} or {\it quiver Hecke algebras}, and proved that the quiver
Hecke algebras provide a categorification of the negative half of
quantum groups associated with all symmetrizable Cartan data.
Furthermore, Khovanov and Lauda conjectured that the {\it cyclotomic
quiver Hecke algebras} give a graded version of Ariki's
categorification theorem in much more generality. That is, the
cyclotomic quiver Hecke algebras should give a
categorification of integrable highest weight modules over all
symmetrizable quantum Kac-Moody algebras. This conjecture was proved
by Kang and Kashiwara \cite{KK11}. Their proof was based on: (i) a
detailed analysis of the structure of quiver Hecke algebras and
their cyclotomic quotients, (ii) the proof of exactness of
restriction and induction functors, (iii) the existence of natural
isomorphisms, which is one of the axioms for the
${\mathfrak{sl}}_{2}$ categorification developed by Chuang and
Rouquier \cite{CR08}.

On the other hand, in \cite{BK01}, Brundan and Kleshchev showed
that, when the defining parameter is a primitive $(2l+1)$-th root of
unity, the representation theory of some blocks of affine
Hecke-Clifford superalgebras and their cyclotomic quotients is
controlled by the representation theory of quantum affine algebras
of type $A_{2l}^{(2)}$ at the crystal level. In \cite{Tsu},
Tsuchioka proved a similar statement for the affine and cyclotomic
Hecke-Clifford superalgebras with the defining parameter at
$2(l+1)$-th root of unity and the quantum affine algebras of type
$D_{l+1}^{(2)}$. On the other hand, in \cite{Wang07}, Wang introduced {\it
spin affine Hecke algebras} (resp.\ {\em cyclotomic spin Hecke
algebras}) and showed that they are {\it Morita superequivalent} to
affine Hecke-Clifford superalgebras (resp.\ cyclotomic
Hecke-Clifford superalgebras).

In \cite{KKT11}, motivated by the works of Brundan-Kleshchev
\cite{BK01}, Kang, Kashiwara and Tsuchioka introduced a new family
of graded superalgebras, called the {\it quiver Hecke
superalgebras}, which is
a super version of Khovanov-Lauda-Rouquier algebras. 
They also defined the notion of {\it quiver Hecke-Clifford superalgebras}
and showed that these superalgebras are weakly Morita superequivalent to the
corresponding quiver Hecke superalgebras. Moreover, they showed
that, after some completion, quiver Hecke-Clifford superalgebras are
isomorphic to affine Hecke-Clifford superalgebras.

Now the natural question arises:
{\it What do these superalgebras categorify?}
\noindent
The purpose of this paper is to provide answers to this
question. We will show that the quiver Hecke superalgebras and their
cyclotomic quotients give a supercategorification of quantum
Kac-Moody algebras and their integrable highest weight modules.
(In \cite{EKL}, Ellis, Khovanov and Lauda dealt with the case when
$I=\{ i \}$.)

Recall that a {\it supercategory} is a category with an endofunctor
$\Pi$ and a natural isomorphism $\xi: \Pi^2 \overset{\sim}
\longrightarrow \text{id}$ such that $\xi \circ \Pi = \Pi \circ
\xi$. Thus a  {\it supercategorification} means a construction of a
supercategory whose Grothendieck group possesses a given
algebraic structure.

To explain our main results more precisely, we fix some notations
and conventions. Let $A= A_{0} \oplus A_{1}$ be a superalgebra and
let $\phi_{A}$ be the involution defined by $\phi_{A}(a) =
(-1)^{\epsilon}a$ for $a \in A_{\epsilon}$ with $\epsilon = 0, 1$.
We denote by $\Mod(A)$ the category of left $A$-modules, which
becomes a supercategory with the functor $\Pi$ induced by
$\phi_{A}$. On the other hand, we denote by $\Mod_{\rm super}(A)$
the category of left $A$-supermodules $M=M_{0} \oplus M_{1}$ with
$\Z_{2}$-degree preserving homomorphisms as morphisms. The category
$\Mod_{\rm super}(A)$ is endowed with a structure of supercategory
given by the parity shift functor $\Pi$.

For each $n \ge 0$, let $R(n)$ be the quiver Heck superalgebra over
a base  field $\k$
generated by $e(\nu) (\nu \in I^n)$, $x_{k}$ ($1 \le k \le n$),
$\tau_{l}$ ($1 \le l< n$) with the defining relations given in
Definition~\ref{def:Quiver Hekce superalg}. Set $R(\beta)= e(\beta) R(n)$, where $\beta \in
\mathtt{Q}^{+}$, $e(\beta)= \sum_{\nu \in I^{\beta}} e(\nu)$. For a
dominant integral weight $\Lambda$, let $R^{\Lambda}(\beta)$ denote
the cyclotomic quiver Hecke superalgebra at $\beta$ (see Definition
\ref{Def: cyclo}). Let $\Mod(R(\beta))$ (resp.\ $\Proj(R(\beta))$)
be the category of $\Z$-graded (resp.\ finitely generated
projective) left $R(\beta)$-modules. We also denote by
$\Rep(R(\beta))$ the category of $\Z$-graded $R(\beta)$-modules that
are
finite-dimensional over $\k$.
We define the categories $\Mod(R^{\Lambda}(\beta))$,
$\Proj(R^{\Lambda}(\beta))$ and $\Rep(R^{\Lambda}(\beta))$ in a
similar manner.

On the other hand, let $U_q(\g)$ be the quantum Kac-Moody algebra
associated with a symmetrizable Cartan datum and let
$U_{\A}^{-}(\g)$ be the $\A$-form of the negative half of
$U_{q}(\g)$ with $\A = \Z[q,q^{-1}]$. For a dominant integral weight
$\Lambda$, we denote by $V(\Lambda)$ the integrable highest weight
module  generated by the  highest weight vector $v_\Lambda$ with weight $\Lambda$.
We also denote by
$V_{\A}(\Lambda)$ the $\A$-form
$U_{\A}^{-}(\g)v_\Lambda$ of $V(\Lambda)$
and by $V_{\A}(\Lambda)^{\vee}$
its dual $\A$-form.

Our goal in this paper is to prove the following isomorphisms
\begin{equation} \label{eq:goal-1}
V_{\A}(\Lambda) \overset{\sim} \longrightarrow [\Proj(R^{\Lambda})],
\quad V_{\A}(\Lambda)^{\vee}  \overset{\sim} \longrightarrow
[\Rep(R^{\Lambda})],
\end{equation}
\begin{equation}\label{eq:goal-2}
U_{\A}^{-}(\g)  \overset{\sim} \longrightarrow [\Proj(R)], \quad
U_{\A}^{-}(\g)^{\vee}  \overset{\sim} \longrightarrow [\Proj(R)],
\quad
\end{equation}
where $$\Proj(R^{\Lambda}) = \soplus_{\beta \in \mathtt{Q}^{+}}
\Proj(R^{\Lambda}(\beta)), \ \  \Rep(R^{\Lambda}) = \soplus_{\beta
\in \mathtt{Q}^{+}} \Rep(R^{\Lambda}(\beta)),$$
$$\Proj(R) = \soplus_{\beta \in \mathtt{Q}^{+}} \Proj(R(\beta)),
\ \ \Rep(R) = \soplus_{\beta \in \mathtt{Q}^{+}} \Proj(R(\beta)),$$
and $[\;\bullet\;]$ denotes the Grothendieck group.

For this purpose, we define  the $i$-restriction and  the $i$-induction
superfunctors $( i\in I)$
\begin{equation*}
\begin{aligned}
E_{i}^{\Lambda} & \cl \Mod(R^{\Lambda}(\beta + \alpha_i))
\longrightarrow \Mod(R^{\Lambda}(\beta)),\\
F_{i}^{\Lambda} & \cl \Mod(R^{\Lambda}(\beta)) \longrightarrow
\Mod(R^{\Lambda}(\beta+ \alpha_i))
\end{aligned}
\end{equation*}
by
\begin{equation*}
\begin{aligned}
E_{i}^{\Lambda}(N) & = e(\beta, i) N = e(\beta, i)
R^{\Lambda}(\beta+ \alpha_i) \otimes_{R^{\Lambda}(\beta+\alpha_i)}
N, \\
F_{i}^{\Lambda}(M) & = R^{\Lambda}(\beta+ \alpha_i) e(\beta, i)
\otimes_{R^{\Lambda}(\beta)} M,
\end{aligned}
\end{equation*}
where $M \in \Mod(R^{\Lambda}(\beta))$ and $N \in
\Mod(R^{\Lambda}(\beta + \alpha_i))$.

The first main result of this paper shows that the functors
$E_{i}^{\Lambda}$ and $F_{i}^{\Lambda}$ are exact and hence they
induce well-defined operators on $[\Proj(R^{\Lambda})]$ and
$[\Rep(R^{\Lambda})]$. To prove this, we take a detailed
analysis of the structure of $R(\beta)$ and $R^{\Lambda}(\beta)$,
and show that $R^{\Lambda}(\beta+\alpha_i) e(\beta, i)$ is a
projective right $R^{\Lambda}(\beta)$-supermodule (Theorem
\ref{Thm: P-injective}).

Next, we prove a super version of
${\mathfrak{sl}}_{2}$-categorification. That is, for $\lambda =
\Lambda - \beta$, we show that there exist natural isomorphisms
given below (Theorem \ref{Thm: Main}):

(i) if $\langle h_i, \lambda \rangle \ge 0$, then we have
\begin{equation*}
\Pi_iq_i^{-2}F^{\Lambda}_iE^{\Lambda}_i \oplus \soplus^{\langle
h_i,\lambda \rangle-1}_{k=0}\Pi_i^k q_i^{2k} \overset{\sim}{\to}
E^{\Lambda}_iF^{\Lambda}_i,
\end{equation*}

(ii) if $\langle h_i,\lambda \rangle < 0$, then we have
\begin{equation*}
\Pi_iq_i^{-2}F^{\Lambda}_iE^{\Lambda}_i \overset{\sim}{\to}
E^{\Lambda}_iF^{\Lambda}_i \oplus \soplus^{-\langle h_i,\lambda
\rangle-1}_{k=0}\Pi_i^{k+1}q_i^{-2k-2},
\end{equation*}
where $\Pi_i =\id$ if $i$ is even and $\Pi_i = \Pi$ if $i$ is odd.

We then show that the endomorphisms induced on
$[\Proj(R^\Lambda)]$ and $[\Rep(R^\Lambda)]$ by the parity shift functor
$\Pi$ coincide with the identity.
Therefore the above isomorphisms (i) and (ii) of functors
show that $[\Proj(R^\Lambda)]$ and $[\Rep(R^\Lambda)]$ have a structure of $U_\A(\g)$-module.

Next, we show that the simple objects of $\Rep(R^\Lambda)$ form a
basis of $[\Rep(R^\Lambda)]$, which shares a particular property of
 {\em strong perfect bases} introduced by
Berenstein-Kazhdan (\cite{BerKaz07}).
By the general theory of strong
perfect bases, we conclude that $[\Rep(R^\Lambda)]$ is isomorphic to
$V_\A(\Lambda)^\vee$. Then by duality, $[\Proj(R^\Lambda)]$ is
isomorphic to
$V_\A(\Lambda)$. 
Finally, by taking the projective limit with respect to $\Lambda$,
we conclude that there exist natural isomorphisms in
\eqref{eq:goal-2} (Theorem \ref{th:main1}, Corollary
\ref{cor:main2}).

\smallskip

 In principle, this paper follows the outline of \cite{KK11}.
However, our argument is substantially different from the one in
\cite{KK11} in the following sense:
\bnum
\item Our supercategorification theorem would give
various
applications. For instance, we can generalize
the Brundan-Kleshchev
 categorification theorem (\cite{BK01}) in the level of crystal to the quantum level.

\item We make use of the strong perfect basis theory to
characterize the $\A$-forms of $V(\Lambda)$ and obtain a
supercategorification of $V(\Lambda)$.

\item By taking a projective limit, we obtain a categorification
of $U_q^-(\g)$, which is opposite to the usual approach.

\item Since we deal with skew polynomials rather than
polynomials, the calculation involved are much more subtle and
complicated.

\ee

\smallskip
Note that any simple object $M$ of $\Rep(R^\Lambda)$ is
self-associate, i.e., $\Pi M\simeq M$. Hence the parity functor
$\Pi$ induces the identity on the Grothendieck group
$[\Rep(R^\Lambda)]$. On the other hand, any simple object of the
category $\Rep_{\mathrm{super}}(R^\Lambda)$ of finite-dimensional
$R^\Lambda$-supermodules is never self-associate.  Hence, $\Pi$
induces a non-trivial action on the Grothendieck group
$[\Rep_{\mathrm{super}}(R^\Lambda)]$, and
$$[\Rep(R^\Lambda)]\simeq\dfrac{[\Rep_{\mathrm{super}}(R^\Lambda)]}
{(\Pi-1)[\Rep_{\mathrm{super}}(R^\Lambda)]}.$$ The study of
$\Rep_{\mathrm{super}}(R^\Lambda)$ will be carried out in a
forthcoming paper. Note that in \cite{HW}, Hill and Wang studied the
categories of supermodules over quiver Hecke superalgebras under an
additional condition on the Cartan datum.

\medskip
This paper is organized as follows. In the following two sections,
we recall some of the basic properties of quantum Kac-Moody
algebras, integrable highest weight modules and supercategories. In
Section 4, we take a detailed analysis of the structure of
quiver Hecke superalgebras. In Section 5, we prove the existence of
natural isomorphisms and short exact sequences which are necessary
in proving our main results. In Section 6, we show that the
Grothendieck group $[\Rep(R)]$ has a strong perfect basis. In
Section 7 and 8, we prove that the  superfunctors
$E_{i}^{\Lambda}$ and $F_{i}^{\Lambda}$ are exact and they send
projectives to projectives. In Section 9, we prove that
$E_{i}^{\Lambda}$ and $F_{i}^{\Lambda}$ satisfy certain commutation
relations, which is a super version of
${\mathfrak{sl}}_{2}$-categorification. In the final section, we
conclude that the quiver Hecke superalgebras and their cyclotomic
quotients provide supercategorification of quantum Kac-Moody
algebras and their integrable highest weight modules.

\section{ Quantum Kac-Moody algebras }

Let $I$ be an index set.  An integral square matrix
$\car=(a_{ij})_{i,j \in I}$ is called a {\it Cartan matrix}  if it
satisfies
$$
\text{(i) $a_{ii}=2$, \quad (ii) $a_{ij} \le 0$ for $i \neq j$, \quad
 (iii) $a_{ij} =0 $ if $a_{ji}=0$.}$$
In this paper, we assume that $\car$ is {\it symmetrizable};  i.e.,
there is a diagonal matrix $\mathtt{D}={\rm diag}(s_i \in \Z_{>0} \
|
 \ i \in I)$ such that $\mathtt{D}\car$ is symmetric.

A \emph{Cartan datum} $(\car,\mathtt{P},\Pi,\Pi^{\vee})$ consists of
\begin{enumerate}
\item[(1)] a Cartan matrix $\car$,
\item[(2)] a free abelian group $\mathtt{P}$, called the \emph{weight lattice},
\item[(3)] $\Pi= \{ \alpha_i \in \mathtt{P} \mid \ i \in I \}$, called  the set of \emph{simple roots},
\item[(4)] $\Pi^{\vee}= \{ h_i \ | \ i \in I  \} \subset \mathtt{P}^{\vee}\seteq
\Hom(\mathtt{P},\Z)$, called
the set of \emph{simple coroots},
\end{enumerate}
satisfying the following conditions:
\begin{enumerate}
\item[(a)] $\langle h_i,\alpha_j \rangle = a_{ij}$ for all $i,j \in I$,
\item[(b)] $\Pi$ is linearly independent.
\end{enumerate}

The weight lattice $\mathtt{P}$ has a symmetric bilinear pairing $(
\ | \ )$ satisfying
$$ (\alpha_i | \lambda) = s_i \langle h_i, \lambda \rangle \text{ for all } \lambda \in \mathtt{P}.$$
Therefore, we have $(\alpha_i | \alpha_j) = s_i a_{ij}$. We set
$\mathtt{P}^+\seteq\{ \Lambda \in \mathtt{P} \ | \ \langle h_i,
\Lambda \rangle \in \Z_{\ge 0}, \text{ for all } i \in I \} $, which
is called  the set of {\it dominant integral weights}. The free
abelian group $\mathtt{Q}\seteq \bigoplus_{i \in I} \Z \alpha_i$ is
called the \emph{root lattice}. Set $\mathtt{Q}^+= \sum_{i \in I}
\Z_{\ge 0} \alpha_i$. For $\beta = \sum k_i \alpha_i \in
\mathtt{Q}^+$,  $|\beta| \seteq \sum_{i\in I} k_i$ is called the {\em
height} of $\beta$.

For an indeterminate $q$, set $q_i=q^{s_i}$
and define the $q$-integers
\begin{equation*}
 \begin{aligned}
 \ &[n]_i =\frac{ q_i^n - q^{-n}_{i} }{ q_i - q^{-1}_{i} },
 \ &[n]_i! = \prod^{n}_{k=1} [k]_i ,
\ &\left[\begin{matrix}m \\ n\\ \end{matrix} \right]_i= \frac{
[m]_i! }{[m-n]_i! [n]_i! }.
 \end{aligned}
\end{equation*}

\begin{definition} \label{Def: KM}
The {\em quantum Kac-Moody algebra} $\Uqg$ associated with a Cartan
datum $(\car,\mathtt{P},\Pi,\Pi^{\vee})$ is the associative algebra
over $\mathbb{Q}(q)$ with ${\bf 1}$ generated by 
$e_i,f_i$ $(i \in I)$ and $q^{h}$ $(h \in \mathtt{P}^{\vee})$
subject to the following defining relations:
\begin{enumerate}
  \item[(Q1)] $q^0=1$,\; $ q^{h}  q^{h'}= q^{h+h'} $ for $ h,h' \in
  \mathtt{P}^{\vee}$,
  \item[(Q2)] $ q^{h} e_i  q^{-h}= q^{\langle h, \alpha_i \rangle}  e_i$,\;
        $q^{h} f_i  q^{-h} = q^{-\langle h, \alpha_i \rangle } f_i$
        for $h \in \mathtt{P}^{\vee}$, $i \in I$,
 \item[(Q3)] $ e_if_j-f_je_i=\delta_{i,j}\dfrac{K_i-K^{-1}_i}{q_i-q^{-1}_i }$ where $K_i= q_i^{
 h_i}$,
 \item[(Q4)] $\displaystyle \sum^{1-a_{ij}}_{r=0} (-1)^r  e^{(1-a_{ij}-r)}_i  e_j  e^{(r)}_i
                     =\sum^{1-a_{ij}}_{r=0} (-1)^r  f^{(1-a_{ij}-r)}_i  f_j  f^{(r)}_i
                     =0$ if $i \ne j $,\\
where $ e^{(k)}_i\seteq\dfrac{e^{k}_i}{[k]_i!}$ and
$f^{(k)}_i\seteq\dfrac{ f^{k}_i}{[k]_i!}$ for $k \in \Z_{\ge 0}$.
\ee
\end{definition}

Let $\Uqmg$ (resp.\ $U_q^+(\g)$) be the $\mathbb{Q}(q)$-subalgebra
of $\Uqg$ generated by $f_i$'s (resp.\ $e_i$'s) and let $\Uqg^0$ be
the $\mathbb{Q}(q)$-subalgebra of $\Uqg$ generated by $q^h$ $(h \in
{\mathtt P}^{\vee})$. Then we have the \emph{ triangular
decomposition}
$$\Uqg \simeq \Uqmg\ot \Uqg^0 \ot U_q^+(\g),$$
and the \emph{root space decomposition}
$$\Uqg= \bigoplus_{\alpha \in \mathtt{Q}} \Uqg_\alpha\quad \text{where
 $\Uqg_\alpha\seteq \set{ x \in \Uqg }{\text{$\ q^h x q^{-h}= q^{\langle h,\alpha \rangle}x$ for any
 $h \in  \mathtt{P}^\vee$}}$.}$$

 Let $\A=\Z[q,q^{-1}]$ and denote by $U_\A(\g)^-$ (resp.\ $U_\A(\g)^+$) the $\A$-subalgebra of
 $\Uqmg$ generated by $f_i^{(n)}$ (resp.\ $e_i^{(n)}$) and denote by
$U^0_\A(\g)$ the $\A$-subalgebra generated by $q^h$ and
$\prod^m_{k=1} \dfrac{1-q^kq^h}{1-q^k}$ for all $m \in \Z_{> 0}$ and
$h \in \mathtt{P}^\vee$. Let $U_\A(\g)$ be the $\A$-subalgebra
generated by $U^0_\A(\g)$, $U^+_\A(\g)$ and $U^-_\A(\g)$. Then
$U_\A(\g)$ also has the triangular decomposition
$$  U_\A(\g) \simeq U^-_\A(\g) \otimes U^0_\A(\g) \otimes U^+_\A(\g).$$

Fix $i \in I$. For any $P\in \Uqmg$, there exist unique $Q,R\in \Uqmg$ such that
$$ e_iP-Pe_i= \dfrac{K_iQ- K^{-1}_iR}{q_i - q^{-1}_i}.$$
We define the endomorphisms $e_i'$, $e_i'' : \Uqmg\to \Uqmg$ by
$$e_i'(P)=R, \quad e_i''(P)=Q.$$
Regarding $f_i$ as the endomorphism of $\Uqmg$ defined by the left
multiplication by $f_i$, we obtain the $q$-boson commutation
relations:
\begin{equation} \label{eq: e_i' ev}
e_i'f_j = q_i^{-a_{ij}}f_je_i' + \delta_{i,j}.
\end{equation}

\begin{definition}[\cite{Kash91}]
The {\it quantum  boson algebra} $\Bqg$ associated with a Cartan
matrix $\car$ is the associative algebra over $\mathbb{Q}(q)$
generated by $e_i', f_i \ (i \in I)$ satisfying the following defining
relations :
\begin{enumerate}[{\rm(1)}]
\item $e_i'f_j = q_i^{-a_{ij}}f_je_i' + \delta_{i,j}$,

\item $\sum\limits_{r=0}^{1-a_{ij}} (-1)^r\left[\begin{matrix}1-a_{ij} \\ r\end{matrix} \right]_i
               e_i'^{1-a_{ij}-r}e_j'e_i'^{r}=
               \sum\limits_{r=0}^{1-a_{ij}} (-1)^r\left[\begin{matrix}1-a_{ij} \\ r \end{matrix} \right]_i
                f_i^{1-a_{ij}-r}f_jf_i^{r} =0$
                if $i\neq j $.
\end{enumerate}
\end{definition}

\begin{lemma}[{\cite[Lemma 3.4.7, Corollary 3.4.9]{Kash91}}] \label{Lem: Boson structure} \
\bna
\item If $x \in \Uqmg$ and $e_i'x=0$ for all $i \in I$, then $x$ is a constant multiple of\/ ${\bf 1}$.
\item $\Uqmg$ is a simple $\Bqg$-module.
\end{enumerate}
\end{lemma}

The $\mathbb{Q}(q)$-vector space $U_q^{-}(\g)$ has a unique non-degenerate symmetric bilinear form
$( \ , \ )$ satisfying the following properties (\cite[Proposition
3.4.4]{Kash91}):
\begin{equation*}
({\bf 1},{\bf 1})=1, \ \
  (e_i' u,v)=(u, f_iv)\quad \text{for $i\in I$ and $u,v\in \Uqmg$.}
\end{equation*}
The dual of $U^-_\A(\g)$ is defined to be
$$ U_\A^-(\g)^{\vee} \seteq \{ v \in U^-_q(\g) \ | \ (u,v) \in \A \text{ for all } u \in U_\A^-(\g) \}.$$

\begin{definition} \label{Def: Oint}
We denote by $\Oint$ the abelian category consisting of
$\Uqg$-modules $V$ satisfying the following conditions:
\begin{enumerate}
 \item $V$ has a {\em weight decomposition} with finite-dimensional weight spaces; i.e.,
 $$V\seteq \bigoplus_{\mu \in \mathtt{P} }V_{\mu}\quad \text{with}\quad  \dim V_{\mu}< \infty  ,$$
 where $V_{\mu}\seteq\{v \in V \mid q^h \, v = q^{\langle h, \mu \rangle}v \text{ for all }h \in \mathtt{P}^\vee \} $ ,
 \item there are finitely many $\lambda_1, \ldots, \lambda_s \in \mathtt{P}$ such that
 $$\wt(V)\seteq\{\mu \in \mathtt{P} \mid V_{\mu} \neq 0 \} \subset \bigcup^{s}_{i=1}(\lambda_i-\mathtt{Q}^{+}),$$
 \item for any $i\in I$, the action of $f_i$ on $V$ is locally nilpotent.
\end{enumerate}
\end{definition}

For each $\Lambda \in \mathtt{P}^+$, let $V(\Lambda)$ denote the
irreducible highest weight $\Uqg$-module with highest weight
$\Lambda$. It is generated by a unique highest weight vector
$v_{\Lambda}$ with the defining relations:
\begin{equation*}
 q^h v_{\Lambda} = q^{\langle h, \Lambda \rangle} v_{\Lambda}, \ \
 e_i v_{\Lambda} = 0, \ \
 f_i^{\langle h_i, \Lambda \rangle + 1} v_{\Lambda} =0 \ \
 \text{for all} \ h \in \mathtt{P}^{\vee} \text{ and } \ i \in I.
\end{equation*}
It  is known  that the category $\Oint$ is semisimple and
that all the irreducible objects are isomorphic to $V(\Lambda)$ for
some $\Lambda \in \mathtt{P}^{+}$. (See, for example, \cite[Theorem
3.5.4]{HK02}.)

There exists a unique
non-degenerate symmetric bilinear form $( \ , \ )$ on $V(\Lambda)$
$(\Lambda \in \mathtt{P}^+)$
satisfying
\begin{align*}
  (v_{\Lambda},v_{\Lambda})=1, \ \
  (e_i u,v)=(u, f_iv)\quad \text{for any $u,v\in V(\Lambda)$ and $i\in I$.}
 \end{align*}

 For a $\Uqg$-module $M$, an $\A$-form of $M$ is a $U_\A(\g)$-submodule $M_\A$
 such that $M=\Q(q)\otimes_\A M_\A$ and the weight space
 $(M_\A)_\lambda \seteq M_{\A} \cap M_{\lambda}$ is finitely generated over $\A$
 for any $\lambda\in \mathtt{P}$.
We define an $\A$-form $V_\A(\Lambda)$ of $V(\Lambda)$ by
$$ V_\A(\Lambda) = U_\A(\g) v_\Lambda.$$
The dual of $V_\A(\Lambda)$ is defined to be
$$ V_\A(\Lambda)^{\vee} = \{ v \in V(\Lambda) \ | \ (u,v) \in \A \text{ for all } u \in V_\A(\Lambda) \}.$$
Then $V_\A(\Lambda)^{\vee}$ is also an $\A$-form of $V(\Lambda)$. 

\medskip

 Now, we will give the definition of perfect bases and strong
perfect bases, and  prove their basic properties.  These properties
will play a crucial role in proving our main result (Theorem
\ref{th:main1}).

\medskip
Let $V=\soplus_{\lambda\in \mathtt{P}}V_\lambda$ be a $\mathtt{P}$-graded $\Q(q)$-vector space.
We assume that
there are finitely many $\lambda_1, \ldots, \lambda_s \in \mathtt{P}$ such that
 $$\wt(V)\seteq\{\mu \in \mathtt{P} \mid V_{\mu} \neq 0 \} \subset \bigcup^{s}_{i=1}(\lambda_i-\mathtt{Q}^{+}).$$
Furthermore, assume that $e_i$ ($i\in I$) acts on $V$  in such a
way that  $e_iV_\la\subset V_{\lambda+\al_i}$.
 For any $v \in V$ and $i \in I$, we define
$$ \varepsilon_i(v) \seteq \min \{ n \in \Z_{\ge 0} \mid e_i^{n+1} v=0 \}.$$
If $v=0$, then we will use the convention $\varepsilon_i(0)=-\infty$.
One can easily check that, for $k \in \Z_{\ge 0}$,
$$V^{< k}_i \seteq \{ v \in V \mid \varepsilon_i(v) < k \}= {\rm Ker} \ e^k_i.$$

\begin{definition}[\cite{BerKaz07,KOP11a}] \label{Def: perfect, strongly perfect}
\hfill
\bnum
\item A $\Q(q)$-basis $B$ of $V$ is called a {\em perfect basis} if
\begin{enumerate}
\item $B=\bigsqcup_{\mu\in \wt(V)} B_{\mu}$, where $B_{\mu}\seteq B \cap V_{\mu}$,
\item for any $ b \in B$ and $i \in I$ with $e_i(b) \neq 0$,
there exists a unique element $\tilde{\mathsf{e}}_i(b) \in B$ such that
\begin{align*}
    e_ib -c_i(b)\;\tilde{\mathsf{e}}_i(b)\in V^{< \varepsilon_i(b)-1}_i \mbox{ for some }
    c_i(b) \in \mathbb{Q}(q)^{\times},
\end{align*}
\item if $b,b'\in B$ and $i\in I$ satisfy $\eps_i(b)=\eps_i(b')>0$
and $\tilde{\mathsf{e}}_i(b) = \tilde{\mathsf{e}}_i(b')$, then $ b = b'$.
\end{enumerate}
\item We say that a perfect basis is {\em strong} if
$ c_i(b)= [\varepsilon_i(b)]_i$ for any $b \in B$ and $i\in I$; i.e.,
\begin{align}
    e_ib -[\eps_i(b)]_i \tilde{\mathsf{e}}_i(b)\in V^{< \varepsilon_i(b)-1}_i.
   \label{Eqn perfect basis}
\end{align}
\end{enumerate}
\end{definition}

For a perfect basis $B$, we set $\tilde{\mathsf{e}}_i(b)=0$ if $e_ib=0$.
We can easily see that for a perfect basis $B$
\eq
&&V^{<k}_i=\soplus_{b\in B,\,\eps_i(b)<k}\Q(q)b.
\label{prop:perfect}
\eneq

\begin{example}[\cite{Kash91,Kash93b}] Recall that the $U_q(\g)$-module
 $V(\Lambda)$ has the {\em upper global basis}
 $$G^{\vee}(\Lambda)\seteq\{ G^{\vee}(b) \ | \ b \in B(\Lambda)
\},$$
 which is parametrized by the crystal basis $B(\Lambda)$  of
 $V(\Lambda)$.

Then $G^{\vee}(\Lambda)$ is a strong perfect basis of
$V(\Lambda)$. Moreover, the upper global basis
$G^{\vee}(\Lambda)$  is an $\A$-basis of $V_\A(\Lambda)^\vee$.
\end{example}

For any sequence ${\bf i}=(i_1,\dots,i_m) \in I^{m}$ $(m \ge 1)$, we inductively
define a binary relation $\preceq_{{\bf i}}$ on $V \setminus \{0\} $ as
follows:
 \begin{equation*}
 \begin{aligned}
 \mbox{ if } {\bf i}=(i),\  v \preceq_{{\bf i}} v' &\Leftrightarrow \varepsilon_i(v) \le \varepsilon_i(v'), \\
 \mbox{ if } {\bf i}=(i;{\bf i}'), \ v \preceq_{{\bf i}} v'
&\Leftrightarrow
\begin{cases} \varepsilon_i(v) < \varepsilon_i(v')\\ \mbox{ or} \\
             \varepsilon_i(v)=\varepsilon_i(v'),\; e_i^{\varepsilon_i(v)}(v) \preceq_{{\bf i'}} e_i^{\varepsilon_i(v')}(v'). \end{cases}
 \end{aligned}
  \end{equation*}
We write
\begin{itemize}
\item $v \equiv_{{\bf i}} v'$ if $v \preceq_{{\bf i}} v'$ and $v' \preceq_{{\bf i}} v$,
\item $v' \prec_{{\bf i}} v$  if $v' \preceq_{{\bf i}} v$ and $v \not \equiv_{{\bf i}} v'$.
\end{itemize}
We can easily verify:
\be[(1)]
  \item for all $v \in V \setminus \{ 0 \}$,
 $V^{\prec_{{\bf i}} v }\seteq\{ 0 \} \bigsqcup
 \{ v' \in V \setminus \{ 0 \} \mid v' \prec_{{\bf i}} v \}$ forms a $\mathbb{Q}(q)$-linear subspace of V;
 indeed, we have $$V^{\prec_{{\bf i}} v }
 =\set{u\in V}{e_{i_m}^{\ell_m}\cdots e_{i_1}^{\ell_1}u=0, \
 e_{i_k}^{1+\ell_k}e_{i_{k-1}}^{\ell_{k-1}}\cdots e_{i_1}^{\ell_1}u=0\ \text{for $1\le
 k<m$}},$$
 where $v_0=v$, $\ell_k=\eps_{i_k}(v_{k-1})$ and $v_k=e_{i_k}^{\ell_k}v_{k-1}$ ($1\le k\le m$).

 \item if $v \not\equiv_{{\bf i}} v'$, then
 \begin{align*}   v+v' \equiv_{{\bf i}} \begin{cases} v &\text{if } v' \prec_{{\bf i}} v , \\
                                               v' &\text{if } v \prec_{{\bf i}} v' .  \end{cases} \end{align*}
 \ee

 For $i$ and $v\in V\setminus\{0\}$,
 we set
 $${e_i}^{\tp}(v)\seteq e_i^{(\varepsilon_i(v))}v.$$
 We can easily see the following:
 \eq&&\text{
 If $B$ is strong perfect, then
 ${e_i}^{\tp}(b)=\tilde{\mathsf{e}}_i^{\eps_i(b)}b$ for any $b\in B$.}
 \eneq
 For each ${\bf i}=(i_1,\dots,i_m) \in I^{m}$, define a map
 ${e_{{\bf i}}}^{\tp}\cl V\setminus \{ 0 \} \to V\setminus \{ 0 \}$ by
   $$ {e_{{\bf i}}}^{\tp} \seteq e^{\tp}_{i_m}\circ \cdots \circ e^{\tp}_{i_1}.$$
Then ${e_{{\bf i}}}^{\tp}B \subset \Q(q)^\times B$.
If $B$ is strong perfect, then
we have
${e_{{\bf i}}}^{\tp}B \subset B$

 Let $ V^{H}\seteq \set{ v \in V }{e_iv=0 \mbox{ for all } i \in I}$
 be the space of highest weight vectors in
  $V$ and let $ B^{H} =  V^{H} \cap B$ be the set of highest weight vectors in $B$.
Then, \eqref{prop:perfect} implies that
$$V^H=\soplus_{b\in B^H}\Q(q)b.$$

 In \cite{BerKaz07}, Berenstein and Kazhdan proved the following
version of uniqueness theorem for perfect bases.

\begin{theorem}[\cite{BerKaz07}] \label{Thm: perfect morphism}
Let $B$ and $B'$ be perfect bases of $V$ such that $B^H=(B')^H$.
Then there exist a map $\psi\cl B \isoto B'$ and a map $\xi\cl B \to
\mathbb{Q}(q)^\times$ such that $\psi(b)-\xi(b)b \in V^{\prec_{{\bf
i}} b}$ holds for any $b \in B$ and any ${\bf i}=(i_1,\dots,i_m)$
satisfying $e_{{\bf i}}^{\tp}(b) \in V^{H}$. Moreover, such $\psi$
and $\xi$ are unique and $\psi$ commutes with $\tilde{\mathsf{e}}_i$
and $\eps_i$ $(i \in I)$.
\end{theorem}

{}From the axioms of strong perfect basis and Theorem \ref{Thm:
perfect morphism}, one can easily prove the following corollary.

\begin{corollary} \label{Cor: strong perfect}
Let $B$ and $B'$ be strong perfect bases of $V$ such that
$B^H=(B')^H$. Then the map $\xi\cl B  \to \mathbb{Q}(q)^\times$
given in Theorem \ref{Thm: perfect morphism} is a trivial map; i.e.,
$$\psi(b)-b \in V^{\prec_{{\bf i}} b}\quad\text{for all $b \in B$.}$$
\end{corollary}

\begin{lemma} Let $B$ be a strong perfect basis of $V$.
\bnum
\item For any finite subset $S$ of $B$, there exists
a finite sequence $\mathbf{i}=(i_1,\ldots,i_m)$ of $I$ such that
$e_{\mathbf{i}}^\tp(b)\in B^H$ for any $b\in S$.
\item Let $b_0\in B^H$ and let $\mathbf{i}=(i_1,\ldots,i_m)$
be a finite sequence in $I$.
Then  $S\seteq\set{b\in B}{e_{\mathbf{i}}^\tp(b)=b_0}$
is linearly ordered by $\preceq_{\mathbf{i}}$.
\ee
\end{lemma}
\begin{proof}
(i) is almost evident.
In order to see (ii),
it is enough to show that
if $b,b'\in S$ satisfy $b\equiv_{\mathbf{i}}b'$,
then $b=b'$.
If we set $v_0=b$,
 $\ell_k=\eps_{i_k}(v_{k-1})$ and $v_k=e_{i_k}^{(\ell_k)}v_{k-1}$ ($1\le k\le m$),
 then $v_m=b_0$.
 Similarly, if we set $v_0'=b'$, $\ell'_k=\eps_{i_k}(v'_{k-1})$ and $v'_k=e_{i_k}^{(\ell'_k)}v'_{k-1}$ ($1\le k\le m$),
 then $\ell'_k=\ell_k$ and $v'_m=b_0$.
 Thus we have $v_k=\tilde{\mathsf{e}}_{i_k}^{\ell_k} v_{k-1}$ and
 $v'_k=\tilde{\mathsf{e}}_{i_k}^{\ell_k} v'_{k-1}$.
 Hence  Definition~\ref{Def: perfect, strongly perfect} (i) (c)
 shows that  $v'_k=v_k$ for all $k$.
 \end{proof}

The following proposition gives a  characterization of
$V_\A(\Lambda)^\vee$ by using strong perfect bases.

\begin{proposition} \label{Prop: dual-integral form}
Let $M$ be a $\Uqg$-module in $\Oint$ such that $\wt(M) \subset
\Lambda - \mathtt{Q}^+$. Let $M_\A$ be an $\A$-submodule of $M$.
Assume that $e_i^{(n)}M_\A\subset M_\A$, $(M_\A)_\Lambda = \A
v_\Lambda$ and $M$ has a strong perfect basis $B \subset M_{\A}$
such that $B^H= \{ v_{\Lambda} \}$.
Then we
have \bna
\item $M_\A \simeq V_\A(\Lambda)^\vee$,
\item $B$ is an $\A$-basis of $M_\A$.
\end{enumerate}
\end{proposition}

\begin{proof}
Since $M$ has only one highest weight vector $v_\Lambda$, $M$ is isomorphic to $V(\Lambda)$.

Since $(M_\A)_\Lambda = \A v_\Lambda$ and
\begin{align*}
V_\A(\Lambda)^\vee_\lambda  =\Bigl \{ u \in V(\Lambda)_\lambda \ \big| \
\begin{matrix}
e^{(a_1)}_{i_1} \cdots e^{(a_\ell)}_{i_\ell} u \in \A v_\Lambda \text{ for all } (i_1,\cdots,i_\ell)  \\
\text{ such that } \sum_{k=1}^\ell a_{k} \alpha_{i_k}+\lambda= \Lambda\\ \end{matrix}
 \Bigr \},
\end{align*}
it is obvious that $M_\A \subset V_\A(\Lambda)^\vee$.

Conversely, for $u\in V_\A(\Lambda)^\vee$ , write $u=\sum_{b\in B}c_bb$
with $c_b\in \Q(q)$. Let us show $c_b\in \A$ for any $\lambda\in \mathtt{P}$ and $b\in B_\lambda$.
We take $\mathbf{i}=(i_1,\ldots,i_m)$ such that
$e^{\mathrm{top}}_{\bf{i}}b=v_{\Lambda}$ for any $b\in B_\lambda$.
We shall show $c_b\in \A$ by the descending induction with respect to the linear order
$\preceq_{\mathbf{i}}$.
By the induction hypothesis, we may assume that
$c_{b'}\in\A$ for any $b'\in B$ such that $b\prec_{\mathbf{i}}b'$.

Then setting $v_0=b$, $\ell_k=\eps_{i_k}(v_{k-1})$ and
$v_k=e_{i_k}^{(\ell_k)}v_{k-1}$ ($1\le k\le m$), we have
$e_{i_m}^{(\ell_m)}\cdots e_{i_1}^{(\ell_1)}u=c_bv_\Lambda
+\sum\limits_{b\prec_{\mathbf{i}}b'}c_{b'}e_{i_m}^{(\ell_m)}\cdots
e_{i_1}^{(\ell_1)}b' \in V_\A(\Lambda)^\vee$. Hence we obtain
$c_b\in\A$.
\end{proof}

\section{ Supercategories and superbimodules} \label{Sec: supers}
In this section, we briefly review the notion of supercategory,
superfunctor, superbimodule and their basic properties. (See
\cite[Section 2]{KKT11} for more details.)

\subsection{ Supercategories } \hfill

\begin{definition} \
\bnum
\item
A {\em supercategory} is a category $\shc$ equipped with an
endofunctor $\Pi_\shc$ of $\shc$ and an isomorphism
$\xi_\shc\cl\Pi_\shc^2\overset{\sim}{\to}{\rm id}_{\shc}$ such that
$\xi_\shc\circ\Pi_\shc=\Pi_\shc\circ\xi_\shc \in
\Hom(\Pi_\shc^3,\Pi_\shc)$.
\item
For a pair of supercategories $(\shc,\Pi,\xi)$ and
$(\shc',\Pi',\xi')$, a {\em superfunctor} from $(\shc,\Pi,\xi)$ to
$(\shc',\Pi',\xi')$ is a pair $(F,\alpha_F)$ of a functor $F\cl
\shc\to\shc'$ and an isomorphism $\alpha_F\cl F\circ \Pi\overset{\sim}{\to}
\Pi'\circ F$ such that the diagram \eq&&\xymatrix@C=7ex{
F\circ\Pi^2\ar[r]^-{\alpha_F\circ\Pi}\ar[d]^-{F\circ \xi}&\Pi'\circ
F\circ\Pi
\ar[r]^-{\Pi'\circ\alpha_F}&\Pi'{}^2\circ F\ar[d]_{\xi'\circ F}\\
F\ar[rr]^{{\rm id}_F}&&F}
\eneq
commutes.
If $F$ is an equivalence of categories,
we say that $(F,\alpha_F)$ is an {\em equivalence of supercategories}.

\item Let $(F,\alpha_F)$ and $(F',\alpha_{F'})$ be superfunctors
from a supercategory  $(\shc,\Pi,\xi)$ to $(\shc',\Pi',\xi')$.
A morphism from $(F,\alpha_F)$ to $(F',\alpha_{F'})$ is a
morphism of functors $\varphi\cl F\to F'$ such that
$$\xymatrix@C=8ex@R=4ex{
F\circ\Pi\ar[r]^{\varphi\circ\Pi}\ar[d]_{\alpha_F}&F'\circ\Pi\ar[d]_{\alpha_{F'}}\\
\Pi'\circ F\ar[r]^{\Pi'\circ\varphi}&\Pi'\circ F'
}$$
commutes.
\end{enumerate}
\end{definition}
\noindent In this paper, a supercategory is assumed to be an
additive category.

For superfunctors $F\cl\shc\to\shc'$ and $F'\cl\shc'\to \shc''$, the
composition $F'\circ F\cl \shc\to\shc''$ becomes a superfunctor by
taking the composition
$$\xymatrix{F'\circ F\circ\Pi_{\shc}\ar[r]^{F'\circ \al_F}&
F'\circ\Pi_{\shc'}\circ F\ar[r]^{\al_{F'}\circ F}&\Pi_{\shc''}\circ F'\circ F}$$
as $\al_{F'\circ F}$.

The functors ${\rm id}_{\shc}$ and $\Pi$ are  superfunctors by
taking  $\alpha_{{\rm id}_{\shc}}={\rm id}_{\Pi}\cl{\rm
id}_{\shc}\circ \Pi \to \Pi \circ {\rm id}_{\shc}$ and
$\alpha_{\Pi}= - {\rm id}_{\Pi^2}\cl\Pi \circ \Pi \to \Pi \circ
\Pi$. {\it Note the sign}. The morphism $\alpha_{F}\cl F \circ \Pi
\to \Pi' \circ F$ is a morphism of superfunctors.

\subsection{Superalgebras}\label{sec:superalgebra}
A superalgebra is a $\Z_2$-graded algebra.
Let $A=A_0 \oplus A_1$ be a superalgebra. We denote by $\phi_{A}$ the
involution of $A$ given by $\phi_{A}(a)=(-1)^{\pa(\epsilon)}a$ for
$a\in A_\epsilon$ with $\epsilon=0,1$.
We call $\phi_A$ the {\em parity involution}.

The category of
$A$-modules $\Mod(A)$ is naturally endowed with a structure of
supercategory. The functor $\Pi$ is induced by the parity involution $\phi_A$. Namely, for
$M\in\Mod(A)$, $\Pi M\seteq\set{ \pi(x)}{x \in M }$ with
$\pi(x)+\pi(x')=\pi(x+x')$ and $a \pi(x) = \pi(\phi_A(a)x)$ for $a
\in A$. The morphism $\xi\cl \Pi^2\to\id$ is given by
$\pi\bl\pi(x)\br\mapsto x$.

An $A$-supermodule is an  $A$-module with a decomposition $M= M_0
\oplus M_1$ such that $A_\epsilon M_{\epsilon'} \subset
M_{\epsilon+\epsilon'}$ ($\epsilon,\epsilon' \in \Z_2$).

Let $\Mod_{{\rm super}}(A)$ be the category of $A$-supermodules. The
morphisms in this category are $A$-linear homomorphisms which
preserve the $\Z_2$-grading. Then $\Mod_{{\rm super}}(A)$ is also
endowed with a structure of supercategory. The functor $\Pi$ is
given by the parity change: namely, $(\Pi
M)_\epsilon\seteq\set{\pi(x)}{x\in M_{1-\epsilon}}$ ($\epsilon=0,1$)
and $a\pi(x)=\pi(\phi_{A}(a)x)$ for $a\in A$ and $x\in M$. The
isomorphism $\xi_M\cl\Pi^2M\to M$ is given by
$\pi\bl\pi(x)\br\mapsto x$ ($x\in M$). Then there is a canonical
superfunctor $\Mod_{{\rm super}}(A)\to\Mod(A)$. For an
$A$-supermodule $M$, the {\em parity involution} $\phi_M\cl M\to M$
of $M$ is defined by
$\phi_M|_{M_\epsilon}=(-1)^\epsilon\,\id_{M_\epsilon}$. Then we have
$\phi_M(ax)=\phi_{A}(a)\phi_M(x)$ for any $a\in A$ and $x\in M$.

Let $A$ and $B$ be superalgebras. We define the multiplication on
the tensor product $A \otimes B$  by
\begin{align*}
(a_1\otimes b_1)(a_2\otimes b_2)=(-1)^{\eps'_1\eps_2}(a_1a_2)\otimes (b_1b_2)
\end{align*}
for $a_i\in A_{\eps_i}$, $b_i\in B_{\eps'_i}$ ($\eps_i,\eps_{i}'=0,1$). Then $A \otimes B$ is
again a superalgebra.  Note that we have $A\otimes
B\cong B\otimes A$  as a superalgebra by the supertwist map
\begin{align*}
A\otimes B\isoto B\otimes A,\quad a\otimes b\longmapsto (-1)^{\eps_1\eps_2}b\otimes a
\quad \text{($a\in A_{\eps_1}$, $b\in B_{\eps_2}$).}
\end{align*}
If $M$ and $N$ are an $A$-module and an $B$-module respectively, then
$M\otimes N$ has a structure of $A\otimes B$-module by
\begin{align*}
(a\otimes b)(u\otimes v)=(-1)^{\eps\eps'}(au)\otimes (bv)
\end{align*}
for $a\in A$, $b\in B_{\eps}$, $u\in M_{\eps'}$, $v\in N$ ($\eps,\eps'=0,1$).

\subsection{Superbimodules}
Let $A$ and $B$ be superalgebras. An {\em
$(A,B)$-superbimodule} is an $(A,B)$-bimodule with a
$\Z_2$-grading compatible with the left action of $A$ and the right
action of $B$. For an $(A,B)$-superbimodule $L$, let $F_L\cl
\Mod(B)\to\Mod(A)$ be the functor $N \mapsto L\otimes_BN$. Then
$F_L$ is a superfunctor, where
$$\alpha_{F_L}\cl F_L\Pi N=L\otimes_B\Pi N\to \Pi F_L N=\Pi(L\otimes_BN)$$
is given by $s\otimes \pi(x)\mapsto \pi(\phi_L(s)\otimes x) \ (s\in L, \ x\in N)$.

For an $(A,B)$-superbimodule $L$, its parity twist $\Pi L$ is $\Pi
L$ as a left $A$-module and its right $B$-module structure is given
by $a\pi(s)b=\pi(\phi_A(a)sb)$ ($s\in L$, $a\in A$, $b\in B$). Then
there exists a canonical isomorphism of superfunctors $\eta\cl
F_{\Pi L}\overset{\sim}{\to} \Pi\circ F_L$. The isomorphism
$\eta_N\cl (\Pi L)\otimes_BN\overset{\sim}{\to} \Pi(L\otimes_BN)$ is
given by $\pi(s)\otimes x\mapsto \pi(s\otimes x)$. It is an
isomorphism of superfunctors since one can easily check the
commutativity of the following diagram:
$$\xymatrix@C=10ex@R=2ex{
F_{\,\Pi L}\circ\Pi\ar[r]^-{\eta\circ\Pi}\ar[dd]_{\alpha_{F_{\Pi L}}}
&\Pi\circ F_L\circ \Pi\ar[dr]^{\Pi\circ\alpha_{F_L}}\ar[dd]^{\alpha_{(\Pi\circ F_L)}}\\
&&\Pi\circ\Pi\circ F_L\ar[dl]^(.3){\hs{3ex}\alpha_\Pi\circ F_L=-{\rm id}_{\Pi\circ\Pi\circ F_L}}\\
\Pi\circ F_{\,\Pi L}\ar[r]^-{\Pi\circ\eta}&\Pi\circ\Pi\circ F_L.
}
$$
by using $\phi_{\Pi L}(\pi(s))=-\pi(\phi_L(s))$.

\section{The quiver Hecke superalgebras}
\label{Sec: Quiver Hecke superalgebra} In this section we recall the
construction of quiver Hecke superalgebras
and investigate its basic
properties. We take as a base ring a graded commutative ring $\k=
\bigoplus_{n \in \Z_{\ge 0}} \k_n$.
We assume the following condition:
\eq
&&\text{$2$ is invertible in $\k_0$.}
\eneq

\subsection{ Definition of quiver Hecke superalgebras}
We assume that a decomposition $I = \Iev \bigsqcup \Iod$ is given.
We say that a Cartan matrix $\car=(a_{ij})_{i,j \in I}$ is {\it
colored} by $\Iod$ if
$$ a_{ij} \in 2 \Z \quad \text{ for all } i \in \Iod \text{ and } \ j \in I. $$
{}From now on, we assume that $\car$ is colored by $\Iod$.

We define the {\it parity} function $\pa \cl I \to \{ 0, 1 \}$ by
$$\pa(i)=1  \quad \text{ if } i \in \Iod \quad \text{ and } \quad \pa(i)=0  \quad \text{ if } i \in \Iev.$$
Then we naturally extend the parity function on $I^n$ and $\mathtt{Q}^+$ as follows:
\begin{align*}
& \pa(\nu) \seteq \sum^n_{k=1} \pa(\nu_k) \quad \text{for all $\nu \in I^n$,} \\
& \pa(\beta) \seteq  \sum_{k=1}^r\pa(i_k)\quad \text{for $\beta=
\sum^r_{k=1} \alpha_{i_k} \in \mathtt{Q}^+$.}
\end{align*}

For $i \neq j \in I$ and $r,s \in \Z_{\ge0}$, we take
$t_{i,j;(r,s)} \in \k_{-2(\alpha_i|\alpha_j)-r(\alpha_i|\alpha_i)-s(\alpha_j|\alpha_j)}$
such that
$$
\begin{aligned}
& t_{i,j;(-a_{ij},0)} \in \k_0^{\times},  \ \
t_{i,j;(r,s)}=t_{j,i;(s,r)}, \\
& t_{i,j;(r,s)}=0 \ \ \text{ if $i \in \Iod$ and $r$ is odd}.
\end{aligned}
$$
\noindent Let $\P_{ij}\seteq \k\langle w,z \rangle/\langle zw
-(-1)^{\pa(i)\pa(j)}wz \rangle$ be the $\Z\times\Z_2$-graded
$\k$-algebra where  $w$ and $z$ have the $\Z\times \Z_2$-degree
$((\alpha_i |\alpha_i),\pa(i))$ and $((\alpha_j |\alpha_j),\pa(j))$,
respectively. Let  $\cQ_{i,j}$ be an element of $\P_{ij}$ which is
of the form
\begin{equation} \label{Def: Qij}
 \cQ_{i,j}(w,z)= \begin{cases}
                            \sum_{r,s\in\Z_{\ge0}} t_{i,j;(r,s)} w^r z^s & \text{ if } i \neq j,\\
0 & \text{ if } i = j.                           \end{cases}
\end{equation}

Then  $\cQ_{i,j}(w,z)$ is an even element and $\cQ=(\cQ_{i,j})_{i,j\in I}$ satisfies
\begin{equation} \label{eqn:Q even}
\begin{aligned}
&\cQ_{i,j}(w,z) = \cQ_{j,i}(z,w) \quad \text{for  $i,j \in I $,} \\
&\cQ_{i,j}(w,z)=\cQ_{i,j}(-w,z) \quad  \text{for $i \in \Iod$ and $j
\in I$.}
\end{aligned}
\end{equation}

\begin{definition}[\cite{KKT11}] \label{def:Quiver Hekce superalg} \
The {\em quiver Hecke superalgebra} $R(n)$ of degree $n$ associated with
the Cartan datum $(\car,\mathtt{P},\Pi,\Pi^{\vee})$ and
$(\cQ_{i,j})_{i,j\in I}$
 is the superalgebra over $\k$ generated by $e(\nu)$ $(\nu \in I^n)$, $x_k$ $(
 1 \le k \le n)$, $\tau_{a}$ $(1 \le a \le n-1)$ with the parity
 \begin{equation} \label{eqn: parity}
 \pa(e(\nu))=0, \quad \pa(x_k e(\nu))= \pa(\nu_k), \quad \pa(\tau_a e(\nu))=\pa(\nu_a)\pa(\nu_{a+1})
 \end{equation}
 subject to the following defining relations:
\begin{enumerate}[{\rm(R1)}]
\item $e(\mu)e(\nu)=\delta_{\mu,\nu}e(\nu) \text{ for all } \mu, \ \nu \in I^n$,
     and
      $1 = \sum_{\nu \in I^n} e(\nu)$,

\item $x_px_qe(\nu)= (-1)^{\pa(\nu_p)\pa(\nu_q)}x_qx_p e(\nu) \quad \text{ if } p \neq q,$
\item $x_pe(\nu)=e(\nu)x_p$ and $ \tau_ae(\nu) = e(s_a \,\nu) \tau_a$,
where $s_a=(a,a+1)$ is the transposition on the set of sequences,
\item $\tau_ax_pe(\nu)= (-1)^{\pa(\nu_p)\pa(\nu_a)\pa(\nu_{a+1})}x_p\tau_ae(\nu)$, if $p \neq a, \ a+1,$
\item \begin{align*}
     (\tau_ax_{a+1}-(-1)^{\pa(\nu_a)\pa(\nu_{a+1})}x_{a}\tau_a)e(\nu)
      &=(x_{a+1}\tau_a-(-1)^{\pa(\nu_a)\pa(\nu_{a+1})}\tau_ax_{a})e(\nu) \\
      &=\delta_{\nu_a,\nu_{a+1}}e(\nu), \end{align*}
\item $\tau_a^2e(\nu)=\cQ_{\nu_a,\nu_{a+1}}(x_a,x_{a+1})e(\nu)$,
\item $\tau_a\tau_be(\nu)=(-1)^{\pa(\nu_a)\pa(\nu_{a+1})\pa(\nu_b)\pa(\nu_{b+1})} \tau_b\tau_ae(\nu)$
    if $|a-b|>1$,
\item \begin{align*}
    & (\tau_{a+1}\tau_{a}\tau_{a+1}-\tau_{a}\tau_{a+1}\tau_{a})e(\nu) \\
    & \quad = \begin{cases}
    \dfrac{\cQ_{\nu_{a},\nu_{a+1}}(x_{a+2},x_{a+1})-
    \cQ_{\nu_{a},\nu_{a+1}}(x_{a},x_{a+1}) }{x_{a+2}-x_{a}}e(\nu)
        \qquad \text{ if } \nu_{a}=\nu_{a+2} \in \Iev, \\
    (-1)^{\pa(\nu_{a+1})}(x_{a+2}-x_a)
\dfrac{\cQ_{\nu_{a},\nu_{a+1}}(x_{a+2},x_{a+1})-\cQ_{\nu_{a},\nu_{a+1}}(x_{a},x_{a+1})}{x^2_{a+2}-x^2_{a}}e(\nu) \\
    \qquad\qquad\qquad\qquad\qquad\qquad\qquad\qquad\qquad\qquad\quad
    \text{ if } \nu_{a}=\nu_{a+2} \in \Iod, \\
    \qquad\qquad\qquad\qquad\quad 0
    \qquad\qquad\qquad\qquad\qquad\quad \ \
    \text{ otherwise }.
    \end{cases}\end{align*}
\end{enumerate}%
The $\Z$-grading on
$R(n)$ is given by:
 \begin{equation} \label{eqn: Z-grading}
 \degZ(e(\nu))=0, \quad \degZ(x_k e(\nu))= (\alpha_{\nu_k}|\alpha_{\nu_k}),
 \quad \degZ(\tau_a e(\nu))=-(\alpha_{\nu_a}|\alpha_{\nu_{a+1}}).
 \end{equation}
 \end{definition}

\bigskip
 We understand $R(0) \simeq \k$, and $R(1)$ is isomorphic to $\k^I[x_1]$ where $\k^I=\bigoplus_{i \in I} \k e(i)$
is the direct sum of the copies $\k e(i)$ of the algebra $\k$.

For $\nu=(\nu_1,\ldots,\nu_n) \in I^{n}$ and $1 \le m \le n$, we set
{\allowdisplaybreaks
\begin{align}
\nu_{<m} \seteq(\nu_1,\ldots,\nu_{m-1}) \quad \text{ and } \quad \nu_{>m}\seteq(\nu_{m+1},\ldots,\nu_{n}).
\end{align}
\noindent For $a,b \in \{ 1, \ldots, n \}$ with $a \neq b$, we
define the elements of $R(n)$ by
\begin{align}
e^{{\rm ev}}_{a,b}= \sum_{ \substack{\nu \in I^n, \\ \nu_a=\nu_b \in \Iev} }e(\nu), \quad
e^{{\rm od}}_{a,b}= \sum_{ \substack{\nu \in I^n, \\ \nu_a=\nu_b \in \Iod} }e(\nu)
\quad \text{ and } \quad e_{a,b}=e^{{\rm ev}}_{a,b} + e^{{\rm od}}_{a,b}.
\end{align}

For any $\nu \in I^n$ ($n \ge 2$), let
$$\P_\nu\seteq \k \langle x_1,\ldots,x_n \rangle /
    \langle x_ax_b - (-1)^{\pa(\nu_a)\pa(\nu_b)}x_bx_a \rangle_{1 \le a < b \le n}$$
be the superalgebra generated by $x_k$ ($1\le k\le n$) with
$\Z\times\Z_2$-degree $\bl(\alpha_{\nu_k}|\alpha_{\nu_k}),
\pa({\nu_k})\br$ ($k=1,\ldots$,n). Let $\Pev_\nu$ be the subalgebra
of $\P_\nu$ generated by $x_k^{1+\pa(\nu_k)}$ $(1 \le k \le n)$.
Then $\Pev_\nu$ is isomorphic to the polynomial ring
$\k[x_1^{1+\pa(\nu_1)}, \cdots x_n^{1+\pa(\nu_n)}]$. Set
$$\P_n=\bigoplus_{\nu \in I^n} \P_\nu e(\nu)\quad \text{ and } \quad
\Pev_n=\bigoplus_{\nu \in I^n} \Pev_\nu e(\nu).$$
Then $\Pev_n$ is contained in the center of $\P_n$ and
\begin{equation} \label{eqn: center Q}
\cQ_{\nu_a,\nu_{a+1}}(x_a,x_{a+1})e(\nu) \text{ belongs to }
\Pev_n \text{ for all } \nu \in I^n \text{ and } 1 \le a < n.
\end{equation}

For $1 \le k < n$, we define the algebra endomorphism $\oS_k$ of
$\P_n$ by \eq &&\oS_k(x_pe(\nu))=
(-1)^{\pa(\nu_k)\pa(\nu_{k+1})\pa(\nu_p)}x_{s_k(p)}e(s_k \nu) \quad
\text{for $1 \le p \le n$,}\label{def:sk} \eneq where $s_k=(k,k+1)
\in S_n$ is the transposition which acts on $I^n$ in a natural way.
\noindent For $f \in \P_n$ and $1 \le k < n$, define
\begin{equation} \label{Eqn: partial}
\begin{aligned}
& \partial_k f = \dfrac{f-\oS_k f}{x_{k+1}-x_k}e^{{\rm ev}}_{k,k+1} +
\dfrac{(x_{k+1}-x_k)f-\oS_k f(x_{k+1}-x_k)}{x_{k+1}^2-x_k^2}e^{{\rm od}}_{k,k+1}, \\
& f^{\partial_k} = \dfrac{f-\oS_k f}{x_{k+1}-x_k}e^{{\rm ev}}_{k,k+1} +
\dfrac{f(x_{k+1}-x_k)-(x_{k+1}-x_k)(\oS_k f)}{x_{k+1}^2-x_k^2}e^{{\rm od}}_{k,k+1}.
\end{aligned}
\end{equation}
\noindent Then one can easily show that
\begin{align} \label{Eqn: tau f}
& \partial_k f, \ f^{\partial_k} \in \P_n, \quad
 \tau_k f = (\oS_k f) \tau_k+ \partial_k f, \quad  f \tau_k  = \tau_k (\oS_k f) + f^{\partial_k}
\end{align}
and
\begin{align*}
& \partial_k(x_j) = (x_j)^{\partial_k}=\delta_{j,k+1}e^{{\rm ev}}_{k,k+1}-\delta_{j,k}e^{{\rm ev}}_{k,k+1}+
\delta_{j,k+1}e^{{\rm od}}_{k,k+1}+\delta_{j,k}e^{{\rm od}}_{k,k+1}, \\
& \partial_k(fg) =(\partial_k f)  g+ (\oS_k f) \partial_k g, \quad
  (fg)^{\partial_k} =f (g^{\partial_k})  + (f^{\partial_k})  \oS_k g.
\end{align*}

For $\beta \in \mathtt{Q}^+$ with $|\beta|=n$, set
\begin{align*}
& I^\beta=\{ \nu=(\nu_1,\ldots,\nu_n) \in I^n \ | \ \alpha_{\nu_1}+ \cdots + \alpha_{\nu_n}=\beta \}.
\end{align*}
We define
\begin{align*}
& R(m,n)=R(m)\otimes_{\k}R(n)\subset R(m+n), \\
& e(n)=\sum_{\nu \in I^n}e(\nu), \quad e(\beta)=\sum_{\nu \in I^\beta}e(\nu), \quad
       e(\alpha,\beta)=\sum_{\mu \in I^\alpha, \ \nu \in I^\beta}e(\mu,\nu), \\
& R(\beta)=e(\beta)R(n),  \quad R(\alpha,\beta)=R(\alpha)\otimes_{\k}R(\beta)\subset R(\alpha+\beta),  \\
& e(n,i^k)=\sum_{ \substack{\nu \in I^{n+k}, \\ \nu_{n+1}=\cdots=\nu_{n+k}=i} }e(\nu), \quad
  e(i^k,n)=\sum_{ \substack{\nu \in I^{n+k}, \\ \nu_{1}=\cdots=\nu_{k}=i} }e(\nu),\\
& e(\beta,i^k)=e(\beta,k\alpha_i)= e(\beta+k\alpha_i)e(n,i^k),
\\
&e(i^k,\beta)=e(k\alpha_i,\beta)=e(\beta+k\alpha_i)e(i^k,n)
\end{align*}
for $\alpha,\beta \in \mathtt{Q}^+$.

\begin{proposition} [{\cite[Corollary 3.15]{KKT11}}] \label{Prop: PBW}
For each $w \in S_n$, we choose a reduced expression $s_{i_1}\cdots
s_{i_\ell}$ of $w$ and write $\tau_w=\tau_{i_1}\cdots\tau_{i_\ell}$.
Then
$$\{ x_1^{a_1}\cdots x_n^{a_n} \tau_w e(\nu) \ | \ a=(a_1,\ldots,a_n) \in \Z_{\ge 0}^n, \ w \in S_n, \ \nu \in I^n \} $$
forms a basis of the $\k$-module $R(n)$.
\end{proposition}

\begin{remark} \label{Rmk: group algebra}
In general, $\tau_w$ depends on the choice of reduced expressions of
$w$. However, we still  write $\tau_w$ after choosing a reduced
expression of $w$. In $ I =\{ i \}$ case, by the axioms  in
Definition $\ref{def:Quiver Hekce superalg}$, $\pm\tau_w$ does not depend on the
choice of reduced expressions of $w \in S_n$;
 i.e., for any two reduced expressions
$w=s_{i_1} \cdots s_{i_r}=s_{j_1} \cdots s_{j_r}$, we have
 $$ \tau_{i_1} \cdots \tau_{i_r} = \pm \tau_{j_1} \cdots \tau_{j_r}.$$
\end{remark}

By the proposition above we have:
\begin{lemma}  The algebra $R(n+1)$ has a direct sum decomposition 
\begin{equation} \label{eqn: R(n+1) wrt R(n,1)}
R(n+1)= \soplus_{a=1}^{n+1}R(n,1)\tau_n \cdots \tau_a =
\soplus_{a=1}^{n+1}R(n)\otimes \k^I[x_{n+1}] \tau_n \cdots \tau_a.
\end{equation}
In particular, $R(n+1)$ is a free $R(n,1)$-module of rank
$n+1$.
\end{lemma}

Let $\Mod(R(\beta))$ (resp.\ $\Proj(R(\beta))$) be the category of
arbitrary (resp.\ finitely generated projective) $\Z$-graded
$R(\beta)$-modules. We denote by $\Rep(R(\beta))$ the category of
$\Z$-graded left $R(\beta)$-modules which are coherent over $\k_0$.
Note that we assume modules to be only $\Z$-graded but not
$\Z\times\Z_2$-graded. The morphisms in these categories are
$R(\beta)$-linear homomorphisms which preserve $\Z$-grading. Then
$\Mod(R(\beta))$ is an abelian category, and $\Proj(R(\beta))$,
$\Rep(R(\beta))$ are subcategories of $\Mod(R(\beta))$ stable under
extensions. Hence they are exact categories.

Since $R(\beta)$ is a superalgebra, the category
$\Mod(R(\beta))$ has a supercategory structure induced by the
parity involution $\phi\seteq\phi_{R(\beta)}$ as seen in \S\;\ref{sec:superalgebra}.

For a $\Z$-graded $R(\beta)$-module $M=\bigoplus_{i\in \Z}M_i$, let
$M\langle k \rangle$ denote the $\Z$-graded $R(\beta)$-module
obtained from $M$ by shifting the grading by $k$; i.e., $M \langle k
\rangle_i \seteq \bigoplus_{i\in \Z} M_{i+k}$. We also denote by $q$
the  grading  shift functor \eq&&(qM)_i=M_{i-1}.\eneq

Let $\pi$ be an {\em odd} element with the defining equation
$\pi^2=1$. For any superring $R$, we define
$$R^\pi\seteq R \ot_\Z \Z[\pi]\simeq R\oplus R\pi.$$
Thus $\A^\pi=\Z[q,q^{-1},\pi]$ with $\pi^2=1$. We denote by
$[\Proj(R(\beta))]$ and $[\Rep(R(\beta))]$ the Grothendieck group of
$\Proj(R(\beta))$ and $\Rep(R(\beta))$, respectively. Then $[\Proj
(R(\beta))]$ and $[\Rep (R(\beta))]$ have the $\A^\pi$-module
structure given by $q[M] = [qM]$ and $\pi[M]=[\Pi M]$, where $[M]$
is the isomorphism classes of an $R(\beta)$-module $M$.

For $\alpha, \beta \in \mathtt{Q}^+$, consider the natural embedding
\begin{align*}
\iota_{\alpha,\beta}\cl R(\alpha) \otimes R(\beta) \hookrightarrow R(\alpha+\beta),
\end{align*}
which maps $e(\alpha) \otimes e(\beta)$ to $e(\alpha,\beta)$. For $M \in \Mod(R(\alpha,\beta))$ and
$N \in \Mod(R(\alpha+\beta))$, we define
\begin{align*}
\ind_{\alpha, \beta} M &= R(\alpha+\beta)e(\alpha,\beta)\otimes_{R(\alpha,\beta)}M
\in \Mod(R(\alpha+\beta)), \\
\res_{\alpha, \beta} N &= e(\alpha, \beta) N  \in \Mod(R(\alpha,\beta)).
\end{align*}
Then the Frobenius reciprocity
holds:
\begin{align} \label{Eq:reciprocity}
\Hom_{R(\alpha+\beta)}(\ind_{\alpha,\beta}M, N) \simeq \Hom_{R(\alpha,\beta)}(M, \res_{\alpha,\beta}N).
\end{align}

Given $\alpha, \alpha', \beta, \beta' \in \mathtt{Q}^+$ with $\alpha+\beta = \alpha' + \beta'$, let
$$ _{\alpha, \beta}R_{\alpha', \beta'} \seteq e(\alpha, \beta) R(\alpha+\beta) e(\alpha', \beta') . $$
We also write ${}_{\alpha+\beta}R_{\alpha, \beta}\seteq
R(\alpha+\beta)e(\alpha, \beta)$ and
${}_{\alpha',\beta'}R_{\alpha'+\beta'}\seteq e(\alpha',
\beta')R(\alpha'+\beta')$. Note that $_{\alpha, \beta}R_{\alpha',
\beta'}$ is a $\Z \times \Z_2$-graded $(R(\alpha,\beta),
R(\alpha',\beta'))$-superbimodule. Now we obtain Mackey's Theorem
for quiver Hecke superalgebras.

\begin{proposition} \label{Prop:Mackey}
The $\Z \times \Z_2$-graded $(R(\alpha) \otimes R(\beta), R(\alpha') \otimes R(\beta'))$-superbimodule 
$_{\alpha,\beta}R_{\alpha', \beta'}$ has a
graded filtration with graded subquotients isomorphic to
$$ \Pi^{\pa(\gamma)\pa(\beta+\gamma-\beta')}
({_{\alpha}R_{\alpha-\gamma, \gamma}}) \otimes
({_{\beta}R_{\beta+\gamma-\beta', \beta'-\gamma}}) \otimes_{R'}
({_{\alpha-\gamma, \alpha'+\gamma-\alpha}R_{\alpha'}}) \otimes
({_{\gamma, \beta'-\gamma}R_{\beta'}}) \langle -(\gamma | \beta +
\gamma - \beta') \rangle ,$$ where $R' = R(\alpha-\gamma)\otimes
R(\gamma)\otimes R(\beta + \gamma - \beta') \otimes R(\beta' -
\gamma)$ and $\gamma$ ranges over the set of $\gamma \in
\mathtt{Q}^+$ such that $\alpha-\gamma$, $\beta'-\gamma$  and
$\beta+\gamma-\beta'=\alpha'+\gamma-\alpha$ belong to
$\mathtt{Q}^+$. 
\end{proposition}

The proof is similar to that of \cite[Proposition 2.18]{KL1}.

\medskip

Hereafter, {\em an $R(n)$-module always means  a $\Z$-graded
$R(n)$-module.}

%%%%%%%%%%%%%R(k\alpha_i) P(i^n) L(i^n)%%%%%%%%%%%%%%%%%%
\subsection{ The algebra $R(n\alpha_i)$ }

In this subsection, we briefly review the results about
$R(n\alpha_i)$ developed in \cite{EKL,HW,KL1,R08}. For the sake of
simplicity, we assume that \eq&&\text{ $\k_0$ is a field and the
$\k_i$'s are finite-dimensional over $\k_0$} \label{cond:k0} \eneq
throughout this subsection.

Under the condition \eqref{cond:k0},
the $\Z$-graded algebra $R(\beta)$ satisfies the conditions:
\be[{\rm(a)}]
\item its $\Z$-grading is bounded below and
\item each homogeneous
subspace $R(\beta)_t$ is finite-dimensional over $\k_0$ ($t \in \Z$). \ee Hence
we have \eq&&\parbox{75ex}{ \bnum
\item $R(\beta)$ has the Krull-Schmidt direct sum property for finitely generated modules,
\item Any simple module in $\Mod(R(\beta))$ is finite-dimensional over $\k_0$ and has an indecomposable
finitely generated projective cover
(unique up to isomorphism),
\item there are finitely many simple modules in $\Rep(R(\beta))$
up to grade shifts and isomorphisms.
\ee
}\label{property of R(beta)-mod}
\eneq

We now consider the case $\beta=n\alpha_i$.
For $1 \le k < n$, let $\mathbf{b}_k\seteq \tau_k x_{k+1} \in R(n\al_i)$. Then, by 
a direct computation, we have 
\begin{equation} \label{eqn: braid b_s}
\begin{aligned}
& \mathbf{b}_r \mathbf{b}_s = \mathbf{b}_s \mathbf{b}_r \quad \text{ if } |r-s|>1, \\
& \mathbf{b}_r \mathbf{b}_{r+1} \mathbf{b}_r = \mathbf{b}_{r+1}\mathbf{b}_r\mathbf{b}_{r+1}, \quad
\text{ for } 1 \le r < n-1.
\end{aligned}
\end{equation}
Thus, for $w \in S_n$, $\mathbf{b}_w$ is well-defined.

We denote by $w[1,n]$  the longest element of the symmetric group $S_n$,  and set
\begin{equation} \label{eqn: b(w[1,n])}
\mathbf{b}(i^n)\seteq\mathbf{b}_{w[1,n]}.
\end{equation}
Then \eqref{eqn: braid b_s} implies
$$\mathbf{b}(i^n)^2=\mathbf{b}(i^n) \quad \text{ and } \quad
\mathbf{b}_k\mathbf{b}(i^n)=\mathbf{b}(i^n)\mathbf{b}_k=\mathbf{b}(i^n) \quad \text{ for } 1 \le k < n.$$

Set  \eq &&\Pi_i=\Pi^{\pa(i)}, \qquad \pi_i\seteq\pi^{\pa(i)}
\eneq  and define
\begin{equation}
 \begin{aligned}
 \ &[n]^\pi_i =\frac{ (\pi_iq_i)^n - q^{-n}_{i} }{ \pi_iq_i - q^{-1}_{i} },
 \ &[n]^\pi_i! = \prod^{n}_{k=1} [k]^\pi_i .
 \end{aligned}
\end{equation}
Recall that $q_i=q^{(\alpha_i|\alpha_i)/2}$.
The algebra $R(n\alpha_i)$ decomposes into the direct sum of  indecomposable
projective $\Z \times \Z_2$-graded modules as follows:
\eq\label{eq:divided}
&&R(n\alpha_i) \simeq [n]_i^\pi !  P(i^{n})=\soplus_{k=0}^{n-1}q_i^{1-n+2k}\Pi_i^{k}P(i^{n}),
\eneq
where
$$ P(i^{n})\seteq \Pi_i^{\frac{n(n-1)}{2}} R(n\alpha_i) \mathbf{b}(i^n)
                \left \langle \dfrac{n(n-1)}{4} (\alpha_i | \alpha_i) \right \rangle.$$
Note that $P(i^{n})$ is an indecomposable projective $\Z \times \Z_2$-graded module unique up to isomorphism and
 $\Z \times \Z_2$-grading shift.
 Note also that
 $R(n\alpha_i) \mathbf{b}(i^n)\simeq R(n\alpha_i)/\sum\limits_{k=1}^{n-1}R(n\alpha_i)\tau_k$.

On the other hand, there exists an irreducible $\Z\times \Z_2$-graded $R(n\alpha_i)$-module $L(i^n)$
which is unique up to isomorphism and
 $\Z \times \Z_2$-grading shift:
\begin{equation} \label{Eqn: Def of L(i^n)}
 L(i^n)\seteq \ind^{R(n\alpha_i)}_{\k[x_1] \ot \cdots \ot \k[x_n]} {\mathbf 1},
\end{equation}
where ${\mathbf 1}$ is the trivial $\k[x_1] \ot \cdots \ot \k[x_n]$-module
which is isomorphic to $\k_0$.
Hence $L(i^n)$ is the $R(n\alpha_i)$-module generated by the element
$u(i^n)$ of $\Z\times \Z_2$ degree 
$(0,0)$ with the defining relation $$x_ku(i^n)=0 \ (1\le k\le n), \
\  \k_su(i^n)=0 \  (s>0).$$

By Proposition \ref{Prop: PBW},
the $R(n\alpha_i)$-module $L(i^n)$ has a $\k_0$-basis
$$\set{ \tau_w \cdot u(i^n) }{ \ w \in S_n  }.$$
Set
$$ L_k\seteq \{ v \in L(i^n) \ | \ x_n^k \cdot v =0 \} \quad (k \ge 0).$$
Since $x_n$ anticommutes with all $x_i$ ($i=1,\ldots,n-1$) and
$\tau_j$ ($j=1,\ldots,n-2$),  $L_k$ has an
$R((n-1)\alpha_i)$-module structure. Moreover $L_k$ has a
$\k_0$-basis
$$\set{ \tau_w \tau_{n-1} \cdots \tau_{s} u(i^n) }{ w \in S_{n-1}, n-k+1\le s\le n  },$$
and $L_{n}=L(i^n)$. Thus we have a supermodule isomorphism
\begin{equation} \label{Eqn: Decom L(i^n) into L(i^n-1) L(1)}
L_k / L_{k-1} \simeq  \Pi^{k-1}_i L(i^{n-1}) \left\langle (1-k)(\alpha_i|\alpha_i) \right\rangle
\quad \text{ for } 1\le k \le n.
\end{equation}
Here the $\Z\times\Z_2$-grading shift is caused by
the ($\Z\times\Z_2$)-degree of $\tau_{n-1}\cdots\tau_{n-k+1}u(i^n)$.
Note that $L_k=x_n^{n-k}L(i^n)$ for $0 \le k \le n$.

\begin{lemma} \label{Lem: R(n alpha i) tau w[1,n]} Let $w[1,n]$ be the longest element of $ S_n$.
Then we have
$$R(n\alpha_i) \tau_{w[1,n]}R(n\alpha_i) = R(n\alpha_i).$$
\end{lemma}
\begin{proof}
We have
$$ \tau_{w[1,n]}=\tau_{w[1,n-1]}\tau_{n-1}\tau_{n-2}\cdots \tau_{1}.$$
Hence by induction, it is enough to show that for $1\le a \le n-1$
\begin{align} \label{eqn: ind w[1,n]}
\tau_{w[1,n-1]}\tau_{n-1}\tau_{n-2}\cdots \tau_{a+1}
\in R(n\alpha_i)\tau_{w[1,n-1]}\tau_{n-1}\tau_{n-2}\cdots \tau_{a}R(n\alpha_i).
\end{align}
Note that
\begin{align*}
x_n \tau_{w[1,n-1]}\tau_{n-1}\tau_{n-2}\cdots \tau_{a}
& = \pm \tau_{w[1,n-1]}x_n\tau_{n-1}\tau_{n-2}\cdots \tau_{a} \\
& = \pm \tau_{w[1,n-1]}(\pm \tau_{n-1}x_{n-1}+1)\tau_{n-2}\cdots \tau_{a} \\
& = \pm \tau_{w[1,n-1]}\tau_{n-1}x_{n-1}\tau_{n-2}\cdots \tau_{a}\\
& \qquad \qquad  \vdots \\
& = \pm \tau_{w[1,n-1]}\tau_{n-1}\tau_{n-2}\cdots \tau_{a+1}x_{a+1}\tau_{a} \\
& = \pm \tau_{w[1,n-1]}\tau_{n-1}\tau_{n-2}\cdots \tau_{a+1}(\pm \tau_{a}x_a +1).
\end{align*}
Thus we have $\eqref{eqn: ind w[1,n]}$ and our assertion follows.
\end{proof}

\section{The superfunctors $E_i$, $F_i$ and $\overline{F}_i$}

In this section, we define the superfunctors $E_i$, $F_i$ and
$\overline{F}_i$ on $\Mod(R(\beta))$ and investigate their relations
among themselves. In the later sections, these relations  will
play an important role in proving
the categorification theorem for cyclotomic quiver Hecke
superalgebras.

{\em Recall that an $R(n)$-module means a $\Z$-graded $R(n)$-module.}

\subsection{Definitions of simple root superfunctors}

Let
\begin{align*}
& E_i\cl \Mod(R(\beta+\alpha_i)) \to \Mod(R(\beta)),\\
& F_i\cl \Mod(R(\beta)) \to \Mod(R(\beta+\alpha_i))
\end{align*}
be the superfunctors given by
\begin{align*}
 &E_i(N)=e(\beta,i)N \simeq e(\beta,i)R(\beta+\alpha_i)\otimes_{R(\beta+\alpha_i)}N \\
 &\hs{18ex}\simeq \Hom_{R(\beta+\alpha_i)}(R(\beta+\alpha_i)e(\beta,i),N), \\
 & F_i(M)=R(\beta+\alpha_i)e(\beta,i)\otimes_{R(\beta)}M
\end{align*}
for $N \in \Mod(R(\beta+\alpha_i))$ and $M \in \Mod(R(\beta))$.

Then $(F_i,E_i)$ is an adjoint pair; i.e.,
$\Hom_{R(\beta+\al_i)}(F_iM,N)\simeq\Hom_{R(\beta)}(M,E_iN)$.
 Let $n=|\beta|$. There are natural transformations:
\eqn
&& x_{E_i}\cl E_i \to \Pi_i q^{-2}_iE_i,\hs{20ex} x_{F_i}\cl F_i \to \Pi_i q^{-2}_iF_i,  \\
&&\tau_{E_{ij}}\cl E_iE_j \to \Pi^{\pa(i)\pa(j)} q^{(\alpha_i|\alpha_j)} E_jE_i,
\hs{6ex} \tau_{F_{ij}}\cl F_iF_j \to \Pi^{\pa(i)\pa(j)} q^{(\alpha_i|\alpha_j)} F_jF_i\\
&&\hs{40ex}\text{(recall that $q_i\seteq q^{(\al_i|\al_i)/2}$ and $\Pi_i\seteq\Pi^{\pa(i)}$)}
\eneqn
induced by
\bna
\item the left multiplication by $x_{n+1}$ on $e(\beta,i)N$ for $N \in \Mod(R(\beta+\alpha_i))$,
\item the right multiplication by $x_{n+1}$ on the kernel $R(\beta+\alpha_i)e(\beta,i)$ of the functor
$F_i$,
\item the left multiplication by $\tau_{n+1}$ on $e(\beta,i,j)N$
for $N \in \Mod(R(\beta+\alpha_i+\alpha_j))$,
\item the right multiplication by $\tau_{n+1}$ on the kernel $R(\beta+\alpha_i+\alpha_j)e(\beta,j,i)$
of the functor $F_iF_j$.
\end{enumerate}
By the adjunction, $\tau_{E_{ij}}$ induces  a natural
transformation
$$F_jE_i \to \Pi^{\pa(i)\pa(j)}q^{(\alpha_i|\alpha_j)} E_iF_j.$$
(See Theorem~\ref{Thm: Comm E_i F_j} for more detail.)

Let $\xi_n\cl R(n) \to R(n+1)$ be the algebra homomorphism given by
$$ \xi_n(x_k) = x_{k+1}, \quad \xi_n(\tau_\ell)=\tau_{\ell+1}, \quad
   \xi_n(e(\nu))= \sum_{i \in I} e(i,\nu) $$
for all $1 \le k \le n$, $1 \le \ell <n$ and $\nu \in I^n$.
We denote by $R^1(n)$ the image of $\xi_n$.

For each $i \in I$ and $\beta \in \mathtt{Q}^+$, let
$\overline{\F}_{i,\beta}\seteq R(\beta+\alpha_i)v(i,\beta)$ be the
$R(\beta+\alpha_i)$-supermodule
generated by
$v(i,\beta)$ of $\Z\times\Z_2$-degree $(0,0)$
with the defining relation
$e(i,\beta)v(i,\beta)=v(i,\beta)$. The supermodule
$\overline{\F}_{i,\beta}$ has an
$(R(\beta+\alpha_i),R(\beta))$-superbimodule structure whose right
$R(\beta)$-action is given by
$$ a v(i,\beta) \cdot b = a \xi_{n}(b) v(i,\beta) \quad \text{ for } a \in R(\beta+\alpha_i) \text{ and }
  b \in R(\beta).$$
In a similar way, we define  the $(R(n+1),R(n))$-superbimodule
structure on $R(n+1)v(1,n)$  by
\begin{align*}
a v(1,n) \cdot b= a \xi_{n}(b) v(1,n) \quad \text{for $a \in R(n+1)$ and $b \in R(n)$.}
\end{align*}
Hence $$R(n+1)v(1,n) \simeq \soplus_{i \in I,\;|\beta|=n}
R(\beta+\al_i)v(i,\beta).$$

Now, for each $i \in I$, we define the superfunctor
$$ \overline{F}_i\cl \Mod(R(\beta)) \to \Mod(R(\beta+\alpha_i)) \text{ by } N \mapsto
\overline{\F}_{i,\beta} \otimes_{R(\beta)}N.$$

\subsection{Functorial relations} We shall investigate the commutation relations for the superfunctors $E_i$,
$F_i$ and $\ol{F}_i$ $(i \in I)$.

\begin{proposition} \label{Prop: twist by tau n}
The homomorphism of $(R(n),R(n-1))$-superbimodules
$$ \tilde{\rho}\cl R(n)e(n-1,j)\tens_{R(n-1)}q^{-(\alpha_i|\alpha_j)}\Pi^{\pa(i)\pa(j)}e(n-1,i) R(n)
 \longrightarrow e(n,i)R(n+1)e(n,j)$$
given by
$$ x \otimes \pi^{\pa(i)\pa(j)} y \longmapsto x \tau_n y \quad (x \in R(n)e(n-1,j), \ y \in e(n-1,i) R(n))$$
induces an isomorphism of $(R(n),R(n))$-superbimodules
\eq \label{Eq: E_iF_i}
 &&\ba{l}\hs{1ex}\rho\cl  R(n) e(n-1,j) \tens_{R(n-1)}q^{-(\alpha_i|\alpha_j)}\Pi^{\pa(i)\pa(j)}e(n-1,i) R(n)
 \oplus e(n,i)R(n,1)e(n,j) \\[1ex]
\hs{35ex}
 \overset{\sim}{\longrightarrow} e(n,i)R(n+1)e(n,j).
 \ea
\eneq
\end{proposition}

\begin{proof} The homomorphism $\tilde{\rho}$ is well-defined since we have
$$a \ e(n-1,j) \tau_n e(n-1,i) = e(n-1,j) \tau_n e(n-1,i) \; \phi^{\pa(i)\pa(j)}(a)
\ \text{for any $a \in R(n-1)$,}$$
where $\phi$ is the parity involution defined in \S\;\ref{sec:superalgebra}.
Thus it induces a homomorphism
$$ \rho'\cl  R(n) e(n-1,j) \tens_{R(n-1)}q^{-(\alpha_i|\alpha_j)}\Pi^{\pa(i)\pa(j)}e(n-1,i) R(n)
\to \dfrac{ e(n,i)R(n+1)e(n,j)}{e(n,i)R(n,1)e(n,j)}. $$ Thus it is
enough to show that $\rho'$ is an isomorphism. Since $$R(n)=
\bigoplus_{a=1}^n \tau_a \cdots \tau_{n-1} \k^I[x_n] \ot_\k
R(n-1),$$ we have \eqn
&&R(n) e(n-1,j) \ot_{R(n-1)} e(n-1,i) R(n) \\
&&\hs{10ex} = \left(\bigoplus_{a=1}^n \tau_a \cdots \tau_{n-1}
\k [x_ne(j)] \ot_\k R(n-1)\right)
\ot_{R(n-1)} e(n-1,i) R(n)\\
&&\hs{10ex} \simeq \bigoplus_{a=1}^n \tau_a \cdots \tau_{n-1} \k [x_ne(j)] \ot_\k e(n-1,i) R(n).
\eneqn
On the other hand,
\eqn
\dfrac{ e(n,i)R(n+1)e(n,j)}{e(n,i)R(n,1)e(n,j)}
&\simeq &\dfrac{\bigoplus_{a=1}^{n+1}e(n,i)\tau_a \cdots \tau_n \k [x_{n+1}e(j)] \ot_\k e(n-1,i)R(n)}
{e(n,i)\k [x_{n+1}e(j)] \ot_\k R(n)}\\
&\simeq&\bigoplus_{a=1}^{n}e(n,i)\tau_a \cdots \tau_n \k
[x_{n+1}e(j)] \ot_\k e(n-1,i)R(n). \eneqn
By $\eqref{Eqn: tau f}$,
for $f \in \k [x_{n}e(j)]$, $y \in e(n-1,i)R(n)$ and $1 \le a \le
n$, we have
\begin{align*}
\tau_{a}\cdots\tau_{n-1}f\tau_n y &
= \tau_{a}\cdots\tau_{n-1} \bl\tau_n (\oS_n f)+(f{}^{\partial_n })e_{n,n+1}\br y \\
&\equiv  \tau_{a}\cdots\tau_{n}(\oS_n f) y \quad \mod e(n,i)R(n,1)e(n,j).
\end{align*}
Hence $\rho'$ is right $R(n)$-linear and
$\rho'(\tau_{a}\cdots\tau_{n-1}f)=\tau_{a}\cdots\tau_{n}(\oS_n f)$.
Since $f\mapsto \oS_n f$ induces an isomorphism $\k[x_ne(j)] \simeq \k[x_{n+1}e(j)]$, our assertion follows.
\end{proof}

\begin{theorem} \label{Thm: Comm E_i F_j}
There exist natural isomorphisms
\begin{align*}
E_iF_j \overset{\sim}{\to}
\begin{cases}
 q^{-(\alpha_i|\alpha_j)}F_j \Pi^{\pa(i)\pa(j)} E_i & \text{ if } i \neq j, \\
 q^{-(\alpha_i|\alpha_i)}F_i \Pi_i E_i \oplus \k[t_i] \ot {\rm Id}  & \text{ if } i = j,
\end{cases}
\end{align*}
where $t_i$ is an indeterminate of $(\Z \times \Z_2)$-degree $\bl(\alpha_i|\alpha_i),\pa(i)\br$ and
 $$\k[t_i] \ot {\rm Id} \cl\Mod(R(\beta)) \to \Mod(R(\beta))$$ is the
functor defined by $M \mapsto \k[t_i] \ot M$.
\end{theorem}

\begin{proof}
Note that the kernels of $F_jE_i$ and $E_iF_j$ on $\Mod(R(\beta))$ are given by
$$
\begin{aligned}
& R(\beta-\alpha_i+\alpha_j)
e(\beta-\alpha_i,j)\kern-2ex\tens_{R(\beta-\alpha_i)}\kern-2ex
e(\beta-\alpha_i,\alpha_i)R(\beta) \quad \text{ and } \\
& e(\beta+\alpha_j-\alpha_i,i)R(\beta+\alpha_j)e(\beta,j),
\end{aligned}
$$
respectively. Since
\begin{align*}
& R(n)e(n-1,j)\otimes_{R(n-1)}e(n-1,i)R(n)e(\beta) \\
& \qquad \qquad \qquad  \simeq R(\beta-\alpha_i+\alpha_j) e(\beta-\alpha_i,j) \tens_{R(\beta-\alpha_i)}
e(\beta-\alpha_i,\alpha_i)R(\beta)\quad\text{and} \\
& e(n,i)R(n+1)e(n,j)e(\beta,j)=e(\beta+\alpha_j-\alpha_i,i)R(\beta+\alpha_j)e(\beta,j),
\end{align*}
our assertion is obtained by applying the exact functor $ \ \bullet \ e(\beta,j)$ on $\eqref{Eq: E_iF_i}$.
\end{proof}

\begin{remark} For an $R(\beta)$-module $M$, the $R(\beta)$-module structure on $\k[t_i] \ot M$ is given by
$$ a (t^k_i \ot s) = t^k_i \ot \phi^{k\pa(i)}(a)s \quad \text{ for } a \in R(\beta), \ s \in M.$$
Thus we have an isomorphism of functors
$$\k[t_i] \ot {\rm Id} \simeq \bigoplus_{k \ge 0} (q^{(\alpha_i|\alpha_i)} \Pi_i)^k.$$
\end{remark}

\begin{remark}\label{rem:tx} The morphism $\k[t_i] \ot {\rm Id} \to E_iF_i$ intertwines
$$t_i\cl \k[t_i] \ot {\rm Id} \to q_i^{-2}\Pi_i \k[t_i] \ot {\rm Id} \quad
\text{ and } \quad E_i x_{F_i}\cl E_iF_i \to  q_i^{-2}E_i\Pi_i
F_i.$$ Furthermore, the morphism $E_iF_i \to q_i^{-2}F_i \Pi_i E_i$
intertwines
$$ E_i x_{F_i} \quad \text{ and } \quad q_i^{-2}x_{F_i} \Pi_i E_i.$$
\end{remark}

\begin{proposition} \label{Pro: R(n)R1(n)}
There exists an injective  $(R(n), R(n))$-bimodule
homomorphism
$$ \Phi\cl R(n)v(1,n-1) \otimes_{R(n-1)} R(n) \to R(n+1)v(1,n) $$
given by
$$x \, v(1,n-1) \otimes y \longmapsto x\,\xi_n(y)v(1,n).$$
Moreover, its image $R(n)R^1(n)$ has a decomposition
$$R(n)R^1(n) = \bigoplus_{a=2}^{n+1}R(n,1)\tau_n \cdots \tau_a =
                  \bigoplus_{a=0}^{n-1}\tau_a \cdots \tau_1 R(1,n). $$
\end{proposition}

\begin{proof} Since the proof is similar to that of \cite[Proposition 3.7]{KK11}, we omit it.
\end{proof}

By a direct calculation, for $1 \le k \le n$, $1 \le \ell \le n-1$
and $\nu \in I^\beta$, we can easily see that
\begin{align*}
& x_k e(\nu,i) \tau_n \cdots \tau_1 e(i,\nu)\equiv
    (-1)^{\pa(i)\pa(\beta)\pa(\nu_{k})}\tau_n \cdots \tau_1 x_{k+1}e(i,\nu), \\
&\tau_\ell e(\nu,i) \tau_n \cdots \tau_1 e(i,\nu) \equiv
 (-1)^{\pa(i)\pa(\beta)\pa(\nu_\ell)\pa(\nu_{\ell+1})}\tau_n \cdots \tau_1 \tau_{\ell+1}e(i,\nu),
    \\
& x_{n+1}e(\nu,i) \tau_n \cdots \tau_1 e(i,\nu)\equiv
    (-1)^{\pa(i)\pa(\beta)}\tau_n \cdots \tau_1 x_1 e(i,\nu) \quad {\rm mod} \ R(n)R^1(n).
\end{align*}
Note that
\begin{align*}
& \pa(\tau_n \cdots \tau_1 e(i,\nu)) = \pa(i)\pa(\beta),  \quad
\pa(x_k e( \nu,i))= \pa(\nu_k), \\
& \pa(\tau_\ell e(\nu,i)) = \pa(\nu_{\ell})\pa(\nu_{\ell+1}), \quad
 \pa(x_{n+1}e(\nu,i)) = \pa(i).
\end{align*}
Hence
\eq
&&
\ba{rl}a \tau_n \cdots \tau_1 e(i,\beta) \equiv& \tau_n \cdots \tau_1 e(i,\beta)
\phi^{\pa(i)\pa(\beta)}(\xi_n(a)),\\[1ex]
x_{n+1}e(\beta,i) \tau_n \cdots \tau_1 e(i,\beta)\equiv&
    (-1)^{\pa(i)\pa(\beta)}\tau_n \cdots \tau_1 x_1 e(i,\beta)\\[1ex]
    &\hs{5ex} \text{$\mod R(n)R^1(n)$
\quad for any $a \in R(\beta)$.}
\ea
 \label{Rmk: superlagebra iso of tensors}
\eneq

By Proposition \ref{Pro: R(n)R1(n)}, there exists a right $R(n)$-linear map
$$\varphi_1\cl R(n+1) v(1,n) \to R(n) \otimes
\k^I[x_{n+1}]$$ given by \begin{equation} \label{Eq: varphi 1}
\begin{aligned}
R(n+1)v(1,n) \to & {\rm Coker}(\Phi)
\cong \dfrac{\bigoplus_{a=1}^{n+1}R(n,1)\tau_n\cdots \tau_a}{\bigoplus_{a=2}^{n+1}R(n,1)\tau_n\cdots \tau_a}
\overset{\sim}{\gets} R(n,1)\tau_n\cdots \tau_1  \\
&\overset{\sim}{\gets} R(n,1) \cong R(n) \ot \k^I[x_{n+1}]   \cong  R(n) \ot \k^I[t].
\end{aligned}
\end{equation}
Similarly, there is another map $\varphi_2\cl R(n+1)v(1,n) \to \k^I[x_1]\otimes R(n)$ given by
\begin{equation} \label{Eq: varphi 2}
\begin{aligned}
R(n+1)v(1,n) \to & {\rm Coker}(\Phi)
\cong \dfrac{\bigoplus_{a=0}^{n}\tau_a\cdots \tau_1 R(1,n)}{\bigoplus_{a=0}^{n-1}\tau_a\cdots \tau_1R(1,n)}
\overset{\sim}{\gets} \tau_n\cdots \tau_1 R(1,n)  \\
& \overset{\sim}{\gets} R(1,n) \cong \k^I[x_{1}] \otimes R(n) \cong \k^I[t] \otimes R(n).
\end{aligned}
\end{equation}
By restricting $\Phi$ to
\begin{align*}
R(\beta+\alpha_j-\alpha_i)v(j,\beta-\alpha_i)\ot_{R(\beta-\alpha_i)} e(\beta-\alpha_i,i)R(\beta),
\end{align*}
which is the kernel of $\overline{F}_j E_i$ on $\Mod(R(\beta))$,
\eqref{Eq: varphi 1} and \eqref{Eq: varphi 2} can be rewritten as
\begin{align*}
& e(\beta+\alpha_j-\alpha_i,i)R(\beta+\alpha_j)v(j,\beta)  \\
& \hs{10ex}\To[\vphi_1] {\rm Coker}(\Phi)
\cong \dfrac{\bigoplus_{a=1}^{n+1}R(\beta+\alpha_j-\alpha_i,i)\tau_n\cdots \tau_a e(j,\beta)}
{\bigoplus_{a=2}^{n+1}R(\beta+\alpha_j-\alpha_i,i)\tau_n\cdots \tau_ae(j,\beta)} \\
& \hs{10ex}\isofrom \delta_{i,j}R(\beta,i)\tau_n\cdots \tau_1
 \isofrom \delta_{i,j}R(\beta,i) \\
& \hs{10ex} \simeq \delta_{i,j} R(\beta) \ot \k [x_{n+1}e(i)]   \simeq  \delta_{i,j} R(\beta) \ot \k[t_i]
\end{align*}
and
\begin{align*}
& e(\beta+\alpha_j-\alpha_i,i)R(\beta+\alpha_j)v(j,\beta) \\
& \hs{10ex}\To[\varphi_2] {\rm Coker}(\Phi)
\cong \dfrac{\bigoplus_{a=0}^{n}e(\beta+\alpha_j-\alpha_i,i)\tau_a\cdots \tau_1 R(j,\beta)}
{\bigoplus_{a=0}^{n-1}e(\beta+\alpha_j-\alpha_i,i)\tau_a\cdots \tau_1R(j,\beta)} \\
& \hs{10ex}\isofrom\delta_{i,j}\tau_n\cdots \tau_1 R(i,\beta)
 \isofrom \delta_{i,j}R(i,\beta) \\
& \hs{10ex}\simeq \delta_{i,j} \k [x_{1}e(i)] \otimes R(\beta) \simeq \delta_{i,j} \k[t_i] \otimes R(\beta).
\end{align*}

Therefore by \eqref{Rmk: superlagebra iso of tensors}, $\varphi_1$
and $\varphi_2$ coincide and we obtain:

\begin{theorem} \label{Thm: Comm E_i bar F_j }
\bnum
\item There is a natural isomorphism
 $$ \overline{F}_j E_i \overset{\sim}{\to} E_i \overline{F}_j \quad \text{ for } i \neq j.$$
\item There is an exact sequence in $\Mod(R(\beta))$
 $$ 0\to \overline{F}_i E_i M \to E_i \overline{F}_i M \to
 \Pi^{\pa(i)\pa(\beta)}q^{-(\alpha_i | \beta )} \k[t_i] \ot M  \to 0,$$
 which is functorial in $M\in\Mod(R(\beta))$. Here $t_i$ is an indeterminate of $(\Z \times \Z_2)$-degree
 $((\alpha_i|\alpha_i),\pa(i))$.
\end{enumerate}
\end{theorem}

\section{Crystal structure and strong perfect bases}

In this section, we will show that we can choose a set of
irreducible $R(\beta)$-modules ($\beta \in \mathtt{Q}^+$) which
gives a strong perfect basis of $[\Rep(R(\beta))]$. We will also
show that irreducible modules are always isomorphic to their parity
changes.
 To prove these facts, we
need to employ the categorical crystal theories which are developed
in \cite{KL1,K05,LV09}. However, in the quiver Hecke superalgebra
case, we can use the arguments similar to those given in
\cite[Section 3.2]{KL1} with slight modifications. Therefore we only
state the required results  (I1)--(I4) below, and we will focus on
the basic properties of perfect bases.

\medskip
In this section, we assume \eqref{cond:k0}; i.e., $\k_0$ is a field
and the $\k_i$'s are finite-dimensional over $\k_0$.

\medskip
For $M \in \Rep(R(\beta))$ and $i\in I$, define
\begin{equation} \label{eqn: ceystal operators}
\begin{aligned}
&\Delta_{i^k} M = e( \beta- k\alpha_i,i^k) M \in \Rep(R(\beta-k\alpha_i,k\alpha_i)), \\
& \varepsilon_i(M) = \max\{ k \ge 0 \mid \Delta_{i^k} M \ne 0 \}, \\
& E_i(M) = e(\beta-\alpha_i,i)M \in \Rep(R(\beta-\alpha_i)), \\
& F_i'(M) = \ind_{\beta,\alpha_i}(M \bt L(i)) \in \Rep(R(\beta+\alpha_i)), \\
& \tilde{e}_i(M) = \soc(E_i(M)) \in \Rep(R(\beta-\alpha_i)), \\
& \tilde{f}_i(M) = \hd(F_i' M) \in \Rep(R(\beta+\alpha_i)).
\end{aligned}
\end{equation}
Here, $\soc(M)$ means the {\it socle} of $M$, the largest semisimple subobject of $M$ and
$\hd(M)$ means the {\it head} of $M$, the largest semisimple quotient of $M$.
We set $\eps_i(M)=-\infty$ for $M=0$.

Then we have the following statements.
\be[{({I}1)}]
\item If $M$ is an irreducible $R(\beta)$-module and $\eps_i(M)>0$, then
$\te_iM$ is irreducible.
\item If $M$ is an irreducible $R(\beta)$-module, then
$\tilde{f}_iM$ is irreducible.
\item Let $M$ be an irreducible $R(\beta)$-module
and $\ve=\ve_i(M)$. Then $\Delta_{i^\ve}M$ is isomorphic to $N \bt
L(i^\ve)$ for some irreducible $R(\beta-\ve\alpha_i)$-module $N $
with $\ve_i(N)=0$. Moreover, $N\simeq\te_i^\eps(M)$.\label{I1}
\item Let $M$ be an irreducible $R(\beta)$-module. Then $\tilde{e}_i\tilde{f}_iM \simeq M$.
If $\ve_i(M)>0$, then we have $\tilde{f}_i\tilde{e}_iM \simeq
M$.\label{I4} \ee

\begin{lemma}
Let $z$ be an element of $[\Rep(R(\beta))]$ such that $[E_i^k]z=0$.
Here $[E_i^k]$ is the map from $[\Rep(R(\beta))]$ to
$[\Rep(R(\beta-k\al_i))]$ induced by the exact functor $E_i^k$. Then
$z$ is a linear combination of $[M]$'s, where $M$ are irreducible
$R(\beta)$-modules with $\eps_i(M)<k$.
\end{lemma}
\begin{proof}
Write $z=\sum a_M[M]$,
where $a_M\in \Z$ and $M$ ranges over the set of isomorphic classes of irreducible $R(\beta)$-modules.
Let $\ell$ be the largest $\eps_i(M)$ with $a_M\not=0$.
 Then by (I\ref{I1}),
$[E_i^\ell]z=\dim
L(i^\ell)\sum\limits_{\eps_i(M)=\ell}a_M[\te_i^\ell M]$. Hence if
$\ell\ge k$, then $[E_i^\ell]z=0$, which is a contradiction. Hence
we obtain the desired result.
\end{proof}

\begin{proposition} \label{Prop: strong perfect of irr modules}
Let $M$ be an irreducible module in $\Rep(R(\beta))$. Assume that
$\ve\seteq\ve_i(M)>0$. Then we have
\begin{equation*}
[E_i M] = \pi_i^{1-\ve}q^{1-\ve}_i[\ve]^\pi_i [\tilde{e}_iM] + \sum_{k} [N_k],
\end{equation*}
where $N_k$ are irreducible modules with $\ve_i(N_k) < \ve_i(\tilde{e}_iM)=\ve-1$.
\end{proposition}

\begin{proof}
By (I\ref{I1}), we have $$\Delta_{i^\ve}M \simeq \te_i^\eps M \bt L(i^\ve). $$
Similarly, we have 
$$\Delta_{i^{\eps-1}}\te_i M\simeq  \te_i^\eps M \bt L(i^{\eps-1}).$$
On the other hand, \eqref{Eqn: Decom L(i^n) into L(i^n-1) L(1)}
implies that
$$[L(i^\ve)]= \pi_i^{1-\ve}q^{1-\ve}_i[\ve]^\pi_i [L(i^{\ve-1})]$$
as an element of  $[\Rep(R(0))]$. Thus we obtain
$$[E_i^{\eps-1}]\bl[E_iM]-\pi_i^{1-\ve}q^{1-\ve}_i[\ve]^\pi_i[\te_iM]\br=0.$$
Hence the desired result follows from the preceding lemma.
\end{proof}

By a similar argument to the one in \cite[Corollary 3.19]{KL1}, we
have the following lemma.
\begin{lemma}\label{lem:abs_irr}
For any irreducible $R(\beta)$-module $M$,
$$ \k_0 \simeq {\rm End}_{R(\beta)}(M).$$
\end{lemma}

Now we are ready to prove the following fundamental result on irreducible modules over
quiver Hecke superalgebras.
\begin{theorem} \label{Thm: Pi-invariant}
 For any irreducible $R(\beta)$-module $M$, we have
$$\Pi M\simeq M.$$
In particular, $\Pi$ acts as the identity on $[\Rep(R(\beta))]$ and
$[\Proj(R(\beta))]$.
\end{theorem}

\begin{proof}
We shall prove it by induction on $|\beta|$. If $|\beta|>0$, there exists $i \in I$ such that
$\ve_i(M)>0$.
Since the endofunctor $\Pi$ commutes with the functor $E_i$.
$$\tilde{e}_i (\Pi M) \simeq (\Pi \tilde{e}_i  M).$$
By induction hypothesis, $\Pi \tilde{e}_i (M) \simeq  \tilde{e}_i  M$. Hence we obtain
$$ \tilde{e}_i M \simeq \tilde{e}_i \Pi M.$$
Then $\Pi M\simeq M$ follows from (I\ref{I4}).

By \eqref{property of R(beta)-mod}, our assertion also holds for $[\Proj(R(\beta))]$.
\end{proof}

Together with Proposition~\ref{Prop: strong perfect of irr modules},
we obtain the following corollary.

\begin{corollary}
For any irreducible $R(\beta)$-module $M$, we have
\begin{align}\label{Eqn: strong perfect of irr modules}
[E_i M] = q^{1-\ve}_i[\ve]_i [\tilde{e}_iM] + \sum_{k} [N_k],
\end{align}
where $N_k$'s are irreducible modules with $\eps_i(N_k)<\eps_i(M)-1$.
\end{corollary}

Let $\psi\cl R(\beta) \to R(\beta)$ be the involution
given by
\begin{align}\label{def:psi}
\psi(ab)=\psi(b)\psi(a), \quad \psi(e(\nu))=e(\nu),\quad \psi(x_k)=x_k, \quad  \psi(\tau_l)=\tau_l,
\end{align}
for all $a,b \in R(\beta)$.

For any $M \in \Mod(R(\beta))$, we denote by $M^* =
\Hom_{\k_0}(M,\k_0)$ the $\k_0$-dual of $M$ whose left
$R(\beta)$-module structure is induced by the involution
$\psi$: namely, $(af)(s)=f(\psi(a)s)$ for $f\in \Hom_{\k_0}(M,\k_0)$,
$a\in R(\beta)$ and $s\in M$.
 We say that $M$ is {\em self-dual} if $ M^*\simeq M$.

The following lemma tells that $\te_i$ commutes with the duality up
to a grade shift.
\begin{lemma}
For any irreducible $R(\beta)$-module $M$ such that $\eps_i(M)>0$,
we have
\eq
\bl q_i^{1-\eps_i(M)}\te_iM\br^*\simeq q_i^{1-\eps_i(M)}\te_i(M^*).
\eneq
\end{lemma}
\begin{proof} Set $\eps=\eps_i(M)$.
By \eqref{Eqn: strong perfect of irr modules}, we have
$$ [E_i M] = [\ve]_i [q^{1-\ve}_i \tilde{e}_i M] + \sum_{k} [N_k].$$
Here $N_k$'s are irreducible modules with $\eps_i(N_k)<\eps_i(M)-1$.
Since $E_i$ commutes with the duality functor, we have
$$ [E_i (M^*)] = [\ve]_i [\bl q^{1-\ve}_i \tilde{e}_i M\br^*] + \sum_{k} [(N_k)^*].$$
On the other hand, applying \eqref{Eqn: strong perfect of irr modules} to $M^*$, we obtain
$$[E_i (M^*)] = [\ve]_i [q^{1-\ve}_i \tilde{e}_i (M^*)] + \sum_{k'} [N_{k'}']$$
with $\eps_i(N'_{k'})<\eps_i(M)-1$.
Hence we obtain the desired result.
\end{proof}

\begin{proposition} \label{Prop: self-dual}
For any irreducible $R(\beta)$-module $M$, there exists $r \in \Z$ such that $q^{r}M$ is self-dual, that is,
$$(q^{r}M)^* \simeq q^{r}M.$$
\end{proposition}

\begin{proof}
Using induction on $|\beta|$, we shall show that there exists $r \in
\Z$ such that
 $$q^{r}M \ \text{ is self-dual.}$$
 Assume $|\beta|>0$ and take $i \in I$ such that $\eps\seteq \ve_i(M)>0$.

 Then, by the induction hypothesis, there exists $r \in \Z$  such that
 $q^{r}q^{1-\ve}_i \tilde{e}_i M$ is self-dual.
 Then the preceding lemma implies
 \eqn
q^{1-\ve}_i \tilde{e}_i (q^rM)
\simeq \bl q^{1-\ve}_i \tilde{e}_i(q^r M)\br^*
\simeq q^{1-\ve}_i \tilde{e}_i \bl(q^rM)^*\br.
 \eneqn
 Hence by (I\ref{I4}), we get $q^rM\simeq(q^rM)^*$.
\end{proof}

Finally we obtain the following theorem
which shows the existence of strong perfect basis of $[\Rep R(\beta)]$.
\begin{theorem} \label{Thm: choice of strong perfect basis}
For $\beta \in \mathtt{Q}^+$, let ${\rm Irr}_0R(\beta)$
be the set of isomorphism classes of self-dual irreducible $R(\beta)$-modules. Then
$$ \set{ \; [M] }{M \in {\rm Irr}_0R(\beta) } $$
is an $\A$-basis of $[\Rep R(\beta)]$. Moreover, it is a strong
perfect basis; i.e., it satisfies the property  $\eqref{Eqn perfect
basis}$.
\end{theorem}

\begin{proof}
The proof is an immediate consequence of Proposition \ref{Prop: self-dual} and
\eqref{Eqn: strong perfect of irr modules}.
\end{proof}

The following lemma is a categorification of the $q$-boson relation
\eqref{eq: e_i' ev}.
\begin{lemma} \label{Lem: BqG module}
For all $\beta \in \mathtt{Q}^+$ and  $M \in \Rep(R(\beta))$, we have isomorphisms and exact sequences.
\begin{equation} \label{eqn: Bqg structure}
\begin{aligned}
& E_iF_j'M \simeq  \Pi^{\pa(i)\pa(j)} q^{-(\alpha_i|\alpha_j)}
F_j'E_iM \quad \text{ for } i \neq j, \\
& 0 \to M \to E_iF_i'M \to \Pi_i q_i^{-2} F_i'E_iM \to 0.
\end{aligned}
\end{equation}
\end{lemma}

\begin{proof} By
Theorem~\ref{Thm: Comm E_i F_j} and Remark~\ref{rem:tx},
all the columns and rows in the following commutative diagram are exact except the bottom row.
\begin{align*}
\xymatrix{
&0\ar[d]&0\ar[d]&0\ar[d]\\
0\ar[r]&\delta_{ij}q_i^2\Pi_i\k[t_i]\otimes M\ar[r]\ar[d]^{t_i}&
q_i^2\Pi_iE_i F_j M\ar[r]\ar[d]^{x_{F_j}} &
q_i^2\Pi_i\Pi^{\pa(i)\pa(j)} q^{-(\alpha_i|\alpha_j)}F_jE_iM\ar[r]\ar[d]^{x_{F_j}}&0\\
0\ar[r]&\delta_{ij}\k[t_i]\otimes M\ar[r]\ar[d]&
E_i F_j M\ar[r]\ar[d]&
\Pi^{\pa(i)\pa(j)} q^{-(\alpha_i|\alpha_j)}F_jE_iM\ar[r]\ar[d]&0\\
0\ar[r]&\delta_{ij} M\ar[r]\ar[d]&
E_i F'_j M\ar[r]\ar[d]&
\Pi^{\pa(i)\pa(j)} q^{-(\alpha_i|\alpha_j)}F_j'E_iM\ar[r]\ar[d]&0\\
&0&0&0
}
\end{align*}
Hence the bottom row is also exact.
\end{proof}

\section{Cyclotomic quiver Hecke superalgebras}

In this section, we define the cyclotomic quiver Hecke superalgebra
$R^{\Lambda}$ and study its elementary properties.

\subsection{Definition of cyclotomic quotients}
For $\Lambda \in \mathtt{P}^+$ and $i \in I$, we choose a monic polynomial of degree
$\langle h_i,\Lambda \rangle$
\begin{equation} \label{Eq: cylotomic polynomial}
\begin{aligned}
a^{\Lambda}_i(u)= \sum_{k=0}^{\langle h_i,\Lambda \rangle} c_{i;k}u^{\langle h_i,\Lambda \rangle-k} \\
\end{aligned}
\end{equation}
with $c_{i;k} \in \k_{k(\alpha_i|\alpha_i)}$ such that $c_{i,0}=1$ and
$c_{i;k}=0$ if $i \in \Iod$ and $k$ is odd.
Hence $a^{\Lambda}_i(x_1)e(i)$ has the $\Z \times \Z_2$-degree
$$(\langle h_i, \Lambda \rangle(\alpha_i|\alpha_i), \pa(i)\langle h_i, \Lambda \rangle ).$$

For $1 \le k \le n$, define
$$ a^{\Lambda}(x_k) = \sum_{\nu \in I^n} a^{\Lambda}_{\nu_k}(x_k)e(\nu) \in R(n) .$$

\begin{definition} \label{Def: cyclo}
Let $\beta \in \mathtt{Q}^+$ and $\Lambda\in \mathtt{P}^{+}$.
The {\it cyclotomic quiver Hecke superalgebra} $R^{\Lambda}(\beta)$ at $\beta$ is the quotient algebra
$$ R^{\Lambda}(\beta) = \dfrac{R(\beta)}{R(\beta)a^{\Lambda}(x_1)R(\beta)}\,.$$
\end{definition}

\subsection{Structure of cyclotomic quotients}
We shall prove that the cyclotomic quotients are finitely generated
over $\k$. For the definition of $\oS_a$ and $\partial_a$, see
\eqref{def:sk} and \eqref{Eqn: partial}.
\begin{lemma} \label{Lem: shift}
 Assume that $f e(\nu) M =0$ for $M \in \Mod(R(n))$, $f \in \P_n$, $\nu \in I^n$ and $1 \le a <n$ such that
$\nu_a=\nu_{a+1}=i$.
Then we have
\begin{align*}  & (\partial_{a}f)(x_{a}-x_{a+1})^{\pa(i)} e(\nu)M = 0 \quad \text{ and } \\
& (x^2_{a}+x^2_{a+1})^{\pa(i)} (\oS_{a}f) e(\nu)M = (\oS_{a}f) (x^2_{a}+x^2_{a+1})^{\pa(i)} e(\nu)M = 0.
\end{align*}
\end{lemma}

\begin{proof}
By \eqref{Eqn: partial}, we have 
\begin{align*}
& \ (x^{1+\pa(i)}_{ a+1 }-x^{1+\pa(i)}_{a}) \tau_{a}f\tau_a e(\nu) 
= \left((x_{a+1}-x_a)^{\pa(i)}f - \oS_af(x_{a+1}-x_a)^{\pa(i)}\right)\tau_a e(\nu) \\
&=\left( (x_{a+1}-x_a)^{\pa(i)}f\tau_a-(-1)^{\pa(i)}\oS_af \tau_a (x_{a+1}-x_a)^{\pa(i)} \right)e(\nu)\\
& =\left( (x_{a+1}-x_a)^{\pa(i)}f\tau_a -(-1)^{\pa(i)} \tau_a f (x_{a+1}-x_a)^{\pa(i)} +(-1)^{\pa(i)}\partial_a f(x_{a+1}-x_a)^{\pa(i)}
   \right)e(\nu).
\end{align*}
Hence we have $(\partial_af)(x_{a}-x_{a+1})^{\pa(i)}e(\nu)M=0$. It follows that
\begin{align*}
0 &= (x^{1+\pa(i)}_{a}-x^{1+\pa(i)}_{a+1})(\partial_af)(x_{a}-x_{a+1})^{\pa(i)}e(\nu)M  \\
  & = \left((x_{a}-x_{a+1})^{\pa(i)}f(x_{a}-x_{a+1})^{\pa(i)}-\oS_af(x_{a}-x_{a+1})^{2\pa(i)} \right) e(\nu)M.
\end{align*}
Thus $(x_{a}-x_{a+1})^{2\pa(i)}(\oS_af)e(\nu)M=(x^2_{a}+x^2_{a+1})^{\pa(i)}(\oS_af)e(\nu)M =0.$
\end{proof}

\begin{lemma} \label{Lem: finite dimension}
There exists a monic polynomial $g(u)$ with coefficients in $\k$ such that $g(x_a)=0$ in
$R^{\Lambda}(\beta)$ $(1 \le a \le n)$.
\end{lemma}

\begin{proof}
If $a=1$, $g(x_1^2) = \prod_{i \in I} a^\Lambda_i(-x_1)a^\Lambda_i(x_1)$ satisfies the condition.
Hence,
by induction on $a$, it is enough to show the following statement:
\begin{center}
For any monic polynomial $g(u) \in \k[u]$ and $\nu \in I^n$, we can find
\\ a monic polynomial $h(u)\in \k[u]$ such that
\\ $h(x_{a+1}^2)e(\nu)M=0$  for any $R(\beta)$-module $M$ with
$g(x_a^2)M=0$.
\end{center}

(i) Suppose $\nu_a \neq \nu_{a+1}$. In this case, we have
$$
g(x_{a+1}^2) \cQ_{\nu_a,\nu_{a+1}}(x_a,x_{a+1})e(\nu)M
=g(x_{a+1}^2) \tau_a^2e(\nu)M
 =\tau_a g(x_{a}^2) \tau_a e(\nu)M =0. $$

Since $\cQ_{\nu_a,\nu_{a+1}}(x_a,x_{a+1})$ is a monic polynomial in
$x^{1+p(\nu_{a+1})}_{a+1}$ with coefficients in
$\k[x_a^{1+\pa(\nu_{a})}]$, there exists a monic polynomial $h(u)$ such that
\begin{align*}
h(x_{a+1}^2) \in &\k[x_a^{1+\pa(\nu_a)}e(\nu),x_{a+1}^{1+\pa(\nu_{a+1})}e(\nu)]g(x_a^2)+ \\
& \qquad \k[x_a^{1+\pa(\nu_a)}e(\nu),x_{a+1}^{1+\pa(\nu_{a+1})}e(\nu)] g(x_{a+1}^2)
\cQ_{\nu_a,\nu_{a+1}}(x_a,x_{a+1}).
\end{align*}
Then $h(x_{a+1}^2)e(\nu)M =0$.

(ii) Suppose $\nu_a=\nu_{a+1}$. 
 Then Lemma \ref{Lem: shift} implies
$$g(x_{a+1}^2)(x_a^2+x_{a+1}^2)^{\pa(\nu_a)}e(\nu)M=0.$$
Then we can apply the same argument as (i).
\end{proof}

\begin{lemma} \label{lem: integrability 1}
 Let $f \in \P_{n+1} $ be a monic
polynomial of degree $m$ in $x_{n+1}$ whose coefficients are contained in
$\P_n \ot \k^I$.  Set $\ol{R}=e(n,i^{m+1})R(n+m+1)e(n,i^{m+1})$. Then we have
$$\ol{R}f\ol{R}=\ol{R}.$$
\end{lemma}

\begin{proof}
We will prove the following statement by induction on $k$:
\begin{align} \label{eqn: integrabiltity using group algebra}
\partial_{k-1} \cdots \partial_{n+1}fe(n,i^{m+1}) \tau_{w[n+1,k]} \in \ol{R}f\ol{R}
\end{align}
for $n+1 \le k \le n+m+1$. Here $w[n+1,k]$ is the longest element of the subgroup $S_{[n+1,k]}$
generated by $s_a$ ($n+1\le a<k$). (See Remark \ref{Rmk: group algebra}.)
Assuming \eqref{eqn: integrabiltity using group algebra},
by multiplying $\tau_{w[n+1,k]^{-1}w[n+1,k+1]}$ from the right,
we have
$$\partial_{k-1} \cdots \partial_{n+1}fe(n,i^{m+1}) \tau_{w[n+1,k+1]} \in \ol{R}f\ol{R}.$$
By multiplying
$\tau_k$ from the left, we have
\eqn \label{eq: partial step 1}
&&
\tau_k(\partial_{k-1} \cdots \partial_{n+1}f)e(n,i^{m+1})\tau_{w[n+1,k+1]} \\
&&\hs{10ex}=\bl\oS_k(\partial_{k-1} \cdots \partial_{n+1}f)\tau_k+\partial_{k} \cdots \partial_{n+1} f\br
e(n,i^{m+1}) \tau_{w[n+1,k+1]}\\
&&\hs{10ex}=\partial_{k} \cdots \partial_{n+1} f e(n,i^{m+1}) \tau_{w[n+1,k+1]}\\
&&\hs{15ex}\in \ol{R}f\ol{R}. \eneqn Here we have used the fact that
$\tau_k\tau_{w[n+1,k+1]}=0$.

Thus the induction proceeds and  we obtain \eqref{eqn: integrabiltity using group algebra}
for any $k$.
Since $\partial_{n+m} \cdots \partial_{n+1}f=1$, our assertion follows from
$ \overline{R}\tau_{w[n+1,n+m+1]}\overline{R}=\overline{R}$ in Lemma \ref{Lem: R(n alpha i) tau w[1,n]}.
\end{proof}

\begin{corollary} \label{Cor: integrability}
For $\beta \in \mathtt{Q}^+$ with $|\beta|=n$ and $i \in I$, there exists $m$ such that
$$R^\Lambda(\beta+k\alpha_i)=0 \text{ for any } k \ge m.$$
\end{corollary}

\begin{proof} By Lemma \ref{Lem: finite dimension}, there exists a monic polynomial
$g(u)$ of degree $m$ such that $$g(x_n)R^\Lambda(\beta)=0.$$ Lemma
\ref{lem: integrability 1} implies
$e(n,i^k)R^\Lambda(\beta+k\alpha_i)=0$ for $k> m$. Now our assertion
follows from similar arguments to \cite[Lemma 4.3\;(b)]{KKO11}.
\end{proof}

\section{The superfunctors $E^\Lambda_i$ and $F^\Lambda_i$}
In this section, we define the superfunctors $E_i^{\Lambda}$ and
$F_i^{\Lambda}$ on $\Mod(R^\Lambda(\beta))$ and show that they
induce well-defined exact functors on $\Proj(R^\Lambda(\beta))$ and
$\Rep(R^\Lambda(\beta))$.

For each $i \in I$, we define the superfunctors
\begin{align*}
&E_i^{\Lambda}\cl \Mod(R^{\Lambda}(\beta+\alpha_i)) \to \Mod(R^{\Lambda}(\beta)),\\
&F_i^{\Lambda}\cl \Mod(R^{\Lambda}(\beta)) \to \Mod(R^{\Lambda}(\beta+\alpha_i))
\end{align*}
by
\begin{align*}
&E_i^{\Lambda}(N)=e(\beta,i)N = e(\beta,i)R^{\Lambda}(\beta+\alpha_i) \otimes_{R^{\Lambda}(\beta+\alpha_i)}N, \\
&F_i^{\Lambda}(M)=R^{\Lambda}(\beta+\alpha_i)e(\beta,i)\otimes_{R^{\Lambda}(\beta)}M
\end{align*}
for $M \in \Mod(R^{\Lambda}(\beta))$ and $N\in \Mod(R^{\Lambda}(\beta+\alpha_i))$.

For each $i \in I$, $\beta \in \mathtt{Q}^+$ and $m \in \Z$, let
$${\mathcal K}^m_{i,\beta}\seteq R(\beta+\alpha_i)v(i,\beta)T^m_i$$ be
the $R(\beta+\alpha_i)$-supermodule generated by $v(i,\beta)T^m_i$
with the defining relation
$$e(i,\beta)v(i,\beta)T^m_i=v(i,\beta)T^m_i.$$ We assign to
$v(i,\beta)T^m_i$ the $(\Z \times \Z_2)$-degree $(0,0)$. The
supermodule ${\mathcal K}^m_{i,\beta}$ has an
$(R(\beta+\alpha_i),\k[t_i] \ot R(\beta))$-superbimodule structure
whose right $\k[t_i] \ot R(\beta)$-action is given by
\begin{align} \label{eqn: K_1 action}
&\ba{l}
 a v(i,\beta)T^m_i \cdot b = a \xi_{n}(b) v(i,\beta)T^m_i, \\[.5ex]
 a v(i,\beta)T^m_i \cdot t_i = a \phi_i^m(x_1) v(i,\beta)T^m_i=(-1)^{\pa(i)}a x_1 v(i,\beta)T^m_i
 \ea
\end{align}
for $a \in R(\beta+\alpha_i)$ and $b \in R(\beta)$. Here,
$\phi_i\seteq\phi^{\pa(i)}$ and $\phi$ is the parity involution (see
\S\;\ref{sec:superalgebra}).

{\em In the sequel, we sometimes omit the $\Z$-grading shift functor $q$ when $\Z$-grading can be neglected. }

Set $\Lambda_i\seteq\langle h_i, \Lambda \rangle$. We introduce
$(R(\beta+\alpha_i),R^{\Lambda}(\beta))$-superbimodules
\begin{equation}\label{Def: Kernels}
\begin{aligned}
 F^{\Lambda}&\seteq R^{\Lambda}(\beta+\alpha_i)e(\beta,i), \\
 K_0 &\seteq R(\beta+\alpha_i)e(\beta,i) \otimes_{R(\beta)} R^{\Lambda}(\beta), \\
 K_1 &\seteq {\mathcal K}^{\Lambda_i}_{i,\beta} \otimes_{R(\beta)} \Pi_i^{\Lambda_i+\pa(\beta)}R^{\Lambda}(\beta) \\
  & = R(\beta+\alpha_i)v(i,\beta)T^{\Lambda_i}_i\otimes_{R(\beta)} \Pi_i^{\Lambda_i+\pa(\beta)}R^{\Lambda}(\beta).
\end{aligned}
\end{equation}

For $i \in I$, let $t_i$ be an indeterminate of $\Z \times
\Z_2$-degree $((\alpha_i|\alpha_i),\pa(i))$. Then $\k[t_i]$ is a
superalgebra. The superalgebra $\k[t_i]$ acts on $K_1$ from the
right by the formula given in \eqref{eqn: K_1 action}. Namely,
\begin{align*}
\Bigl( a v(i,\beta)T^{\Lambda_i}_i \ot \pi_i^{\Lambda_i+\pa(\beta)} b \Bigr ) t_i
 = a \phi_i^{\pa(\beta)}(x_1)v(i,\beta)T^{\Lambda_i}_i \ot \pi_i^{\Lambda_i+\pa(\beta)} \phi_i(b)
\end{align*}
for $a \in R(\beta+\alpha_i)$ and $b \in R^{\Lambda}(\beta)$. Here $\pi_i\seteq\pi^{\pa(i)}$
and $\phi_i=\phi^{\pa(i)}$.
On the other hand, $\k[t_i]$ acts on
$R(\beta+\alpha_i)e(\beta,i)$, $F^{\Lambda}$ and $K_0$ by
multiplying $x_{n+1}$ from the right. Thus $K_0$, $F^{\Lambda}$ and $K_1$ have a graded
$\bl R(\beta+\alpha_i),\k[t_i] \ot R^{\Lambda}(\beta)\br$-superbimodule structure.

By  a similar argument to the one given
in \cite[Lemma 4.8, Lemma 4.16]{KK11}, we obtain the following two lemmas:
\begin{lemma} \label{Lem: Proj} \hfill
\bnum
\item
Both $K_1$ and $K_0$ are finitely generated projective right $ \k[t_i] \ot R^{\Lambda}(\beta)$-supermodules.
\item For any monic \ro skew\rf-polynomial  $f(t_i)\in
\P_n[t_i]$,
the right multiplication by $f(t_i)$ on
$K_1$ induces an injective endomorphism of $K_1$. \ee
\end{lemma}

\begin{lemma} \label{Lem: Decom} For $i \in I$ and $\beta \in \mathtt{Q}^+$ with $|\beta|=n$, we have
\bnum
\item $R(\beta+\alpha_i)a^{\Lambda}(x_1)R(\beta+\alpha_i) =\sum_{a=0}^{n}R(\beta+\alpha_i)
a^{\Lambda}(x_1)\tau_1 \cdots \tau_a$,
\item
$R(\beta+\alpha_i)a^{\Lambda}(x_1)R(\beta+\alpha_i)e(\beta,i)$

$=R(\beta+\alpha_i)a^{\Lambda}(x_1)R(\beta)e(\beta,i)
+R(\beta+\alpha_i)a^{\Lambda}(x_1)\tau_1 \cdots
\tau_n e(\beta,i)$.
\end{enumerate}
\end{lemma}

For $1 \le a \le n$, $1 \le \ell < n$ and $\nu \in I^n$,
by applying a similar argument given in \cite[Lemma 4.7]{KKO11}, we have
\begin{equation} \label{Eqn: twisting for P}
\begin{aligned}
x_{a+1}  a^{\Lambda}(x_1)\tau_1 \cdots \tau_n e(\nu,i) & \equiv
(-1)^{\pa(\nu_a)(\Lambda_i+\pa(\nu))\pa(i)} a^{\Lambda}(x_1)\tau_1 \cdots \tau_n  x_ae(\nu,i),    \\
\tau_{\ell+1}  a^{\Lambda}(x_1)\tau_1 \cdots \tau_n e(\nu,i) & \equiv
(-1)^{\pa(\nu_{\ell})\pa(\nu_{\ell+1})(\Lambda_i+\pa(\nu))\pa(i)}
 a^{\Lambda}(x_1)\tau_1 \cdots \tau_n  \tau_\ell  e(\nu,i),  \\
x_1 a^{\Lambda}(x_1)\tau_1 \cdots \tau_n e(\nu,i) & \equiv
(-1)^{\pa(\nu)\pa(i)} a^{\Lambda}(x_1)\tau_1 \cdots \tau_n  x_{n+1} e(\nu,i) \\
& \qquad \qquad \ {\rm mod} \  R(n+1)a^{\Lambda}(x_1)R(n)e(n,i).
\end{aligned}
\end{equation}
Hence we have
\begin{align} \label{eqn: anticommute P}
&\ba{ll}
 a^{\Lambda}(x_1)\tau_1 \cdots \tau_n e(\beta,i)c \equiv&
\phi^{\Lambda_i+\pa(\beta)}_i(\xi_n(c)) a^{\Lambda}(x_1)\tau_1 \cdots \tau_n e(\beta,i)\\[1ex]
 &\mod  R(n+1)a^{\Lambda}(x_1)R(\beta)e(\beta,i)
\ea
\end{align}
for any $\beta \in \mathtt{Q}^+$ with $|\beta|=n$ and $c \in R(n)$.

Let
$P\cl K_1 \to K_0$ be the homomorphism defined by
\begin{align} \label{eqn: The map P}
x v(i,\beta) T^{\Lambda_i}_i \ot \pi^{\Lambda_i+\pa(\beta)}_i y
\longmapsto x a^{\Lambda}(x_1)\tau_1 \cdots \tau_n e(\beta,i)\otimes y
\end{align}
for $x \in R(\beta+\alpha_i)$ and $y \in R^\Lambda(\beta)$.
Then, by $\eqref{Eqn: twisting for P}$ and $\eqref{eqn: anticommute P}$, $P$
becomes an $(R(\beta+\alpha_i),\k[t_i]\ot R(\beta))$-superbimodule homomorphism.

Let ${\rm pr}\cl K_0 \to F^\Lambda$ be the canonical projection.
Using Lemma \ref{Lem: Decom}, one can easily see  that
\begin{align} \label{Eq: tilde P}
{\rm Im}(P) = {\rm Ker} ({\rm pr})
=\dfrac{R(\beta+\alpha_i) a^{\Lambda}(x_1)
R(\beta+\alpha_i)e(\beta,i)}{R(\beta+\alpha_i)a^{\Lambda}(x_1)R(\beta)e(\beta,i)}
\subset K_0.
\end{align}

Hence we obtain an exact sequence of $\bl
R(\beta+\alpha_i),\k[t_i]\ot R(\beta)\br$-superbimodules
$$ K_1  \overset{P}{\longrightarrow} K_0 \overset{{\rm pr}}{\longrightarrow} F^{\Lambda} \longrightarrow 0.$$

We will show that $P$ is actually injective by constructing an
$(R(\beta+\alpha_i),\k[t_i] \ot R(\beta))$ -bilinear homomorphism
$Q$ such that $Q \circ P$ is injective.

For $1 \le a \le n$, we define the elements $\varphi_a$ and $g_a$ of $R(\beta+\alpha_i)$ by
\begin{equation} \label{Eq: sigma_a}
\begin{aligned}
\varphi_a = \sum_{ \substack{\nu \in I^{\beta+\alpha_i},\\ \nu_a \neq \nu_{a+1}}} \tau_a e(\nu)
& + \sum_{ \substack{\nu \in I^{\beta+\alpha_i},\\ \nu_a = \nu_{a+1} \in \Iev}}
 (1- (x_{a+1}-x_a)\tau_a)e(\nu) \\
& + \sum_{ \substack{\nu \in I^{\beta+\alpha_i},\\ \nu_a = \nu_{a+1} \in \Iod}}
 ((x_{a+1}-x_a)- (x_{a+1}^2-x_a^2)\tau_a)e(\nu)
\end{aligned}
\end{equation}
and
\begin{equation} \label{Eq: g_a}
\begin{aligned}
g_a =\sum_{ \substack{\nu \in I^{\beta+\alpha_i},\\ \nu_a \neq \nu_{a+1}}} \tau_a e(\nu)
& + \sum_{ \substack{\nu \in I^{\beta+\alpha_i},\\ \nu_a = \nu_{a+1} \in \Iev}}
 (x_{a+1}-x_a)\bl1- (x_{a+1}-x_a)\tau_a\br e(\nu) \\
& + \sum_{ \substack{\nu \in I^{\beta+\alpha_i},\\ \nu_a = \nu_{a+1} \in \Iod}}(x_{a+1}^2-x_a^2)
 \bl(x_{a+1}-x_a)- (x_{a+1}^2-x_a^2)\tau_a\br e(\nu).
\end{aligned}
\end{equation}
The elements $\varphi_a$'s are called {\it intertwiners}, and
$g_a$'s are their variants.

Note that if $\nu_a=\nu_{a+1} \in \Iev$,
\begin{equation} \label{Eq: reverse varphi_a ev}
\begin{aligned}
\varphi_a e(\nu) & = (x_a \tau_a - \tau_a x_a)e(\nu) = (\tau_a x_{a+1}-x_{a+1} \tau_a)e(\nu) \\
& =(1 - \tau_a(x_a-x_{a+1}))e(\nu),
\end{aligned}
\end{equation}
and if $\nu_a=\nu_{a+1} \in \Iod$,
\begin{equation} \label{Eq: reverse varphi_a od}
\begin{aligned}
\varphi_a e(\nu) & = (x_a^2 \tau_a - \tau_a x_a^2)e(\nu) = (\tau_a x^2_{a+1}-x^2_{a+1} \tau_a)e(\nu) \\
& =((x_a-x_{a+1}) - \tau_a(x^2_a-x^2_{a+1}))e(\nu).
\end{aligned}
\end{equation}

\begin{lemma} For $1\le a \le n$ and $\nu \in I^{n+1}$, we have
\begin{equation} \label{Eq: Commute sigma_a}
\begin{aligned}
& \varphi_a e(\nu) = e(s_a \nu) \varphi_a,\\
&x_{s_a(b)}\varphi_a e(\nu)=(-1)^{\pa(\nu_a)\pa(\nu_{a+1})\pa(\nu_b)} \varphi_a x_be(\nu)\quad(1 \le b \le n+1), \\
& \tau_b \varphi_{a}e(\nu)
=(-1)^{\pa(\nu_a)\pa(\nu_{a+1})\pa(\nu_b)\pa(\nu_{b+1})}
\varphi_{a}\tau_b e(\nu) \quad
\text{ if } |b-a| >1 , \\
& \tau_a \varphi_{a+1} \varphi_a =  \varphi_{a+1} \varphi_a \tau_{a+1},
\end{aligned}
\end{equation}
and
\begin{equation} \label{Eq: Commute g_a}
\begin{aligned}
& g_a e(\nu) = e(s_a \nu) g_a, \\
& x_{s_a(b)}g_ae(\nu) = (-1)^{\pa(\nu_a)\pa(\nu_{a+1})\pa(\nu_b)}g_a x_be(\nu) \quad (1 \le b \le n+1), \\
& \tau_b g_{a}e(\nu)
=(-1)^{\pa(\nu_a)\pa(\nu_{a+1})\pa(\nu_b)\pa(\nu_{b+1})}  g_{a}
\tau_b e(\nu) \quad
\text{ if } |b-a| >1 , \\
& \tau_a g_{a+1} g_a =  g_{a+1} g_a \tau_{a+1}.
\end{aligned}
\end{equation}
\end{lemma}

\begin{proof}  By the defining relations of quiver Hecke
superalgebras, the third equality can be verified immediately.
If $\nu_a \neq \nu_{a+1}$ or $\nu_a=\nu_{a+1} \in \Iev$, the
first and second equalities  were proved in \cite[Lemma
4.12]{KK11}. We will prove the second equality in \eqref{Eq: Commute
sigma_a} when $\nu_a=\nu_{a+1} \in \Iod$. Let $b=a$. Then
\begin{align*}
x_{a+1} \varphi_a e(\nu) & = x_{a+1}^2-x_{a+1}x_a  - (x_{a+1}^2-x_a^2)(x_{a+1}\tau_a)e(\nu) \\
& = x_{a+1}^2-x_{a+1}x_a  - (x_{a+1}^2-x_a^2)(-\tau_ax_{a}+1)e(\nu),
\end{align*}
and
\begin{align*}
\varphi_a x_a e(\nu) & = x_{a+1}x_a-x_a^2 - (x_{a+1}^2-x_a^2)(\tau_a x_a)e(\nu).
\end{align*}
Therefore we have
$$ x_{a+1} \varphi_a e(\nu) + \varphi_a x_a e(\nu) =0.$$
Similarly, we can prove the equality when $b=a+1$.

Let $S=\tau_a \varphi_{a+1} \varphi_a- \varphi_{a+1}
\varphi_a\tau_{a+1}$. Using the second  equality,  we have 
\begin{align*}
(\tau_a \varphi_{a+1} \varphi_a)x_ae(\nu) &=
(-1)^{\pa(\nu_a)(\pa(\nu_a)\pa(\nu_{a+1})+\pa(\nu_a)\pa(\nu_{a+2})
+\pa(\nu_{a+1})\pa(\nu_{a+2}))}
 x_{a+2}(\tau_a \varphi_{a+1} \varphi_a)e(\nu),\\
(\varphi_{a+1} \varphi_a \tau_{a+1})x_ae(\nu)&=(-1)^{\pa(\nu_a)
(\pa(\nu_{a+1})\pa(\nu_{a+2})+\pa(\nu_a)\pa(\nu_{a+2})+\pa(\nu_a)\pa(\nu_{a+1}))}
 x_{a+2}( \varphi_{a+1} \varphi_a \tau_{a+1})e(\nu),\\
 (\tau_a \varphi_{a+1} \varphi_a)x_{a+1}e(\nu) &= (-1)^{\pa(\nu_{a+1})
 (\pa(\nu_{a})\pa(\nu_{a+1})+\pa(\nu_{a})\pa(\nu_{a+2}))}
 \tau_ax_a\varphi_{a+1}\varphi_ae(\nu)\\
 &= (-1)^{\pa(\nu_{a+1})(\pa(\nu_{a+1})\pa(\nu_{a})+\pa(\nu_{a})\pa(\nu_{a+2})
 +\pa(\nu_{a+1})\pa(\nu_{a+2}))}\\
 &\hs{30ex}(x_{a+1}\tau_a-e_{a,a+1})\varphi_{a+1}\varphi_{a}e(\nu), \\
(\varphi_{a+1} \varphi_a \tau_{a+1})x_{a+1}e(\nu) &= (-1)^{\pa(\nu_{a+1})\pa(\nu_{a+2})}
\varphi_{a+1} \varphi_a(x_{a+2}\tau_{a+1}-e_{a+1,a+2})e(\nu)\\
&= (-1)^{\pa(\nu_{a+1})(\pa(\nu_{a+1})\pa(\nu_{a+2})+\pa(\nu_{a})\pa(\nu_{a+2})+\pa(\nu_{a})\pa(\nu_{a+1}))}
x_{a+1}\varphi_{a+1}\varphi_a\tau_{a+1}e(\nu) \\
& \qquad \qquad - (-1)^{\pa(\nu_{a+1})\pa(\nu_{a+2})}\varphi_{a+1}\varphi_{a}e_{a+1,a+2}e(\nu),\\
(\tau_a \varphi_{a+1} \varphi_a)x_{a+2}e(\nu) &=
(-1)^{\pa(\nu_{a+2})(\pa(\nu_{a})\pa(\nu_{a+1})+\pa(\nu_{a+2})\pa(\nu_{a}))}
\tau_a x_{a+1}\varphi_{a+1}\varphi_a e(\nu) \\
&= (-1)^{\pa(\nu_{a+2})(\pa(\nu_{a})\pa(\nu_{a+1})+\pa(\nu_{a+2})\pa(\nu_{a}))}\\
&\hs{20ex}
((-1)^{\pa(\nu_{a+1})\pa(\nu_{a+2})}x_a\tau_a+e_{a,a+1})\varphi_{a+1}\varphi_{a}e(\nu) \\
& =(-1)^{\pa(\nu_{a+2})(\pa(\nu_{a})\pa(\nu_{a+1})+\pa(\nu_{a+2})\pa(\nu_{a})+\pa(\nu_{a+1})\pa(\nu_{a+2}))}
x_a\tau_a\varphi_{a+1}\varphi_a e(\nu) \\
& \qquad \qquad + (-1)^{\pa(\nu_{a+2})(\pa(\nu_{a})\pa(\nu_{a+1})+\pa(\nu_{a+2})\pa(\nu_{a}))}\varphi_{a+1}\varphi_ae_{a+1,a+2} e(\nu), \\
(\varphi_{a+1} \varphi_a \tau_{a+1})x_{a+2}e(\nu) & = \varphi_{a+1} \varphi_a
((-1)^{\pa(\nu_{a+1})\pa(\nu_{a+2})}x_{a+1}\tau_{a+1}+e_{a+1,a+2})e(\nu)\\
&= (-1)^{\pa(\nu_{a+2})(\pa(\nu_{a})\pa(\nu_{a+1})+\pa(\nu_{a+2})\pa(\nu_{a})+\pa(\nu_{a+1})\pa(\nu_{a+2}))}
x_a\tau_a\varphi_{a+1}\varphi_a e(\nu) \\
& \qquad \qquad  + \varphi_{a+1}\varphi_a e_{a+1,a+2}e(\nu).
\end{align*}
Hence we have $S x_{b} = \pm x_{s_{a,a+2}b } S$ for all $b$. Using a
similar argument given in \cite[Lemma 4.12]{KK11}
we conclude that $S=0$.

 The equalities in \eqref{Eq: Commute g_a} easily follows from \eqref{Eq: Commute sigma_a}.
\end{proof}

Hence we have
\begin{equation} \label{Eqn: twisting for Q}
\begin{aligned}
&a  g_n \cdots g_1 e(i,\beta) =
g_n \cdots g_1 e(i,\beta) \phi_i^{\pa(\beta)}(\xi_n(a))\quad\text{for any $a\in R(\beta)$ and,}  \\
&x_{n+1}  g_n \cdots g_1 e(i,\beta)  =
(-1)^{\pa(\beta)\pa(i)} g_n \cdots g_1  e(i,\beta)x_1.
\end{aligned}
\end{equation}
 Using a similar method to the construction of $P$, we
obtain the following proposition:

\begin{proposition} \
There is an $(R(\beta+\alpha_i),\k[t_i] \ot R^\Lambda(\beta))$-bilinear homomorphism
$$Q\cl K_0 \to K_1'\seteq R(\beta+\alpha_i)v(i,\beta) \ot_{R(\beta)}\Pi^{\pa(\beta)}_i R^\Lambda(\beta)$$
defined by
$$a e(\beta,i)\ot b \longmapsto a g_n \cdots g_1  v(i,\beta) \ot \pi_i^{\pa(\beta)} b$$
for $a \in R(\beta+\alpha_i)$ and $b \in R^\Lambda(\beta)$. Here,
the right action of $t_i$ on $K_1'$ is given by
$$av(i,\beta)\otimes \pi_i^{\pa(\beta)}b\longmapsto (-1)^{\pa(i)\pa(\beta)}ax_1v(i,\beta)
\otimes\pi_i^{\pa(\beta)}\phi_i(b).$$
\end{proposition}

\begin{theorem} \label{Thm: A nu}
For each $\nu \in I^{\beta}$, set
$$ \mathsf{A}_{\nu}(t_i) = a^{\Lambda}_i(t_i)
\prod_{\substack{1 \le a \le n, \\[.8ex] \nu_a \neq i}}\cQ_{i,\nu_a}(t_i,x_a)
\prod_{\substack{1 \le a \le n, \\[.8ex] \nu_a = i \in \Iod }} (x_a-t_i)^{2}e(\nu),$$
and define
$$\mathsf{A}(t_i)\seteq\sum_{\nu \in I^\beta}\mathsf{A}_{\nu}(t_i) \in \k[t_i] \ot R^\Lambda(\beta).
$$
Then the composition
$$Q \circ P\cl K_1 \to K_1'$$
coincides with the right multiplication by
$(-1)^{p(i)\Lambda_i\pa(\beta)}\mathsf{A}(t_i)$; i.e.,
$$
\begin{aligned}
av(i,\beta)T_i^{\Lambda_i}\otimes \pi_i^{\Lambda_i+\pa(\beta)}b &
\longmapsto \bl av(i,\beta)\otimes
\pi_i^{\pa(\beta)}\phi_i^{\Lambda_i}(b) \br
(-1)^{p(i)\Lambda_i\pa(\beta)}\mathsf{A}(t_i)
\\
&=av(i,\beta)\mathsf{A}(t_i)\otimes \pi_i^{\pa(\beta)}b.
\end{aligned}
$$
\end{theorem}

\begin{proof}
If $i \in \Iev$, our assertion was proved in \cite[Theorem
4.15]{KK11}. If $i \in \Iod$, then it suffices to show that
\begin{equation} \label{Eq: claim1}
\begin{aligned}
&a^{\Lambda}(x_1)\tau_1 \cdots \tau_n g_n \cdots g_1 e(i,\nu)
= a^{\Lambda}(x_1)\tau_1 \cdots \tau_n e(\nu,i) g_n \cdots g_1  \\
& \qquad \qquad \qquad  \equiv \mathsf{A}'_{\nu} \quad {\rm mod} \ R(\beta+\alpha_i)
a^{\Lambda}(x_2) R^1(\beta),
\end{aligned}
\end{equation}
where
$$ \mathsf{A}'_{\nu} = a^{\Lambda}_i(x_1)
\prod_{\substack{1 \le a \le n, \\[.8ex] \nu_a \neq i}}\cQ_{i,\nu_a}(x_1,x_{a+1})
\prod_{\substack{1 \le a \le n, \\[.8ex] \nu_a = i \in \Iod }} (x_{a+1}-x_1)^{2}e(i,\nu).$$

We will prove $\eqref{Eq: claim1}$ by induction on $|\beta|=n$. If $n=0$, the assertion is obvious.
Thus we may assume that $n \ge 1$.

Note that, by \eqref{Eq: reverse varphi_a od}, we have
\begin{equation} \label{Eq: tau n g n}
\begin{aligned}
\tau_n e(\nu,i) g_n=
\left\{
\begin{array}{ll}
   \tau_n e(\nu,i)\tau_n = \cQ_{i,\nu_n}(x_n,x_{n+1})e(\nu_{<n},i,\nu_n) &  \text{ if } \nu_n \neq i,\\
 \tau_n(x_{n+1}^2-x_n^2)(x_{n+1}-x_n) e(\nu,i) &  \text{ if } \nu_n = i.  \\
\end{array}
\right.
\end{aligned}
\end{equation}

(i) We first assume that $\nu_n \neq i$. Then, by \eqref{eqn: center
Q}, we have
\begin{align*}
& a^{\Lambda}(x_1)\tau_1 \cdots \tau_n g_n \cdots g_1 e(i,\nu) \\
& =a^{\Lambda}(x_1)\tau_1 \cdots \tau_{n-1}\cQ_{i,\nu_n}(x_n,x_{n+1})g_{n-1} \cdots g_1 e(i,\nu) \\
& =a^{\Lambda}(x_1)\tau_1 \cdots \tau_{n-1}g_{n-1} \cdots g_1 e(i,\nu) \cQ_{i,\nu_n}(x_1,x_{n+1}) \\
& \equiv \mathsf{A}'_{\nu_{<n}}\cQ_{i,\nu_n}(x_1,x_{n+1}) =
\mathsf{A}'_{\nu} \quad   \mod
R(\beta+\alpha_i)a^{\Lambda}(x_2)R^1(\beta)e(i,\beta).
\end{align*}

(ii) Assume that $\nu_n = i$. Then we have
\begin{equation} \label{Eq ; nu n i}
\begin{aligned}
& a^{\Lambda}(x_1)\tau_1 \cdots \tau_n g_n \cdots g_1 e(i,\nu) \\
&=a^{\Lambda}(x_1)\tau_1 \cdots \tau_n (x_{n+1}-x_n)(x_{n+1}^2-x_n^2)
g_{n-1} \cdots g_1 e(i,\nu).\\
\end{aligned}
\end{equation}
Note that
\begin{equation} \label{eqn: form of ind}
\begin{aligned}
 & a^\Lambda(x_1) \tau_1 \cdots \tau_{n-1} g_n \cdots g_1e(i,\nu) = \pm g_n \cdots g_1 a^\Lambda(x_2) \tau_2 \cdots \tau_n
\equiv 0 \\
& \qquad\qquad\qquad\qquad\qquad\qquad\qquad\qquad
      \ \mod R(\beta+\alpha_i)a^{\Lambda}(x_2)R^1(\beta)e(i,\beta).
\end{aligned}
\end{equation}
By \eqref{Eq: reverse varphi_a od}, the formula \eqref{eqn: form of
ind} can be written as
$$ a^\Lambda(x_1) \tau_1 \cdots \tau_{n-1} (\tau_n (x^2_{n+1}-x^2_n)- (x_{n+1}-x_n))(x^2_{n+1}-x^2_n)
g_{n-1} \cdots g_1e(i,\nu) \equiv 0.$$
Thus
\begin{align*}
& a^\Lambda(x_1) \tau_1 \cdots \tau_{n-1} \tau_n (x^2_{n+1}-x^2_n)^2 g_{n-1} \cdots g_1e(i,\nu)  \\
& \equiv a^\Lambda(x_1) \tau_1 \cdots \tau_{n-1} (x_{n+1}-x_n)(x^2_{n+1}-x^2_n) g_{n-1} \cdots g_1
e(i,\nu) \\
& \equiv (-1)^{\pa(\nu_{<n})}\mathsf{A}'_{\nu_{<n}}(x_{n+1}-x_1)(x^2_{n+1}-x^2_1).
\end{align*}
Since the right multiplication by $(x_{n+1}^2-x_1^2)$ and
$(x_{n+1}-x_1)$ on $K_1$ are injective by Lemma~\ref{Lem: Proj}, we
conclude that
\begin{align*}
& a^\Lambda(x_1) \tau_1 \cdots \tau_{n-1} \tau_n (x^2_{n+1}-x^2_n) g_{n-1} \cdots g_1e(i,\nu) \\
&\qquad\qquad\qquad\qquad\qquad\qquad
\equiv (-1)^{\pa(\nu_{<n})}\mathsf{A}'_{\nu_{<n}}(x_{n+1}-x_1),
\end{align*}
which implies
\begin{align*}
& (-1)^{\pa(\nu_{<n})}
a^\Lambda(x_1) \tau_1 \cdots \tau_{n-1} \tau_n(x^2_{n+1}-x^2_n)(x_{n+1}-x_n)g_{n-1} \cdots g_1e(i,\nu) \\
& \qquad\qquad\qquad\qquad\qquad\qquad
\equiv (-1)^{\pa(\nu_{<n})}\mathsf{A}'_{\nu_{<n}}(x_{n+1}-x_1)^2.
\end{align*}
 Then, $\eqref{Eq ; nu n i}$, together with $\mathsf{A}'_{\nu} =
\mathsf{A}'_{\nu_{<n}}(x_{n+1}-x_1)^2$, implies  the desired result.
\end{proof}

By applying the same argument given in \cite[Lemma 4.19]{KK11}, we
have the following lemma.

\begin{corollary} \label{cor: B nu}
Set $$K_0'\seteq
R(\beta+\al_i)e(\beta,i)T_i^{\Lambda_i}\otimes_{R(\beta)}
\Pi_i^{\Lambda_i}R^\Lambda(\beta).$$
Then the following diagram
commutes
$$\xymatrix{K_1\ar[r]^P\ar[d]_{(-1)^{\pa(i)\Lambda_i\pa(\beta)}\mathsf{A}(t_i)}
&K_0\ar[d]^{(-1)^{\pa(i)\Lambda_i\pa(\beta)}\mathsf{A}(t_i)}\ar[dl]|Q\\
K'_1\ar[r]_{P'}&K_0'.}
$$
Here, $P'\cl K'_1\to K'_0$ is given by
$$av(i,\beta)\otimes \pi_i^{\pa(\beta)}b\longmapsto
a
a^\Lambda(x_1)\tau_1\cdots\tau_ne(\beta,i)T_i^{\Lambda_i}\otimes\pi_i^{\Lambda_i}b,$$
and $(-1)^{\pa(i)\Lambda_i\pa(\beta)}\mathsf{A}(t_i)\cl K_0\to K'_0$
is given by $$a\otimes b\longmapsto\bl
aT_i^{\Lambda_i}\otimes\pi_i^{\Lambda_i}\phi_i^{\Lambda_i}(b)\br
(-1)^{\pa(i)\Lambda_i\pa(\beta)}\mathsf{A}(t_i) =
(-1)^{\pa(i)\Lambda_i\pa(\beta)}a\mathsf{A}(t_i)T_i^{\Lambda_i}\otimes\pi_i^{\Lambda_i}b
.$$ In particular, for any $\nu\in I^\beta$, we have \eq &&
\ba{l}g_n\cdots g_1 a^\Lambda(x_1)\tau_1 \cdots \tau_n e(\nu,i)\otimes e(\beta)  \\[1ex]
\hs{5ex}= (-1)^{\pa(i)\Lambda_i\pa(\beta)}a^{\Lambda}_i(x_{n+1})\hs{-1ex}
\prod\limits_{\substack{1 \le a \le n, \\[.8ex] \nu_a \neq i}} \hs{-1ex}\cQ_{\nu_a,
i}(x_a,x_{n+1})\hs{-2ex} \prod\limits_{\substack{1 \le a \le n, \\[.8ex] \nu_a = i \in
\Iod }}\hs{-1ex}(x_{a}-x_{n+1})^{2}e(\nu,i)\otimes e(\beta) \ea
\eneq
in
$R(\beta+\al_i)e(\beta,i)\otimes_{R(\beta)}R^\Lambda(\beta)$ .
\end{corollary}

Since $K_1$ is a projective $R^{\Lambda}(\beta)\otimes
\k[t_i]$-supermodule by Lemma \ref{Lem: Proj} and $\mathsf{A}(t_i)$ is a
monic (skew)-polynomial (up to a multiple of an invertible element) in
$t_i$, by  a similar argument to the one  in \cite[Lemma 4.17, Lemma 4.18]{KK11},
we conclude:

\begin{theorem} \label{Thm: P-injective}
The module $F^{\Lambda}$ is a projective right
$R^{\Lambda}(\beta)$-supermodule and we have a short exact sequence
consisting of right projective $R^{\Lambda}(\beta)$-supermodules:
\begin{align}\label{Eq: The exact seq}
0 \to K_1  \xrightarrow{\;P\;}
K_0 \to F^{\Lambda} \to 0.
\end{align}
\end{theorem}

Since $K_1$, $K_0$ and $F^{\Lambda}$ are the kernels of functors
$\overline{F}_i$, $F_i$ and $F^{\Lambda}_i$, respectively, we have:

\begin{corollary} \label{Cor: Exa F_i bar F_i}
 For any $i \in I$ and $\beta \in \mathtt{Q}^+$, there exists an exact sequence of
$R(\beta+\alpha_i)$-modules
\begin{align} \label{Eq: Fi bar Fi Fi Lam}
0 \to \Pi_i^{\Lambda_i+\pa(\beta)}q^{(\alpha_i|2\Lambda-\beta)}
\overline{F}_i M \to F_i M \to F^{\Lambda}_i M \to 0,
\end{align}
which is functorial in $M \in \Mod (R^{\Lambda}(\beta))$.
\end{corollary}

For $\alpha \in \mathtt{Q}^+$, let
$\Proj(R^\Lambda(\alpha))$ denote the category of finitely generated
projective $\Z$-graded $R^\Lambda(\alpha)$-modules,
and by $\Rep(R^\Lambda(\alpha))$ the category of
$\Z$-graded $R^\Lambda(\alpha)$-modules coherent over $\k_0$.
Then we conclude that the functors $E^\Lambda_i$ and $F^\Lambda_i$ are well-defined on
$\soplus_{\al\in\mathtt{Q}^+}\Proj (R^\Lambda(\al))$ and $\soplus_{\al\in\mathtt{Q}^+}\Rep (R^\Lambda(\al))$:

\begin{theorem} \label{Thm: Exact}
Set
\begin{align*}
\Proj(R^{\Lambda}) = \bigoplus_{\alpha \in \mathtt{Q}^+}\Proj(R^{\Lambda}(\alpha)),
 \quad \Rep(R^{\Lambda}) = \bigoplus_{\alpha \in \mathtt{Q}^+}\Rep(R^{\Lambda}(\alpha)).
\end{align*}
 Then the functors $E^{\Lambda}_i$ and $F^{\Lambda}_i$ are well-defined exact functors on
$\Proj(R^{\Lambda})$  and $\Rep(R^{\Lambda})$, and they induce endomorphisms
of the Grothendieck groups $[\Proj(R^{\Lambda})]$ and $[\Rep(R^{\Lambda})]$.
\end{theorem}

\begin{proof}
 By  Theorem~\ref{Thm: P-injective},  $F^{\Lambda}$ is a
finitely generated projective module as a right
$R^{\Lambda}(\beta)$-supermodule and as a left
$R^{\Lambda}(\beta+\alpha_i)$-supermodule . Similarly,
$e(\beta,i)R^\Lambda(\beta+\alpha_i)$ is a finitely generated
projective module as a left $R^{\Lambda}(\beta)$-supermodule and as
a right $R^{\Lambda}(\beta+\alpha_i)$-supermodule. Now our
assertions follow from these facts immediately.
\end{proof}

\section{Commutation relations between $E_i^\Lambda$ and $F_i^\Lambda$}

 The main goal of this section is to show that the superfunctors
$E_i^\Lambda$ and $F_i^\Lambda$ satisfy certain commutation
relations, from which we obtain a supercategorification of
$V(\Lambda)$.

\begin{theorem}
For $i \neq j \in I$, there exists a natural isomorphism
\begin{align} \label{Thm: Comm E_i Lam F_j Lam}
 E^{\Lambda}_iF^{\Lambda}_j \simeq q^{-(\alpha_i|\alpha_j)}\Pi^{\pa(i)\pa(j)}F^{\Lambda}_j E^{\Lambda}_i.
\end{align}
\end{theorem}

\begin{proof}
By Proposition \ref{Prop: twist by tau n}, we already know
\begin{align} \label{Eq: commutiation E_iF_j}
e(n,i)R(n+1)e(n,j) \simeq q^{-(\alpha_i|\alpha_j)}R(n)e(n-1,j)\otimes_{R(n-1)}\Pi^{\pa(i)\pa(j)} e(n-1,i)R(n).
\end{align}
Applying the functor $R^{\Lambda}(n)\otimes_{R(n)} \ \bullet \ \otimes_{R(n)}R^{\Lambda}(n)e(\beta)$ on
$\eqref{Eq: commutiation E_iF_j}$, we obtain
\begin{align*}
& \dfrac{e(n,i)R(n+1)e(\beta,j)}{e(n,i)R(n)a^{\Lambda}(x_1) R(n+1)e(\beta,j)+e(n,i)R(n+1)a^{\Lambda}(x_1)R(n)e(\beta,j)}
 \\
&\qquad \simeq q^{-(\alpha_i|\alpha_j)} R^{\Lambda}(n)e(n-1,j)\otimes_{R^{\Lambda}(n-1)}\Pi^{\pa(i)\pa(j)}
e(n-1,i)R^{\Lambda}(n)e(\beta) \\
&\qquad \simeq q^{-(\alpha_i|\alpha_j)}  \Pi^{\pa(i)\pa(j)} F^{\Lambda}_j E^{\Lambda}_iR^{\Lambda}(\beta).
\end{align*}

Note that
\begin{align*}
E^{\Lambda}_iF^{\Lambda}_jR^{\Lambda}(\beta)
=\left( \dfrac{e(n,i)R(n+1)e(n,j)}{e(n,i)R(n+1)a^{\Lambda}(x_1)R(n+1)e(n,j)} \right) e(\beta).
\end{align*}
Thus it suffices to show that
\begin{equation}\label{Eq: claim Comm E_i Lam F_j Lam}
\begin{aligned}
& e(n,i)R(n+1)a^{\Lambda}(x_1)R(n+1)e(n,j) \\
& =e(n,i)R(n)a^{\Lambda}(x_1)R(n+1)e(n,j)+e(n,i)R(n+1)a^{\Lambda}(x_1)R(n)e(n,j).
\end{aligned}
\end{equation}

Since, by \eqref{Eq: cylotomic polynomial}, $a^{\Lambda}(x_1)\tau_{k} = \pm \tau_{k}a^{\Lambda}(x_1)$ for all $k \ge 2$, we have
\begin{align*}
& \ R(n+1)a^{\Lambda}(x_1)R(n+1) = \sum_{a=1}^{n+1}R(n+1)a^{\Lambda}(x_1)\tau_a \cdots \tau_n R(n,1) \\
& = R(n+1)a^{\Lambda}(x_1)R(n,1)+R(n+1) a^{\Lambda}(x_1)
\tau_1 \cdots \tau_n R(n,1) \\
& = R(n+1)a^{\Lambda}(x_1)R(n,1)+\sum_{a=1}^{n+1}R(n,1)\tau_n \cdots \tau_a
a^{\Lambda}(x_1)\tau_1 \cdots \tau_n
    R(n,1) \\
& = R(n+1)a^{\Lambda}(x_1)R(n,1)+R(n,1)a^{\Lambda}(x_1)R(n+1)+  R(n,1)\tau_n \cdots \tau_1 a^{\Lambda}(x_1)\tau_1
\cdots \tau_n R(n,1).
\end{align*}

For $i \neq j$, we get
$$ e(n,i)R(n,1)\tau_n \cdots \tau_1 a^{\Lambda}(x_1)\tau_1 \cdots \tau_n R(n,1) e(n,j) =0,$$
and our assertion $\eqref{Eq: claim Comm E_i Lam F_j Lam}$ follows.
\end{proof}

Let us recall the following
commutative diagram, a super-version of  \cite[(5.5)]{KK11}, derived from Theorem \ref{Thm: Comm E_i F_j},
Theorem \ref{Thm: Comm E_i bar F_j } and Corollary \ref{Cor: Exa F_i bar F_i}:
\begin{equation} \label{Dia: Comm at Mod}
\begin{aligned}
\xymatrix@C=4ex
{ & 0 \ar[d] & 0 \ar[d]&  &  & \\
 0 \ar[r] & \Pi_i^{\Lambda_i+\pa(\beta)}q_i^{(\alpha_i|2\Lambda-\beta)}\overline{F}_iE_i M \ar[d] \ar[r]
 & \Pi_iq_i^{-2}F_iE_i M \ar[d]\ar[r] & \Pi_iq_i^{-2} F^{\Lambda}_iE^{\Lambda}_iM \ar[d]\ar[r] & 0 \\
 0 \ar[r] & \Pi_i^{\Lambda_i+\pa(\beta)}q_i^{(\alpha_i|2\Lambda-\beta)}E_i\overline{F}_i M \ar[d] \ar[r]
 & E_iF_i M \ar[d]\ar[r] & E^{\Lambda}_i F^{\Lambda}_i M \ar[r] & 0 \\
 & \Pi_i^{\Lambda_i}q_i^{(\alpha_i|2\Lambda-2\beta)}\k[t_i]\otimes M \ar[d]\ar[r] & \k[t_i]\otimes M \ar[d] & & \\
 & 0 & 0 & &
}
\end{aligned}
\end{equation}

The kernels of these functors provide the following commutative diagram of 
$(R(\beta),R^\Lambda(\beta))$-superbimodules : 
\begin{equation} \label{Dia: Comm at Ker}
\begin{aligned}
\xymatrix@C=3ex
{ & 0 \ar[d] & 0 \ar[d]&  &  & \\
 0 \ar[r] &
 L'_1 \ar[d] \ar[r]
 & L'_0 \ar[d]^F \ar[r]
 & q_i^{-2} F^{\Lambda}_i \Pi_i E^{\Lambda}_iR^{\Lambda}(\beta) \ar[d]\ar[r] & 0 \\
 0 \ar[r] & L_1
 \ar[d]_B \ar[r]^-{P}
 & L_0 \ar[d]^C\ar[r] & E^{\Lambda}_i F^{\Lambda}_i R^{\Lambda}(\beta) \ar[r] & 0 \\
 &  \k[t_i]T^{\Lambda_i}_i\otimes \Pi^{\Lambda_i}_iR^{\Lambda}(\beta) \ar[d]\ar[r]^-{A}
 & \k[t_i]\otimes R^{\Lambda}(\beta) \ar[d] & & \\
 & 0 & 0 & &
}
\end{aligned}
\end{equation}
where
\begin{align*}
& L'_0 = q^{-2}_i R(\beta)e(\beta-\alpha_i,i)\ot_{R(\beta-\alpha_i)}\Pi_ie(\beta-\alpha_i,i)R^\Lambda(\beta).\\
& L'_1 = q_i^{(\alpha_i|2\Lambda-\beta)}R(\beta)v(i,\beta-\alpha_i)T^{\Lambda_i}_i \ot_{R(\beta-\alpha_i)}\Pi_i^{\Lambda_i+\pa(\beta)}
        e(\beta-\alpha_i,i)R^\Lambda(\beta),\\
& L_0 = e(\beta,i)R(\beta+\alpha_i)e(\beta,i) \ot_{R(\beta)}R^\Lambda(\beta),\\
& L_1 =  q_i^{(\alpha_i|2\Lambda-\beta)}
e(\beta,i)R(\beta+\alpha_i)v(i,\beta)T^{\Lambda_i}_i\ot_{R(\beta)}\Pi_i^{\Lambda_i+\pa(\beta)}R^\Lambda(\beta).
\end{align*}
The homomorphisms in the diagram $\eqref{Dia: Comm at Ker}$ can be
described as follows (cf.\ \cite[\S 5.2]{KK11}):
\begin{itemize}
\item $P$ is given by \eqref{eqn: The map P}. It is
      $(R(\beta+\alpha_i),\k[t_i] \ot R^{\Lambda}(\beta))$-bilinear.
\item $A$ is defined by chasing the diagram. Note that it is $R^{\Lambda}(\beta)$-linear
      but {\it not\/} $\k[t_i]$-linear.
\item $B$ is given by taking the coefficient of $\tau_n \cdots \tau_1$. It is
      ($R(\beta),\k[t_i]\ot R(\beta)$)-linear (see the remark below).
\item $F$ is given by $a \ot \pi_i b \longmapsto a \tau_n \ot b$ for
      $a \in R(\beta)e(\beta-\alpha_i,i)$ and $b \in e(\beta-\alpha_i,i)R^\Lambda(\beta)$
      (See Proposition \ref{Prop: twist by tau n}).
\item $C$ is the cokernel map of $F$. It is $(R(\beta),R^{\Lambda}(\beta))$-bilinear but does {\it not} commute with $t_i$.
\end{itemize}

\begin{remark} \label{Rmk: the map B}
The map $B$ can be described as
\begin{align*}
&B\bl x_{n+1}^l a \tau_{n} \cdots \tau_k v(i,\beta)T^{\Lambda_i}_i \ot \pi_i^{\Lambda_i+\pa(\beta)}b\br
 = \delta_{k,1}t_i^l T^{\Lambda_i}_i \ot \pi_i^{\Lambda_i} \phi^{\Lambda_i}_i(a) b
\end{align*}
for  $ a \in R(\beta) $ and $ b \in R^\Lambda(\beta)$. Then
\begin{align*}
& B\left( (x_{n+1}^l a \tau_{n} \cdots \tau_k v(i,\beta)T^{\Lambda_i}_i \ot \pi^{\Lambda_i+\pa(\beta)}_ib) t_i
\right ) \\
& = B\left(\delta_{k,1}(-1)^{\pa(i)(\Lambda_i+\pa(\beta))}
(x_{n+1}^l a \tau_{n} \cdots \tau_k v(i,\beta)T^{\Lambda_i}_i t_i \ot \pi^{\Lambda_i+\pa(\beta)}_i \phi_i(b)
 \right) \\
& = B\left(\delta_{k,1}(x_{n+1}^l a x_{n+1}\tau_{n} \cdots \tau_k v(i,\beta)T^{\Lambda_i}_i \ot
\pi^{\Lambda_i+\pa(\beta)}_i \phi_i(b)\right) \\
& = B\left(\delta_{k,1}(x_{n+1}^{l+1} \phi_i(a) \tau_{n} \cdots \tau_k v(i,\beta)T^{\Lambda_i}_i \ot
\pi^{\Lambda_i+\pa(\beta)}_i \phi_i(b) \right) \\
& = \delta_{k,1} t^{l+1}_i T^{\Lambda_i}_i \ot \pi^{\Lambda_i}_i \phi^{\Lambda_i+1}_i(a)  \phi_i(b).
\end{align*}
On the other hand,
\begin{align*}
&B\left( (x_{n+1}^l a \tau_{n} \cdots \tau_k v(i,\beta)T^{\Lambda_i}_i \ot \pi^{\Lambda_i+\pa(\beta)}_ib)
\right ) t_i = \delta_{k,1} (t^{l}_iT^{\Lambda_i}_i \ot \pi^{\Lambda_i}_i \phi^{\Lambda_i}_i(a)  b) t_i\\
& = \delta_{k,1} t^{l+1}_i T^{\Lambda_i}_i \ot \pi^{\Lambda_i}_i\phi^{\Lambda_i+1}_i(a)  \phi_i(b). \\
\end{align*}
Thus $B$ is right $(\k[t_i]\otimes R(\beta))$-linear.
\end{remark}

\medskip
Define
\begin{align*}
\mathbf{T}=T_i^{\Lambda_i} \ot \pi^{\Lambda_i}_i1
\in \k[t_i]T^{\Lambda_i}_i \otimes \Pi^{\Lambda_i}_i R^{\Lambda}(\beta), \quad
\mathbf{T}_1=v(i,\beta)T_i^{\Lambda_i} \ot \pi^{\Lambda_i+\pa(\beta)}_i1
\in L_1.
\end{align*}
The element $\mathbf{T}$ has $\Z_2$-degree $\pa(i)\Lambda_i$ and
$\mathbf{T}_1$ has $\Z_2$-degree $\pa(i)(\Lambda_i+\pa(\beta))$.
Note that
$$\mathbf{T}t_i=t_i\mathbf{T} \quad \text{ and } \quad \mathbf{T}_1t_i=(-1)^{\pa(i)\pa(\beta)}t_i\mathbf{T}_1.$$

Let $p$ be the number of $\alpha_i$ appearing in $\beta$. Define an
invertible element $\gamma \in \k^{\times}$ by
\begin{equation} \label{Eq: def gamma}
\begin{aligned}
& (-1)^{\pa(i)\Lambda_i\pa(\beta)+p}\prod_{\substack{1 \le a \le n, \\[.8ex] \nu_a \neq i}}
\cQ_{i,\nu_a}(t_i,x_{a})
\prod_{\substack{1 \le a \le n, \\[.8ex] \nu_a = i \in \Iod}}(x_{a}-t_i)^2\\
& = \gamma^{-1}t^{-\langle h_i,\beta \rangle+2(1+\pa(i))p}_i +
   \bl\text{terms of degree $ < -\langle h_i,\beta \rangle+2(1+\pa(i))p $ in $t_i$} \br.
\end{aligned}
\end{equation}
Note that $\gamma$ does not depend on $\nu \in I^\beta$.

Set $\lambda=\Lambda - \beta$ and
\begin{align} \label{Eq: varphi k}
\varphi_k = A(\mathbf{T}t^k_i) \in \k[t_i] \otimes R^{\Lambda}(\beta).
\end{align}

{}From now on, we investigate the kernel and cokernel of the map $A$
which are the key ingredients of the proof of Theorem~\ref{Thm: Main} below.
For this purpose, the following proposition is
crucial.

\begin{proposition} \label{Prop: varphi k}
The element $\gamma\varphi_k$ is a monic \ro skew\rf-polynomial in $t_i$ of degree
$\langle h_i,\lambda\rangle+k$.
\end{proposition}
Here and in the sequel, for $m<0$, we say that a (skew)-polynomial $\varphi$ is a monic polynomial of degree $m$
if $\varphi=0$.

To prove Proposition \ref{Prop: varphi k}, we need some preparation.
Define a map $E\cl L'_0 \to R^\Lambda(\beta)$ by
\begin{align} \label{Eq: Def F}
a \ot \pi_i b \mapsto a\phi_i(b) \quad
\text{for $a \in R(\beta)e(\beta-\alpha_i,i)$ and $b \in e(\beta-\alpha_i,i)R^{\Lambda}(\beta)$.}
\end{align}

We define the endomorphism \  $\ \circ (x_n\ot 1)$ of $L'_0$ by
$$(a \ot \pi_i b)(x_n \ot 1) = (-1)^{\pa(i)} a x_n \ot \pi_i \phi_i(b).$$

\begin{lemma}
Let $$ L'_0\seteq R(\beta)e(\beta-\alpha_i,i)
\otimes_{R(\beta-\alpha_i)}\Pi_i
e(\beta-\alpha_i,i)R^{\Lambda}(\beta).$$ Then for any $z \in L'_0$,
we have
\begin{align}
F(z)t_i = F(z(x_n \otimes 1)) + e(\beta,i) \ot E(z).
\end{align}
\end{lemma}

\begin{proof} We may assume $z = a \ot \pi_i b$.
Note that
$$ F(z) = a \tau_n e(\beta-\alpha_i,i^2) \ot b, \qquad E(z)= a \phi_i(b).$$
Thus
\begin{align*}
F(z)t_{i} & = a \tau_n e(\beta-\alpha_i,i^2)x_{n+1} \ot \phi_i(b) \\
& =a((-1)^{\pa(i)}x_n \tau_n +1)e(\beta-\alpha_i,i^2) \ot \phi_i(b)\\
& = (-1)^{\pa(i)}a x_n \tau_n e(\beta-\alpha_i,i^2) \ot \phi_i(b) +  a e(\beta-\alpha_i,i^2)\ot \phi_i(b) \\
& = (-1)^{\pa(i)}F(ax_n \otimes \pi_i \phi_i(b)) + e(\beta,i) \ot E(z)\\
&=F(z(x_n \otimes 1))+e(\beta,i) \ot E(z).
\end{align*}
\end{proof}

By Proposition \ref{Prop: twist by tau n}, we have
\begin{equation} \label{Eq: Decom wrt F}
\begin{aligned}
& e(\beta,i)R(\beta+\alpha_i)e(\beta,i) \otimes_{R(\beta)} R^{\Lambda}(\beta) \\
& \ =
F\bl(R(\beta)e(\beta-\alpha_i,i)\otimes_{R(\beta-\alpha_i)}e(\beta-\alpha_i,i)R^{\Lambda}(\beta)\br)
\\
& \qquad \oplus e(\beta,i)(\k[t_i] \ot R^{\Lambda}(\beta)).
\end{aligned}
\end{equation}
Using the decomposition $\eqref{Eq: Decom wrt F}$, we may write
\begin{align} \label{Eq: Decom Im P wrt F }
P(e(\beta,i)\tau_n \cdots \tau_1\mathbf{T}_1 t^k_i )
= F(\psi_k) + e(\beta,i)\varphi_k
\end{align}
for uniquely determined $\psi_k \in L'_0$ and $\varphi_k \in \k[t_i]
\otimes R^{\Lambda}(\beta)$. On the other hand, 
we have
\begin{align*}
A( \mathbf{T}t^k_i)
&= AB(e(\beta,i)\tau_n\cdots\tau_1 \mathbf{T}_1 t^k_i) \\
        &= CP(e(\beta,i)\tau_n\cdots\tau_1 \mathbf{T}_1 t^k_i)
        = \varphi_k.
\end{align*}
Hence the definition of $\varphi_k$ coincides with the definition
given in $\eqref{Eq: varphi k}$. Note that
\begin{align*}
F(\psi_{k+1})+e(\beta,i)\varphi_{k+1}
&= P (e(\beta,i)\tau_n \cdots \tau_1 \mathbf{T}_1 t_i^{k+1} )
\\ &= P(e(\beta,i)\tau_n \cdots \tau_1 \mathbf{T}_1 t_i^{k} )t_i \\
&= (F(\psi_{k})+e(\beta,i)\varphi_{k})t_i \\
&= F(\psi_k(x_n \otimes 1))+e(\beta,i)E(\psi_k)+e(\beta,i)\varphi_k t_i\,,
\end{align*}
which yields
\begin{equation} \label{Eq: formulas psi varphi}
\begin{aligned}
\psi_{k+1}=\psi_k(x_n \otimes 1), \quad
\varphi_{k+1} =E(\psi_k)+\varphi_k t_i.
\end{aligned}
\end{equation}

Now we will prove Proposition \ref{Prop: varphi k}. By Corollary \ref{cor: B nu}, the equality
\begin{align*}
& g_n\cdots g_1 x^k_1e(i,\nu)a^\Lambda(x_1)\tau_1 \cdots \tau_n  \\
&=(-1)^{(k+\Lambda_i)\pa(i)\pa(\beta)}x^{k}_{n+1} a^{\Lambda}_i(x_{n+1})
\prod_{\substack{1 \le a \le n, \\ \nu_a \neq i}} \cQ_{\nu_a,
i}(x_a,x_{n+1}) \prod_{\substack{1 \le a \le n, \\ \nu_a = i \in
\Iod }} (x_{a}-x_{n+1})^{2}e(\nu,i)
\end{align*}
holds in $R(\beta+\alpha_i)e(\beta,i) \otimes_{R(\beta)} R^{\Lambda}(\beta)$, which implies
\begin{align*}
&AB(g_n\cdots g_1 x^k_1 e(i,\nu)\mathbf{T}_1) \\
&= C\Bigl( (-1)^{(k+\Lambda_i)\pa(i)\pa(\beta)}x^{k}_{n+1}
a^{\Lambda}_i(x_{n+1}) \prod_{\substack{1 \le a \le n, \\ \nu_a \neq
i}} \cQ_{\nu_a, i}(x_a,x_{n+1}) \prod_{\substack{1 \le a \le n, \\
\nu_a = i \in \Iod }}(x_{a}-x_{n+1})^{2}e(\nu,i) \Bigr)
\\
&= (-1)^{(k+\Lambda_i)\pa(i)\pa(\beta)}t^{k}_i a^{\Lambda}_i(t_i)
\prod_{\substack{1 \le a \le n, \\ \nu_a \neq i}}
\cQ_{\nu_a,i}(x_a,t_i) \prod_{\substack{1 \le a \le n, \\ \nu_a = i
\in \Iod }}(x_{a}-t_i)^{2}e(\nu).
\end{align*}

On the other hand, since $B$ is the map taking the coefficient of
$\tau_n \cdots \tau_1$, we have
\begin{align*}
& B(g_n\cdots g_1 x^k_1e(i,\nu)\mathbf{T}_1) \\
&=
B\left((-1)^{k\pa(i)\pa(\beta)}x^k_{n+1}\prod_{\nu_a =i}-(x^{1+\pa(i)}_{n+1}-x^{1+\pa(i)}_a)^2
e(\nu,i)\tau_n\cdots\tau_1\mathbf{T}_1\right)\\
&=(-1)^{k\pa(i)\pa(\beta)+p}\; t^k_i  \prod_{\nu_a =i} (t^{1+\pa(i)}_{i}-x^{1+\pa(i)}_a)^{2}\mathbf{T}e(\nu).
\end{align*}
Thus we have
\begin{equation} \label{Eq: the image A}
\begin{aligned}
&  A(t^k_i \prod_{\nu_a =i}(t^{1+\pa(i)}_i-x^{1+\pa(i)}_a)^2\mathbf{T}e(\nu) )  \\
& = (-1)^{\Lambda_i\pa(i)\pa(\beta)+p}\; t^{k}_i a^{\Lambda}_i(t_i) \prod_{\substack{1 \le a \le n, \\[.6ex]
\nu_a \neq i}} \cQ_{\nu_a,i}(x_a,t_i)
\prod_{\substack{1 \le a \le n, \\[.6ex] \nu_a = i \in \Iod }}  (x_{a}-t_i)^{2}e(\nu) \\
\end{aligned}
\end{equation}

Set
\begin{align*}
& \mathsf{S}_i =
\sum_{\nu \in I^{\beta}}
\prod_{\nu_a =i}(t^{1+\pa(i)}_i-x^{1+\pa(i)}_a)^{2}e(\nu) \in \k[t_i]\otimes  R^{\Lambda}(\beta),
\\
& \mathsf{F}_i = \gamma (-1)^{\Lambda_i\pa(i)\pa(\beta)+p} a^{\Lambda}_i(t_i) \sum_{\nu \in I^{\beta}}
\Bigl( \prod_{\substack{1 \le a \le n, \\ \nu_a \neq
i}}\cQ_{i,\nu_a}(t_i,x_{a}) \prod_{\substack{1 \le a \le n, \\ \nu_a
= i \in \Iod}}(x_a-t_i)^2 e(\nu) \Bigr) \\
&\hs{60ex}\in \k[t_i] \otimes
R^{\Lambda}(\beta).
\end{align*}
Then they are monic (skew)-polynomials in $t_i$ of degree
$2(1+\pa(i))p$ and $\langle h_i, \lambda \rangle+2(1+\pa(i))p$,
respectively. Note that $\mathsf{S}_i$ is contained in the center of
$\k[t_i] \otimes R^{\Lambda}(\beta)$ and $F_i$ commutes with $t_i$.
Hence $\eqref{Eq: the image A}$ can be expressed in the following
form:
\begin{align} \label{Eq: Rel F and S}
 \gamma A(t_i^{k}\mathsf{S}_i\mathbf{T})=t_i^{k}\mathsf{F}_i.
\end{align}

\begin{lemma} \label{Lem: Rel between F and S}
For any $k \ge 0$, we have
\begin{align} \label{Eq: remaining term}
 t^k_i \mathsf{F}_i= (\gamma \varphi_k)\mathsf{S}_i+\mathsf{h}_k,
\end{align}
where $\mathsf{h}_k\in \k[t_i] \ot R^\Lambda(\beta)$
is a polynomial in $t_i$  of degree $<2(1+\pa(i))p$.
In particular, $\gamma \varphi_k$ coincides with the
quotient of $t^k_i \mathsf{F}_i$ by $\mathsf{S}_i$.
\end{lemma}

\begin{proof}
By $\eqref{Eq: formulas psi varphi}$,
\begin{align} \label{Eq: Degree 1 }
A(a t_i)-A(a)t_i \in \k[t_i]\ot R^{\Lambda}(\beta) \text{ is of degree } \le 0 \text{ in }t_i,
\end{align}
for any $a \in \k[t_i]T^{\Lambda}_i \ot \Pi^{\Lambda_i}_i R^{\Lambda}(\beta)$.
We will show
\begin{equation}  \label{Eq: Degree 2 }
\begin{aligned}
& \text{for any polynomial $f$ in the center of $\k[t_i] \ot R^{\Lambda}(\beta)$ in
$t_i$ of degree $m \in \Z_{\ge 0}$} \\ &
\text{ and $a \in \k[t_i]T^{\Lambda_i}_i \ot R^{\Lambda}(\beta)$, $A(a f)- A(a)f$ is of degree $<m$.}
\end{aligned}
\end{equation}
We will use induction on $m$.  Since $A$ is right
$R^{\Lambda}(\beta)$-linear, 
\eqref{Eq: Degree 2 } holds for $m=0$.  Thus it suffices to show
\eqref{Eq: Degree 2 } when $f=t_ig$.  By the induction hypothesis,  \eqref{Eq:
Degree 2 } is true for $g$.  Then we have
\begin{align*}
A(af )-A(a)f
= (A(a t_i g) - A(a t_i)g)
+ ( A(a t_i) - A(a)t_i)g.
\end{align*}
It follows that the first term is of degree $<\deg(g)$ in $t_i$ and
the second term is of degree $<\deg(g)+1$, which proves $\eqref{Eq:
Degree 2 }$. Thus we have
$$ t^k_i \gamma^{-1}\mathsf{F}_i - \varphi_k \mathsf{S}_i = t^k_i \gamma^{-1}\mathsf{F}_i -A(t^k_i\mathbf{T})\mathsf{S}_i
= A(t^k_i\mathsf{S}_i\mathbf{T})-A(t^k_i\mathbf{T})\mathsf{S}_i$$
by $\eqref{Eq: Rel F and S}$ and it is of degree $<2(1+\pa(i))p$
by applying $\eqref{Eq: Degree 2 }$
for $f=\mathsf{S}_i$.
\end{proof}

Therefore, by Lemma \ref{Lem: Rel between F and S}, we conclude
$\gamma \varphi_k$ is a monic (skew)-polynomial in $t_i$ of degree
$\langle h_i, \lambda \rangle +k$, which completes the proof of
Proposition \ref{Prop: varphi k}.

\begin{theorem} \label{Thm: Main}
Let $\lambda=\Lambda-\beta$. Then there exist natural isomorphisms
of endofunctors on $\Mod(R^{\Lambda}(\beta))$ given below. \bnum
\item If $\langle h_i,\lambda \rangle \ge 0$, then we have
\begin{align} \label{Eq: Comm E_i Lam F_i Lam 1}
\Pi_iq_i^{-2}F^{\Lambda}_iE^{\Lambda}_i \oplus
\bigoplus^{\langle h_i,\lambda \rangle-1}_{k=0}\Pi_i^k q_i^{2k} \overset{\sim}{\to}
E^{\Lambda}_iF^{\Lambda}_i.
\end{align}
\item If $\langle h_i,\lambda \rangle < 0$, then we have
\begin{align} \label{Eq: Comm E_i Lam F_i Lam 2}
\Pi_iq_i^{-2}F^{\Lambda}_iE^{\Lambda}_i \overset{\sim}{\to}
E^{\Lambda}_iF^{\Lambda}_i \oplus \bigoplus^{-\langle h_i,\lambda \rangle-1}_{k=0}\Pi_i^{k+1}q_i^{-2k-2 }.
\end{align}
\end{enumerate}
\end{theorem}

\begin{proof}
Due to Proposition \ref{Prop: varphi k} and  \eqref{Eq: formulas psi
varphi}, we can apply the arguments in \cite[Theorem 5.2]{KK11} with
a slight modification. Hence we will give only a sketch of proof.

{}From the Snake Lemma, we get an exact sequences of
$R^{\Lambda}(\beta)$-superbimodules:
$$0 \to {\rm Ker } A \to q_i^{-2}F^\Lambda_i \Pi_i E^\Lambda_iR^{\Lambda}(\beta) \to
 E^\Lambda_iF^\Lambda_iR^{\Lambda}(\beta) \to {\rm Coker } A \to 0.$$
 If $a\seteq\langle h_i,\lambda \rangle \ge 0$, then Proposition \ref{Prop: varphi k} yields
 $${\rm Ker A} =0, \quad \bigoplus_{k=0}^{a-1} \k t_i^k \ot R^{\Lambda}(\beta) \simeq
{\rm Coker} A $$ and our first assertion follows.

If $a\seteq\langle h_i,\lambda \rangle < 0$, then Proposition
\ref{Prop: varphi k} implies ${\rm Coker A} =0$. By \eqref{Eq:
formulas psi varphi}, we can prove that there is an isomorphism
$${\rm Ker } A \simeq \bigoplus_{k=0}^{-a-1} \k t_i^k \ot R^{\Lambda}(\beta),$$
which completes the proof. 
\end{proof}
`

\section{Supercategorification}

In this section, applying the results obtained in the previous sections,
we will show that the quiver Hecke superalgebra $R(\beta)$ and its cyclotomic quotient $R^\Lambda(\beta)$
($\beta \in \mathtt{Q}^+$) give supercategorifications of $U^-_\A(\g)$ and $V_\A(\Lambda)$, respectively.

{}From now on, we assume \eqref{cond:k0}; i.e.,  $\k_0$ is a field
and the $\k_i$'s are finite-dimensional over $\k_0$. Set
$$\Proj(R^\Lambda)= \bigoplus_{\beta \in \mathtt{Q}^+}\Proj(R^\Lambda(\beta)) \quad \text{ and }
\quad \Rep(R^\Lambda)= \bigoplus_{\beta \in \mathtt{Q}^+}\Rep(R^\Lambda(\beta)). $$

Recall the anti-involution $\psi\cl R^\Lambda(\beta) \to R^\Lambda(\beta)$  given
by \eqref{def:psi}.
For $N \in \Mod(R^\Lambda(\beta))$, let $N^{\psi}$ be the right $R^\Lambda(\beta)$-module
obtained from $N$ by the anti-involution $\psi$ of $R^{\Lambda}(\beta)$.
By \eqref{property of R(beta)-mod} and
Theorem \ref{Thm: Pi-invariant}, we have the pairing
$$ [\Proj(R^\Lambda)] \times [\Rep(R^\Lambda)] \to \A$$
given by
\begin{equation} \label{eqn: A-dual}
([P],[M]) \mapsto  \sum_{n \in \Z} q^n\dim_{\k_0}(P^\psi \ot_{R^{\Lambda}} M)_n.
\end{equation}
 Lemma~\ref{lem:abs_irr} implies
\begin{lemma}\label{lem:duality}
The Grothendieck groups
$[\Proj(R^\Lambda)]$ and $[\Rep(R^\Lambda)]$ are $\A$-dual to each
other by this pairing.
\end{lemma}
{}From Theorem \ref{Thm: Exact}, we can define endomorphisms
$\mathsf{E}_i$ and $\mathsf{F}_i$, induced by $E^\Lambda_i$ and
$F_i^\Lambda$, on the Grothendieck groups $[\Proj(R^\Lambda)]$ and
$[\Rep(R^\Lambda)]$ as follows: \eqn &&
\xymatrix@C=13ex{[\Proj(R^\Lambda(\beta))]\ar@<.8ex>[r]^-{\mathsf{F}_i\seteq[F_{i}^{\Lambda}]}
&[\Proj(R^\Lambda(\beta+\alpha_i))]
\ar@<.8ex>[l]^-{\mathsf{E}_i\seteq[q_i^{1-\langle
h_i,\Lambda-\beta\rangle} E_{i}^{\Lambda}]}
},\\
&&\xymatrix@C=13ex{[\Rep(R^\Lambda(\beta))]
\ar@<.8ex>[r]^-{\mathsf{F}_i\seteq[q_i^{1-\langle
h_i,\Lambda-\beta\rangle} F_{i}^{\Lambda}]}
&[\Rep(R^\Lambda(\beta+\alpha_i))]\ar@<.8ex>[l]^-{\mathsf{E}_i\seteq[E_{i}^{\Lambda}]}.
} \eneqn

Then, from the isomorphisms $\eqref{Thm: Comm E_i Lam F_j Lam}$,
$\eqref{Eq: Comm E_i Lam F_i Lam 1}$, $\eqref{Eq: Comm E_i Lam F_i
Lam 2}$ and Theorem \ref{Thm: Pi-invariant}, we obtain the following
identities in $[\Proj(R^\Lambda(\beta))]$ and $[\Rep(R^\Lambda)(\beta)]$:

\begin{equation} \label{Eq: The com rel}
\begin{aligned}
& \mathsf{E}_i \mathsf{F}_j = \mathsf{F}_j \mathsf{E}_i \quad \text{ if } i \neq j, \\
& \mathsf{E}_i \mathsf{F}_i = \mathsf{F}_i \mathsf{E}_i  +
\dfrac{q_i^{\langle h_i,\Lambda-\beta \rangle}-q_i^{-\langle h_i,\Lambda-\beta \rangle}}{q_i-q^{-1}_i}
\quad\text{if $\langle h_i,\Lambda-\beta \rangle\ge0$,}\\
&\mathsf{E}_i \mathsf{F}_i +
\dfrac{q_i^{-\langle h_i,\Lambda-\beta \rangle}-q_i^{\langle
h_i,\Lambda-\beta \rangle}}{q_i-q^{-1}_i} =
\mathsf{F}_i \mathsf{E}_i \quad\text{ if $\langle h_i,\Lambda-\beta \rangle\le0$}.
\end{aligned}
\end{equation}

Let $\mathsf{K}_i$ be an endomorphism on $[\Proj(R^\Lambda(\beta))]$ and $[\Rep(R^\Lambda(\beta))]$ given by
$$ \mathsf{K}_i|_{[\Proj(R^{\Lambda}(\beta))]} \seteq q_i^{\langle h_i,\Lambda-\beta \rangle}, \quad
\mathsf{K}_i|_{[\Rep(R^{\Lambda}(\beta))]}\seteq q_i^{\langle h_i,\Lambda-\beta \rangle}.$$
Then \eqref{Eq: The com rel} can be rewritten as the commutation relation (Q3)
in Definition~\ref{Def: KM}:
\begin{equation} \label{eqn: Comm e_i f_i}
 [\mathsf{E}_i,\mathsf{F}_j] =\delta_{i,j}\dfrac{\mathsf{K}_i-\mathsf{K}^{-1}_i}{q_i-q^{-1}_i}.
\end{equation}

Define the  superfunctors ${E^\Lambda_i}^{(n)}$ and
${F^\Lambda_i}^{(n)}$ , \eqn &&
\xymatrix@C=13ex{\Mod(R^\Lambda(\beta))\ar@<.8ex>[r]^-{{F^\Lambda_i}^{(n)}}
&\Mod(R^\Lambda(\beta+n\alpha_i))
\ar@<.8ex>[l]^-{{E^\Lambda_i}^{(n)}} } \eneqn  by
\begin{align*}
& {E^\Lambda_i}^{(n)}\cl N\longmapsto (R^\Lambda(\beta)\ot P(i^{n}))^\psi
\ot_{R^\Lambda(\beta)\ot R(n\alpha_i)}e(\beta,i^n)N, \\
& {F^\Lambda_i}^{(n)}\cl M \longmapsto R^\Lambda(\beta+n\alpha_i) e(\beta,i^n)
\ot_{R^\Lambda(\beta)\ot R(n\alpha_i)}(M \bt P(i^{n}))
\end{align*}
for $M \in \Mod(R^\Lambda(\beta))$ and $N \in \Mod(R^\Lambda(\beta+n\alpha_i))$.
Then they induce the endomorphism $\mathsf{E}^n_i/[n]_i!$ and $\mathsf{F}^n_i/[n]_i!$
by \eqref{eq:divided}.

Note that
\eq
&&
\parbox{70ex}{
\be[(i)]
\item the action of $\mathsf{E}_i$ on $[\Proj(R^\Lambda)]$ and $[\Rep(R^\Lambda)]$ is locally nilpotent,
\item  if the module $[M]$ in $[\Rep(R^\Lambda(\beta))]$ satisfies
$\mathsf{E}_i[M]=0$
for all $i \in I$, then $\beta=0$.\label{prop:ht}
\ee
}\label{prop:gro}
\eneq
By Corollary \ref{Cor: integrability}, we see that
\begin{center}
the action $\mathsf{F}_i$ on $[\Proj(R^\Lambda)]$ and $[\Rep(R^\Lambda)]$ are locally nilpotent.
\end{center}
Therefore, by \eqref{eqn: Comm e_i f_i} and
\cite[Proposition B.1]{KMPY96}, the endomorphisms $\mathsf{F}_i$ and $\mathsf{E}_i$
satisfy the quantum Serre relation (Q4) in Definition \ref{Def: KM}.
Hence $[\Proj R^\Lambda]$ and $[\Rep R^\Lambda]$ are endowed with a $U_\A(\g)$-module structure.

Note that $[\Proj(R)]\seteq \bigoplus_{\beta \in
\mathtt{Q}^+}[\Proj(R(\beta))]$ and $[\Rep(R)]\seteq
\bigoplus_{\beta \in \mathtt{Q}^+}[\Rep(R(\beta))]$ are also
$\A$-dual to each other. The exact functors $E_i\cl
\Rep(R(\beta+\alpha_i)) \to \Rep(R(\beta))$ and
 $F_i'\cl \Rep(R(\beta)) \to \Rep(R(\beta+\alpha_i))$ defined in \eqref{eqn: ceystal operators}
 induce endomorphisms $\mathsf{E}_i'$ and
$\mathsf{F}_i'$ on $[\Rep(R)]$, respectively.
Hence, \eqref{eqn: Bqg structure} implies the following commutation relation in $[\Rep(R)]$:
\begin{align} \label{eqn: Bqg identity}
\mathsf{E}_i'\mathsf{F}_j' = q^{-(\alpha_i|\alpha_j)}\mathsf{F}_j'\mathsf{E}_i'+ \delta_{i,j}.
\end{align}

Similarly, we define \eqn &&
\xymatrix@C=13ex{\Proj(R(\beta))\ar@<.8ex>[r]^-{F_i}
&\Proj(R(\beta+\alpha_i)) \ar@<.8ex>[l]^-{E'_i}
},\\
\eneqn
by
\begin{align*}
 F_iP \seteq R(\beta+\alpha_i)e(\beta,i) \ot_{R(\beta)} P \quad \text{ and } \quad
 E_i'Q \seteq \dfrac{e(\beta,i)R(\beta+\alpha_i)}{e(\beta,i)x_{n+1}R(\beta+\alpha_i)} \ot_{R(\beta+\alpha_i)} Q,
\end{align*}
where $|\beta|=n$. Then they are well-defined on $\Proj(R)$, and we
obtain an exact sequence \eqn &&\xymatrix{ 0\ar[r]
&\delta_{i,j}\;\id\ar[r]& E_i'F_j\ar[r]&
\Pi^{\pa(i)\pa(j)}q^{-(\alpha_i|\alpha_j)}F_jE_i'\ar[r]&0. } \eneqn
Thus the exact functors induce the endomorphisms $\mathsf{E}_i'$ and
$\mathsf{F}'_i$ on $[\Proj(R)]$ and satisfy the same equation in
\eqref{eqn: Bqg identity}. (See \cite[Lemma 5.1]{KOP11}, for more
details.)

 Let $\Ir_0(R^\Lambda(\beta))$ be the set of isomorphism classes of
self-dual irreducible $R^\Lambda(\beta)$-modules,
and $\Ir_0(R^\Lambda)\seteq\bigsqcup_{\beta\in\mathtt{Q}^+}\Ir_0(R^\Lambda(\beta))$.
Then $\set{[S]}{S\in\Ir_0(R^\Lambda)}$ is a strong perfect basis
of $[\Rep(R^\Lambda)]$ by Theorem \ref{Thm: choice of strong perfect basis}.
By Proposition \ref{Prop: dual-integral form},
(\ref{prop:gro} \ref{prop:ht}) and Lemma~\ref{lem:duality}, we  conclude:
\begin{theorem}\label{th:main1}
 For $\Lambda \in \mathtt{P}^+$, we have
\begin{align}
V_\A(\Lambda)^\vee \simeq [\Rep(R^\Lambda)] \quad \text{ and } \quad
V_\A(\Lambda) \simeq [\Proj(R^\Lambda)]
\end{align}
as $U_{\A}(\g)$-modules.
\end{theorem}
The fully faithful exact functor
$\Rep(R^\Lambda(\beta))\to\Rep(R(\beta))$ induces an
$\A$-linear homomorphism $[\Rep(R^\Lambda)]\to
{[\Rep(R)]}$. Hence  $[\Rep(R^\Lambda)]\to {[\Rep(R)]}$ is
injective and its cokernel is a free $\A$-module. By the
duality, the homomorphism $[\Proj(R)]\to{[\Proj(R^\Lambda)]}$ is surjective.

Denote by $B_\A^{\mathrm{low}}(\g)$ (resp.\
$B_\A^{\mathrm{up}}(\g)$) the $\A$-subalgebra of $B_q(\g)$ generated
by $e_i'$ and $f_i^{(n)}$ (resp.\ by ${e'_i}^n/[n]_i!$ and $f_i$)
for all $i \in I$ and $n \in \Z_{>0}$.

As a $U^-_\A(\g)$-module, $U^-_\A(\g)$ is  the  projective limit of
$V_\A(\Lambda)$.
Hence,  Theorem~\ref{th:main1} implies the following corollary:

\begin{corollary}\label{cor:main2}
There exist isomorphisms:
$$
\begin{aligned}
& U^-_\A(\g)^\vee \simeq [\Rep(R)] \quad \text{as a
$B_\A^{\mathrm{up}}(\g)$-module}, \\
& U^-_\A(\g) \simeq [\Proj(R)] \quad \ \text{as a
$B_\A^{\mathrm{low}}(\g)$-module}.
\end{aligned}
$$
\end{corollary}

\bibliographystyle{amsplain}

%%%%%%%%%%%%%%%%%%%%%%%%%%%%%%%%%%%%%%%%%%%%%%%%%%%%%%%%%%%%%%%%%%%
\end{document}